\documentclass[]{amsart} \usepackage{amsfonts}

\usepackage{bbding}
\usepackage[dvips]{graphicx}
\usepackage{amsfonts,amssymb,amsmath,amsthm,amscd}

\usepackage{mathrsfs}
\usepackage{xcolor}
\usepackage{empheq}
\usepackage{adforn}
\usepackage{cancel}
\usepackage{xr}
\usepackage{mdframed}
\usepackage{tikz-cd}
\usepackage{amsfonts}
\usepackage{mathdots}
\usepackage{pinlabel}

\usepackage{enumerate}
\usepackage{lmodern}
\usepackage{mdframed}
\usepackage{stmaryrd}
\usepackage{blindtext}
\usepackage{leftidx}
\usepackage{easytable}

\usepackage{fouriernc}
\usepackage[scaled=0.875]{helvet}
\usepackage{courier}
\usepackage[marginparwidth=2.5cm]{geometry}
\definecolor{navie}{RGB}{0, 70, 140}
\usepackage[bookmarks=true,colorlinks=true,urlcolor=red!60!black,linkcolor=red!60!black,citecolor=navie,pagebackref=true,pdfstartview=FitV]{hyperref}

\usepackage{enumitem}

\makeatletter
\newcommand*{\enuma}[1]{%
	\expandafter\@enuma\csname c@#1\endcsname%
}
\newcommand*{\@enuma}[1]{%
	$\ifcase#1\or(a)\or(b_-)\or(b_+)\or(c)%
	\else\@ctrerr\fi$%
}
\AddEnumerateCounter{\enuma}{\@enuma}{53.13}
\makeatother

\makeatletter
\newcommand*{\enumb}[1]{%
	\expandafter\@enumb\csname c@#1\endcsname%
}
\newcommand*{\@enumb}[1]{%
	$\ifcase#1\or(b_-)\or(b_+)%
	\else\@ctrerr\fi$%
}
\AddEnumerateCounter{\enumb}{\@enumb}{53.13}
\makeatother

\usepackage{enumerate}
\newlist{steps}{enumerate}{1}
\setlist[steps, 1]{itemsep=8pt,leftmargin=0cm,itemindent=.5cm,labelwidth=\itemindent,labelsep=0cm,align=left,label = \textbf{\emph{Step \arabic*}:\,}}

\makeatletter
\newtheorem*{rep@theorem}{\rep@title}
\newcommand{\newreptheorem}[2]{%
\newenvironment{rep#1}[1]{%
 \def\rep@title{#2 \ref{##1}}%
 \begin{rep@theorem}}%
 {\end{rep@theorem}}}
\makeatother

\makeatletter
\newtheorem*{rep@cor}{\rep@title}
\newcommand{\newrepcor}[2]{%
\newenvironment{rep#1}[1]{%
 \def\rep@title{#2 \ref{##1}}%
 \begin{rep@cor}}%
 {\end{rep@cor}}}
\makeatother

\makeatletter
\newtheorem*{rep@prop}{\rep@title}
\newcommand{\newrepprop}[2]{%
\newenvironment{rep#1}[1]{%
 \def\rep@title{#2 \ref{##1}}%
 \begin{rep@prop}}%
 {\end{rep@prop}}}
\makeatother

\newtheorem{corollary}{Corollary}[section]
\newtheorem{corx}{Corollary}

\newtheorem{theorem}[corollary]{Theorem}
\newtheorem{thmx}[corx]{Theorem}

\newtheorem{proposition}[corollary]{Proposition}
\newtheorem{propx}[corx]{Proposition}

\newrepcor{cor}{Corollary}
\newreptheorem{theorem}{Theorem}
\newrepprop{prop}{Proposition}

\newtheorem*{theorem*}{Theorem}
\newtheorem{lemma}[corollary]{Lemma}

\theoremstyle{definition} 
\newtheorem{definition}[corollary]{Definition}

\newtheorem{example}[corollary]{Example}

\theoremstyle{remark} \newtheorem{remark}[corollary]{Remark} \numberwithin{equation}{section}
\newtheorem*{remark*}{Remark}
\numberwithin{figure}{section}

\renewcommand{\phi}{\varphi}

\newcommand{\R}{\mathbb{R}}

\newcommand{\epi}[1]{\mathrm{ep}({#1})}
\newcommand{\gra}[1]{\mathrm{gr}({#1})}
\newcommand{\sepi}[1]{\mathrm{ep\!}^\circ({#1})}
\newcommand{\dom}[1]{\mathsf{dom}({#1})}
\newcommand{\env}[1]{\overline{#1}}
\newcommand{\interior}{\mathsf{int}\,}

\newcommand{\C}{\mathsf{C}}

\newcommand{\hsp}[1]{\mathcal{S}_{#1}}
\newcommand{\hnull}[1]{\mathcal{N}_{#1}}

\renewcommand{\matrix}[1]{\begin{pmatrix}#1\end{pmatrix}}

\newcommand{\SC}{\mathsf{S}}
\renewcommand{\P}{\mathbb{P}}

\newcommand{\T}{\mathsf{T}}

\newcommand{\transp}[1]{\leftidx{^\mathrm{t}}{#1}}

\newcommand{\id}{\mathrm{id}}

\newcommand{\Axis}{\mathsf{Axis}}

\newcommand{\pa}{\partial}

\newcommand{\GL}{\mathrm{GL}}
\newcommand{\SL}{\mathrm{SL}}

\newcommand{\PGL}{\mathrm{PGL}}
\newcommand{\RP}{\mathbb{RP}}
\newcommand{\Hom}{\mathrm{Hom}}

\newcommand{\Aut}{\mathsf{Aut}}

\newcommand{\eg}{\textit{e.g.\@ }}
\newcommand{\cf}{\textit{c.f.\@ }}
\newcommand{\ie}{\textit{i.e.\@ }}

\DeclareMathOperator{\End}{End}

\DeclareMathOperator{\SO}{SO}

\newcommand{\D}{\mathsf{D}}

\newcommand{\Isom}{\mathsf{Isom}}

\newcommand{\LC}{\mathsf{LC}}

\newcommand{\A}{\mathbb{A}}

\begin{document}

\title{Affine deformations of quasi-divisible convex cones} 
\author{Xin Nie}
\address{Xin Nie: Tsinghua University, Beijing 100086, China} \email{nie.hsin@gmail.com}
\author{Andrea Seppi}
\address{Andrea Seppi: CNRS and Universit\'e Grenoble Alpes, 100 Rue des Math\'ematiques, 38610 Gi\`eres, France.} \email{andrea.seppi@univ-grenoble-alpes.fr}
\maketitle

\begin{abstract}
For any subgroup of $\SL(3,\R)\ltimes\R^3$ obtained by adding a translation part to a subgroup of $\SL(3,\R)$ which is the fundamental group of a finite-volume convex projective surface, we first show that under a natural condition on the translation parts of parabolic elements, the affine action of the group on $\R^3$ has convex domains of discontinuity that are regular in a certain sense, generalizing a result of Mess for globally hyperbolic flat spacetimes. We then classify all these domains and show that the quotient of each of them is an affine manifold foliated by convex surfaces with constant affine Gaussian curvature. The proof is based on a correspondence between the geometry of an affine space endowed with a convex cone and the geometry of a convex tube domain. As an independent result, we show that the moduli space of such groups is a vector bundle over the moduli space of finite-volume convex projective structures, with rank equaling the dimension of the Teichm\"uller space.
\end{abstract}

\setcounter{tocdepth}{1}
\tableofcontents

\section{Introduction}
Given a Lie group $G$ containing $\SO_0(2,1)=\Isom^+(\mathbb{H}^2)$ and a closed hyperbolic surface $S$ with fundamental group $\pi_1(S)$ identified as a Fuchsian group in $\SO_0(2,1)$, representations $\pi_1(S)\to G$ that are deformations of the inclusion are objects of study in \emph{higher Teichm\"uller theory}.
We study in this paper the case where $G$ is the group $\SL(3,\R)\ltimes\R^3$ of special affine transformations of the real affine space $\A^3$. This can be viewed as the combination of two well studied cases:
\begin{itemize}
\item
the isometry group $\SO(2,1)\ltimes\R^{2,1}$ of the Minkowski space $\R^{2,1}$, where the deformations give rise to \emph{maximal globally hyperbolic flat spacetimes} \cite{mess} (see also \cite{beneguad,MR2110829,bonsanteJDG,kraschl,belraouti,bonsanteseppiIMRN,fillastreseppi2});
\item the special linear group $\SL(3,\R)$, where the deformations yield convex real projective structures (see \eg \cite{goldman_convex,benoist_survey,Kim-Papadopoulos,CLM}).
\end{itemize}
We also extend the setting by allowing $S$ to have punctures. 

\subsection*{Affine deformations, regular domains and CAGC surfaces}
Given a proper convex cone $C$ in $\R^3$ (see \S \ref{subsec_cspacelike} for the definition), we let $\Aut(C)<\SL(3,\R)$ denote the group of special linear transformations preserving $C$, which is also the group of orientation-preserving projective automorphisms of the convex domain $\P(C)\subset\RP^2$. 

Following \cite{benoist_survey, marquis_around}, $\P(C)$ is said to be \emph{divisible} (resp.\@ \emph{quasi-divisible}) by a group $\Gamma<\SL(3,\R)$ if $\Gamma$ is discrete, contained in $\Aut(C)$, and the quotient $\P(C)/\Gamma$ is compact (resp.\@ has finite volume with respect to the Hilbert metric). Furthermore, we always assume $\Gamma$ is torsion-free, so that the quotient is a closed (resp.\@  finite-volume) \emph{convex projective surface}. 
Abusing the terminology, we also say that $C$ is (quasi-)divisible by $\Gamma$ if $\P(C)$ is.

Given a map $\tau:\Gamma\to\R^3$, a subgroup in $\SL(3,\R)\ltimes\R^3$ of the form 
$$
\Gamma_\tau:=\big\{(A,\tau(A))\in\SL(3,\R)\ltimes\R^3\,\big|\, A\in\Gamma\big\}
$$
is called an \emph{affine deformation} of $\Gamma$. The group relation forces $\tau$ to be an element in the space $Z^1(\Gamma,\R^3)$ of cocycles. We call $\tau$ \emph{admissible} if for every parabolic element
$A\in\Gamma$,  $\tau(A)$ is contained in the $2$-dimensional subspace of $\R^3$ preserved by $A$. This condition is vacuous if $C$ is divisible by $\Gamma$, in which case there is no parabolic element.

In \cite{nie-seppi}, we generalized standard notions in Minkowski geometry, such as spacelike/null planes and regular domains, to \emph{$C$-spacelike/$C$-null planes} and \emph{$C$-regular domains} in $\A^3$, defined with respect to a proper convex cone $C$ in the underlying vector space $\R^3$ (\cf \S \ref{subsec_cspacelike} and \S \ref{subsec_cregular} below). Our first result is:
\begin{thmx}\label{thm_intro1}
	Let $C\subset\R^3$ be a proper convex cone quasi-divisible by a torsion-free group $\Gamma< \SL(3,\R)$ and let $\tau\in Z^1(\Gamma,\R^3)$. Then 
\begin{enumerate}[label=(\arabic*)]
\item\label{item_intro11} 
 There exists a $C$-regular domain in $\A^3$ preserved by $\Gamma_\tau$ if and only if $\tau$ is admissible.
	In this case, there is a unique continuous map $f$ from $\pa\P(C)$ to the space of $C$-null planes in $\A^3$ which is equivariant in the sense that $f(A.x)=(A,\tau(A)).f(x)$ for all $x\in\pa\P(C)$ and $A\in\Gamma$. The complement of the union of planes $\bigcup_{x\in\pa\P(C)}f(x)$ in $\A^3$ has two connected components $D^+$ and $D^-$, which are $C$-regular and $(-C)$-regular domains preserved by $\Gamma_\tau$, respectively. 
	\item\label{item_intro12} If $C$ is divisible by $\Gamma$, then $D^+$ is the unique $C$-regular domain preserved by $\Gamma_\tau$. Otherwise, assume the surface $S:=\P(C)/\Gamma$ has $n\geq1$ punctures and $\tau$ is admissible, then all the $C$-regular domains preserved by $\Gamma_\tau$ form a family $(D_\mu)$ parameterized by $\mu\in\R_{\geq0}^n$, such that $D_{(0,\cdots,0)}=D^+$ and we have $D_\mu\subset D_{\mu'}$ if and only if $\mu$ is coordinate-wise larger than or equal to $\mu'$. 
	\item\label{item_intro13}
 $\Gamma_\tau$ acts freely and properly discontinuously on every $C$-regular domain preserved by it, with quotient homeomorphic to $S\times\R$.
\end{enumerate}
\end{thmx}
When $C$ is the future light cone $C_0\subset\R^{2,1}$, the divisible case of this theorem is part of the seminal work of Mess \cite{mess}. Brunswic \cite{Brunswic_1, Brunswic_2} has obtained results in the quasi-divisible case for $C_0$ as well. For general $C$, in the divisible case, the equivariant continuous map given by Part \ref{item_intro11} is related to the Anosov property of $\Gamma_\tau$, studied by Barbot \cite{Barbot_Anosov} and Danciger-Gu\'eritaud-Kassel \cite{DGK} in different but related settings.
\begin{remark*}
For the trivial deformation $\tau=0$, we have $D^+=C$, and any other $C$-regular domain preserved by $\Gamma_\tau=\Gamma$ is obtained by first choosing a $\Gamma$-invariant family of $C$-null planes intersecting $C$, such that each plane in the family is preserved by some parabolic element, and then trimming $C$ along these planes. Each puncture of $S$ corresponds to a conjugacy class of parabolic elements and gives rise to an $\R_{\geq0}$-worth of choices. For a general admissible $\tau$,  $D_\mu$ is obtained from the maximal $C$-regular domain $D^+$ preserved by $\Gamma_\tau$ by the same construction, which explains Part \ref{item_intro12} of Theorem \ref{thm_intro1}.
Note that although $D^\pm$ is the maximal $(\pm C)$-regular domain preserved by $\Gamma_\tau$, it is observed in \cite{mess} that $D^+\cup D^-$ might not be the maximal domain of discontinuity for the $\Gamma_\tau$-action on $\A^3$, even in the divisible and $C=C_0$ case.
\end{remark*}

Our second result establishes a canonical ``time function'' on each $C$-regular domain in $\A^3$ preserved by $\Gamma_\tau$, whose level surfaces have Constant Affine Gaussian Curvature (CAGC):
\begin{thmx}\label{thm_intro2}
Let $C\subset\R^3$ be a proper convex cone quasi-divisible by a torsion-free group $\Gamma< \SL(3,\R)$, $\tau\in Z^1(\Gamma,\R^3)$ be an admissible cocycle and $D$ be a $C$-regular domain preserved by $\Gamma_\tau$. 
Then for any $k>0$, $D$ contains a unique complete affine $(C,k)$-surface $\Sigma_k$ generating $\pa D$. This surface is preserved by $\Gamma_\tau$ and is asymptotic to the boundary of $D$. Moreover, $(\Sigma_k)_{k>0}$ is a foliation of $D$, and the function 
$K:D\to\R$ given by $K|_{\Sigma_k}=\log k$ is convex.
\end{thmx}
Here, \emph{affine $(C,k)$-surfaces} are a particular class of convex surfaces with CAGC $k$, whose supporting planes are $C$-spacelike (see \S \ref{subsec_affineck} for details). The main result of \cite{nie-seppi} is a  statement similar to Theorem \ref{thm_intro2}, for $C$-regular domains without any group action assumed, but instead assuming that the planar convex domain $\P(C)$ satisfies the \emph{interior circle condition} at every boundary point. When $C$ is quasi-divisible, $\pa \P(C)$ is known to have at most $\C^{1,\alpha}$-regularity \cite{Benoist_I, Guichard}, and the condition is not satisfied. 

Our main motivation for establishing Theorems \ref{thm_intro1} and \ref{thm_intro2} is to produce affine $3$-manifolds which generalize maximal globally hyperbolic flat spacetimes:
\begin{corx}\label{coro_intro}
Let $C\subset\R^3$ be a proper convex cone quasi-divisible by a torsion-free group $\Gamma< \SL(3,\R)$ and denote the surface $\P(C)/\Gamma$ by $S$. Let $\tau\in Z^1(\Gamma,\R^3)$ be an admissible cocycle and $D$ be a $C$-regular domain preserved by $\Gamma_\tau$. Then there is a homeomorphism from the affine manifold $M:=D/\Gamma_\tau$ to $S\times\R_+$, such that each slice $S\times\{k\}$ is a locally strongly convex surface with CAGC $k$ with respect to the affine structure on $M$, and the projection to the $\R_+$-factor is a locally $\log$-convex function on $M$ with respect to the affine structure.
\end{corx}
For the future light cone $C_0$, an affine $(C_0,k)$-surface is just a spacelike, future-convex surface in $\R^{2,1}$ with classical Gaussian curvature $k^\frac{2}{3}$ (or intrinsic curvature $-k^\frac{2}{3}$; \cf \cite[Prop.\@ 3.7]{nie-seppi}). When $S$ is closed (\ie when $C$ is divisible by $\Gamma$), some of the statements in Corollary \ref{coro_intro} are contained in the works Barbot-B\'eguin-Zeghib \cite{bbz} for $C_0$ and Labourie \cite[\S 8]{labourie} for general $C$.

\subsection*{Moduli space of admissible deformations}
Two natural questions that one might ask while looking at the above results are: what are all the quasi-divisible proper convex cones in $\R^3$, and what are all their admissible affine deformations? 

It follows from results of Marquis \cite{marquis_fourier} that in the above setting, the orientable surface $S=\P(C)/\Gamma$ is homeomorphic to either the torus or the surface $S_{g,n}$ of negative Euler characteristic with genus $g$ and $n$ punctures. Since the case of torus is simple 
(see Remark \ref{remark_torus}), we will only look into the above questions for $S_{g,n}$.

 
The first question essentially asks for a description of the moduli space $\mathcal{P}_{g,n}$ of finite-volume convex projective structures on $S_{g,n}$. For $n=0$, Goldman \cite{goldman_convex} first provided a Fenchel-Nielsen type description, then Labourie \cite{labourie} and Loftin \cite{loftin_amer} obtained a holomorphic one. The two descriptions are generalized by Marquis \cite{marquis} and Benoist-Hulin \cite{benoist-hulin}, respectively, to $n\geq1$. These results imply that $\mathcal{P}_{g,n}$ is homeomorphic to a ball of dimension $16g-16+6n$. 

The second question is concerned with the moduli space $\widehat{\mathcal{P}}_{g,n}$ of representations $\rho:\pi_1(S_{g,n})\to\SL(3,\R)\ltimes\R^3$ such that the $\SL(3,\R)$-component of $\rho$ defines a finite-volume convex projective structure and the $\R^3$-component is given by an admissible cocycle. With elementary arguments, we show: 
\begin{propx}\label{prop_intro1}
For any $g,n\geq0$ with $2-2g-n<0$, $\widehat{\mathcal{P}}_{g,n}$ is a topological vector bundle over $\mathcal{P}_{g,n}$ of rank $6g-6+2n$.
\end{propx}
Note that the rank equals the dimension of the Teichm\"uller space $\mathcal{T}_{g,n}$. In fact, $\mathcal{T}_{g,n}$ is naturally contained in  $\mathcal{P}_{g,n}$, and the part of $\widehat{\mathcal{P}}_{g,n}$ over $\mathcal{T}_{g,n}$ can be identified with the tangent bundle of $\mathcal{T}_{g,n}$.
While Mess \cite{mess} has introduced several new ideas to study this part, generalization of his methods to $\widehat{\mathcal{P}}_{g,n}$ is an interesting task not yet undertaken.

\subsection*{Affine space with a cone vs.\@ convex tube domain}
The main tool in the proof of Theorems \ref{thm_intro1} and \ref{thm_intro2}, also used implicitly in \cite{nie-seppi}, is a correspondence between the following two  geometries:
\begin{itemize}
	\item The geometry of $\A^3$ with respect to the group $\Aut(C)\ltimes \R^3$ of special affine transformations whose linear parts preserve a given proper convex cone $C\subset\R^3$.
	\item The geometry of a \emph{convex tube domain} in $\R^3$, \ie an open set of the form $T=\Omega\times\R$ with $\Omega$ a bounded planar convex domain, with respect a group $\Aut(T)$ of certain projective transformations, which we call the \emph{automorphisms} of $T$. 
\end{itemize}
We refer to \S \ref{subsec_auto} below for the precise definition of automorphisms of $T$, only mentioning here that they can be roughly understood as the projective transformations $\Phi\in\PGL(4,\R)$ preserving $T$ which are given by matrices of the form
$$
\matrix{B&\\[3pt]\transp{Y}&1},\quad B\in\SL(3,\R),Y\in\R^3
$$
(multiplying $B$ and the lower-right $1$ by different constants yields a projective transformation preserving $T$ but not in $\Aut(T)$), and the groups in the two geometries are isomorphic to each other through
$$
\matrix{A&X\\[3pt]&1}\longleftrightarrow \matrix{	\transp{\!A}^{-1}&\\[3pt]\transp{(A^{-1}X)}&1},
$$
where the first matrix represents the element $(A,X)$ in $\Aut(C)\ltimes\R^3$. As one might guess from the appearance of the inverse transpose $\transp{\!A}^{-1}$, the convex domain $\Omega$ in the second geometry can be identified with a section of the cone $C^*\subset\R^{3*}$ dual to the cone $C$ from the first geometry.

When $C$ is the future light cone $C_0\subset\R^{2,1}$, the first geometry is just that of the Minkowski space $\R^{2,1}$, whereas the second is the \emph{co-Minkowski geometry} of the round tube (see \eg \cite{struvestruve,dancigerGT,barbot-fillastre,fillastreseppisurvey,DMS}). We proceed to give more details for general $C$.

A \emph{polarization} of the affine space $\A^3$ is a choice of a point $p_\infty$ on the plane at infinity $P_\infty:=\mathbb{RP}^3\setminus\A^3\cong\RP^2$. Given $p_\infty$, we define the \emph{dual polarized affine space} $\mathbb{A}^{3*}$ as the space of affine planes in $\A^{3}$ not containing $p_\infty$ at infinity, which is an affine chart in the dual projective space $\mathbb{RP}^{3*}=\{\text{affine planes in $\A^3$}\}\cup\{P_\infty\}$, with polarization just given by $P_\infty$.

Given a proper convex cone $C\subset\R^3$ whose projectivization $\P(C)\subset P_\infty$ contains $p_\infty$, every $C$-spacelike plane in $\A^3$ can be viewed as a point in $\A^{3*}$. We will show:
\begin{propx}\label{prop_intro2}
	The set $\hsp{C}:=\big\{\text{$C$-spacelike planes in $\A^3$}\big\}$ is a convex tube domain in $\A^{3*}$, whose underlying planar convex domain is a section of the dual cone $C^*$. The natural action of $\Aut(C)\ltimes\R^3$ on $\hsp{C}$ induces an isomorphism between $\Aut(C)\ltimes\R^3$ and the automorphism group of $\hsp{C}$ as a convex tube domain.
\end{propx}
\begin{table}[ht]
	\begin{TAB}(c)[4pt]{|c|c|}{|c|c|c|c|c|c|c|c|c|c|}
		\textbf{objects in $\A^3$}& \textbf{objects in $\A^{3*}$}\\
		point& plane not containing $P_\infty$ at infinity\\
		plane not containing $p_\infty$ at infinity & point \\
		$C$-spacelike plane & point in $\Omega\times\R$\\
		$C$-null plane & point in $\pa\Omega\times\R$  \\
		\parbox[l]{5.3cm}{affine transformation in $\Aut(C)\ltimes\R^3$}&
		\parbox[l]{4cm}{automorphism of $\Omega\times\R$}
		\\
		\parbox[l]{6.6cm}{subgroup of $\Aut(C)\ltimes\R^3$ of the form $\Gamma_\tau$, where $\Gamma$ quasi-divides $C$ and $\tau$ is admissible}&
		\parbox[l]{7.3cm}{subgroup of $\Aut(\Omega\times\R)$ projecting bijectively to a group quasi-dividing $\Omega$ s.t.\@ every element with parabolic projection has fixed point in $\pa\Omega\times\R$}
		\\
		\parbox[l]{6.4cm}{$C$-regular (resp.\@ $(-C)$-regular) domain $D$}&
		\parbox[l]{6.2cm}{graph in $\pa\Omega\times\R$ of a lower (resp.\@ upper) semicontinuous function $\phi$ on $\pa\Omega$}
		\\
		\parbox[l]{6.7cm}{smooth, strongly convex, complete $C$-convex surface generating a $C$-regular domain $D$ (\cf \S \ref{subsec_cregular})}
		&
		\parbox[l]{6.7cm}{
			graph in $\Omega\times\R$ of a function $u\in\SC_0(\Omega)$ (\cf \S \ref{subsec_cregular}) whose boundary value $\phi$ corresponds to $D$ (\cf the last row)}\\
		affine $(C,k)$-surface (\cf \S \ref{subsec_affineck})&
		\parbox[l]{6.6cm}{
			graph in $\Omega\times\R$ of some $u\in\SC_0(\Omega)$ satisfying  $\det\D^2u=k^{-\frac{2}{3}}w_\Omega^{-4}$ (Prop.\@ \ref{prop_mongeampere})   
		}
	\end{TAB}
	\caption{Dictionary between the two geometries.}
	\label{table_dic}
\end{table}
In summary, we have obtained the first few rows of  Table \ref{table_dic} (where we write the convex tube domain $\hsp{C}$ as $\Omega\times\R$).
The rest of the table will be explained in \S \ref{sec:cregular}. This dictionary enables us to deduce Theorems \ref{thm_intro1} and \ref{thm_intro2} from the following dual results about convex tube domains:
\begin{thmx}\label{thm_intro3}
	Let $\Omega$ be a bounded convex domain in $\R^2\subset\mathbb{RP}^2$ quasi-divisible by a torsion-free group $\Lambda$ of projective transformations, and  $\Lambda'<\Aut(\Omega\times\R)$ be a group of automorphisms of the convex tube domain $\Omega\times\R$
	which projects to $\Lambda$ bijectively. Then 
	\begin{enumerate}[label=(\arabic*)]
		\item\label{item_d1} The following conditions are equivalent to each other:
		\begin{enumerate}[label=(\alph*)]
		\item\label{item_d11} every element of $\Lambda'$ with parabolic projection in $\Lambda$ has a fixed point in $\pa\Omega\times\R$; 
		 \item\label{item_d12} there exists a continuous function $\phi\in\C^0(\pa\Omega)$ with graph $\gra{\phi}\subset\pa\Omega\times\R$  preserved by $\Lambda'$;
		 \item\label{item_d13} there exists a lower semicontinuous function $\widehat{\phi}:\Omega\to\R\cup\{+\infty\}$, which is not constantly $+\infty$, with graph preserved by $\Lambda'$;
		 \end{enumerate}
	\item\label{item_d1.5} Suppose these conditions are fulfilled. Then the function $\phi$ in \ref{item_d12} is unique. On the other hand, the function $\widehat{\phi}$ in \ref{item_d13} is unique and equals $\phi$ only if $\Omega$ is divisible by $\Lambda$. Otherwise, all the $\widehat{\phi}$'s can be described as follows. Let $\mathcal{F}\subset\pa\Omega$ be the set of fixed points of parabolic elements in $\Lambda$ and pick $p_1,\cdots,p_n\in\mathcal{F}$ such that $\mathcal{F}$ is the disjoint union of the orbits $\Lambda.p_j$, $j=1,\cdots,n$. For each $\mu=(\mu_1,\cdots,\mu_n)\in\R_{\geq0}^n$, let $\phi_\mu:\pa\Omega\to\R$ be the function with graph preserved by $\Lambda'$ such that 
	$$\phi_\mu(p_j)=\phi(p_j)-\mu_j\ \text{ for all $j$;}\ \ \phi_\mu=\phi\ \text{ on $\pa\Omega\setminus\mathcal{F}$}
	$$ 
	(with $\phi$ from \ref{item_d12}). Then $\phi_\mu$ is lower semicontinuous, and every $\widehat{\phi}$  in \ref{item_d13} equals some $\phi_\mu$.
    \item\label{item_d2} Let $w_\Omega\in\C^0(\overline{\Omega})\cap\C^\infty(\Omega)$ be the unique convex solution (established by Cheng-Yau \cite{chengyau1}, see Thm.\@ \ref{thm_chengyau} below) to the Dirichlet problem of Monge-Amp\`ere equation
    $$
    \begin{cases}
    \det\D^2w=w^{-4}\\
    w|_{\pa\Omega}=0
    \end{cases}
    $$
 Then for any $\phi_\mu$ from Part \ref{item_d1.5} and any $t>0$, the Dirichlet problem
		$$
		\begin{cases}
		\det\D^2u=e^{-t}w_\Omega^{-4}\\
		u|_{\pa\Omega}=\phi_\mu
		\end{cases}
		$$
	  has a unique convex solution $u_t\in \C^\infty(\Omega)$. It has the following properties:
	  \begin{itemize}
	   \item $\|\D u_t\|$ tends to $+\infty$ on the boundary of $\Omega$;
	  	\item the graph $\gra{u_t}\subset\Omega\times\R$ is  preserved by $\Lambda'$;
        \item for every fixed $x\in\Omega$, $t\mapsto u_t(x)$ is a strictly increasing concave function, with value tending to $-\infty$ and $\overline{\phi}_\mu(x)$ as $t$ tends to $-\infty$ and $+\infty$, respectively. 
	\end{itemize}	
	\end{enumerate}
\end{thmx}
In the last part, the function $\env{\phi}_\mu$ is the \emph{convex envelope} of $\phi_\mu$. We refer to \S \ref{subsec_convexfunction} for its definition and for the precise meaning of the boundary value $u|_{\pa\Omega}$ when $u$ is a convex function on $\Omega$.

Theorem \ref{thm_intro3} gives a picture similar to the familiar ones from quasi-Fuchsian hyperbolic manifolds and globally hyperbolic anti-de Sitter spacetimes, see Figure \ref{figure_quasifuchsian}.
\begin{figure}[h]
	\includegraphics[width=9cm]{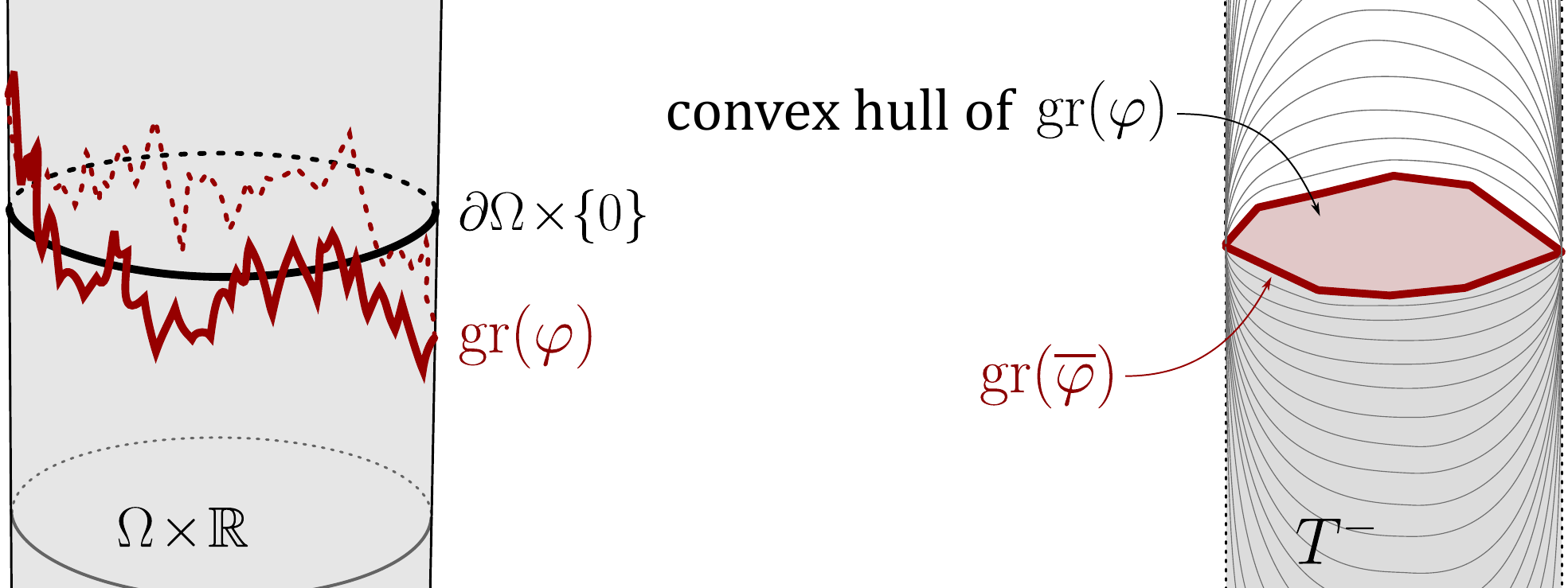}
	\caption{Invariant Jordan curve on the boundary of the convex tube domain (left) and a section view of its convex hull (right).}
	\label{figure_quasifuchsian}
\end{figure}
Namely, the group $\Lambda<\Aut(\Omega)$, viewed as automorphisms of $\Omega\times\R$, is analogous to a Fuchsian group acting on $\mathbb{H}^3$ or $\mathbf{AdS}_3$, which preserves a slice in the $3$-space as well as the boundary circle of the slice. After a perturbation, we obtain the group $\Lambda'$ which preserves the Jordan curve $\gra{\phi}$ on the boundary. The graphs of the family of convex functions $(u_t)$ produced by Part \ref{item_d2} then gives a canonical foliation for the lower part $T^-$ of $\Omega\times\R$ outside of the convex hull of the curve, while the upper part is foliated in the same way by graphs of concave functions. A new phenomenon here is that when the surface $\Omega/\Lambda$ has $n$ punctures, for each $\mu\in\R_{\geq0}^n$, there is a subdomain $T_\mu^-$ of $T^-$ preserved by $\Lambda'$, namely the part of $\Omega\times\R$ underneath $\gra{\env{\phi}_\mu}$, which is obtained by modifying the Jordan curve at parabolic fixed points and making it discontinuous. The theorem asserts that every $T_\mu^-$ is foliated in the same way as $T^-$.

%
 
\begin{remark*}
	The correspondence between the two geometries holds in any dimension $d\geq2$, despite of our restriction to $d=3$ here, and Theorem \ref{thm_intro1} also holds in higher dimensions as long as we restrict to divisible convex cones rather than quasi-divisible ones. We assume $d=3$ in this paper for two reasons. First, the theory of convex projective manifolds of dimension greater than $2$ is less complete than the case of surfaces --  there is no known description of their moduli spaces, and the structure of their ends is more complicated. The second, more irremediable reason is that even in the Minkowski setting, as shown by Bonsante and Fillastre \cite{bonsante-fillastre}, the existence of hypersurfaces with constant Gauss-Kronecker curvature in a regular domain in $\R^{d-1,1}$ is problematic when $d\geq4$, due to the existence of Pogorelov-type non-strictly convex solutions to the underlying Monge-Amp\`ere problem (see also \cite[Remark 8.1]{nie-seppi}). Therefore, Theorem \ref{thm_intro2} cannot be generalized to higher dimensions.
\end{remark*}

\subsection*{Organization of the paper}
We first give more details about the correspondence between the two geometries and prove Prop.\@ \ref{prop_intro2}
in \S \ref{sec:correspondence}, then we explain 
the last three rows of Table \ref{table_dic} in \S \ref{sec:cregular}.
In \S \ref{sec:affine deformation}, we review some backgrounds about quasi-divisible convex cone and affine deformations, then prove Prop.\@ \ref{prop_intro1}. Finally, in \S \ref{sec:last proof}, we prove Thm.\@ \ref{thm_intro3} using results from \cite{nie-seppi} and explain how the other main results of this paper, namely Thm.\@ \ref{thm_intro1}, \ref{thm_intro2} and Cor.\@ \ref{coro_intro}, are deduced from Thm.\@ \ref{thm_intro3}.
\subsection*{Acknowledgments}
We are deeply indebted to Thierry Barbot for the discussions related to the subject of this paper and for his interest and encouragement.
The first author would like to thank Yau Mathematics Sciences Center for the hospitality during the preparation of the paper.

\section{Correspondence between the two geometries}\label{sec:correspondence}
In this section, we prove Proposition \ref{prop_intro2} after giving details about the constructions involved, which correspond to the first few rows of Table \ref{table_dic}. 
\subsection{From $C$-spacelike planes to convex tube domain}\label{subsec_cspacelike}
A convex domain in a vector space or an affine space is said to be \emph{proper} if it does not contain any entire straight line. In a vector space,  a \emph{convex cone} is by definition a convex domain invariant under positive scalings. We henceforth fix a proper convex cone $C\subset\R^3$ and a splitting $\R^3=\R^2\times\R$ such that 
$$
C=\{t(x,1)\mid x\in \Omega^*,\ t>0\}
$$ 
for a bounded convex domain $\Omega^*$ in $\R^2$ containing the origin\footnote{The reason for the asterisk in the notation is that $\Omega^*$ will be the dual of another convex domain $\Omega$ introduced later. Although it seems more natural here to switch the notations for $\Omega^*$ and $\Omega$, this would be inconvenient for our purpose because we mainly work with $\Omega$ in this paper.}. A point in $\R^3$ will often be written in the form $(x,\xi)$, where $x\in\R^2$ and $\xi\in\R$ are the ``horizontal'' and ``vertical'' coordinates, respectively.
Also fix an inner product ``$\,\cdot\,$'' on $\R^2$.

Using the splitting, we endow the affine space $\A^3\cong\R^3$ with the polarization given by the point at infinity corresponding to the vertical lines $\{x\}\times\R$ (see the introduction for the notions of polarization and dual affine space). We then also identify the dual affine space $\A^{3*}$ with $\R^3$ in the following specific way:
\begin{align}
\R^3&\overset\sim\to\A^{3*}:=\big\{\text{non-vertical affine planes in $\A^3$}\big\},\label{eqn_identification_r3}\\
(x,\xi)&\mapsto \text{ graph of the affine function }y\mapsto x\cdot y-\xi.\nonumber
\end{align}

Following \cite{nie-seppi}, we introduce:
\begin{definition}\label{def_cspacelike}
A $2$-dimensional subspace $P_0\subset\R^3$ is said to be \emph{$C$-spacelike} if $P_0$ meets the closure $\overline{C}$ of $C$ only at the origin $0\in\R^3$, and is said to be \emph{$C$-null} if $P_0\cap\overline{C}$ is a subset of $\pa C$ that is not the single point $0$. An affine plane $P\subset \A^3$ is said to be $C$-spacelike/$C$-null if the vector subspace underlying $P$ (\ie the translation of $P$ to $0$) is. We let $\hsp{C}$ and $\hnull{C}$ denote the set of all $C$-spacelike and $C$-null planes in $\A^3$, respectively.
\end{definition}

Since a $C$-spacelike or $C$-null plane does not contain any vertical line, we can view it as a point in $\A^{3*}$, hence view $\hsp{C}$ and $\hnull{C}$ as subset of $\A^{3*}$. In order to describe these subsets, we consider the following set $\Omega\subset\R^2$ derived from the above $\Omega^*$:
$$
\Omega:=\{x\in\R^2\mid x\cdot y<1\ \text{ for all }y\in\overline{\Omega^*}\}.
$$
It can be shown that $\Omega$ is also a bounded convex domain containing the origin. It is in fact the dual of $\Omega^*$ in the sense of Sasaki \cite{sasaki}, which explains the  notation.

Now the 3rd and 4th rows of Table \ref{table_dic} can be stated precisely as:
\begin{proposition}\label{prop_hspace}
The bijection \eqref{eqn_identification_r3} identifies the convex tube domain $\Omega\times\R\subset\R^3$ and its boundary $\pa\Omega\times\R$ with the subsets $\hsp{C}$ and $\hnull{C}$ of $\A^{3*}$, respectively.
\end{proposition}
The proof is straightforward and we omit it here.

\subsection{$\Omega$ as a section of $-C^*$}\label{subsec_omega}
The convex domain $\Omega$ can be interpreted geometrically as follows. Recall that the \emph{dual cone}
$C^*$ of $C$ is defined as the convex cone in the dual vector space $\R^{3*}$ (the space of all linear forms on $\R^3$) consisting of linear forms with positive values on $\overline{C}\setminus\{0\}$. We extend the inner product on $\R^2$ to $\R^3$ by setting 
$$
(x,\xi)\cdot(y,\eta):=x\cdot y+\xi\eta,
$$ 
and use it to identify the dual vector space $\R^{3*}$  with $\R^3$ itself. Then it is easy to check that 
$\Omega$ is exactly the section of the opposite dual cone $-C^*\subset\R^{3*}\cong\R^3$ by the horizontal plane $\R^2\times\{-1\}$. In other words, we can write
\begin{equation}\label{eqn_oppositecone}
-C^*=\{t(x,-1)\mid x\in\Omega,\ t>0\}.
\end{equation}

The significance of this interpretation is that it identifies $\Omega$ projectively with the convex domain $\P(C^*)$ in $\mathbb{RP}^{2*}:=\P(\R^{3*})$. We let $\Aut(\Omega)$ denote the group of orientation-preserving projective transformations of $\Omega$, which identifies with the subgroup $\Aut(C^*)$ of $\SL(3,\R)$ that preserves the cone $C^*$. Using \eqref{eqn_oppositecone}, one checks that the image of $x\in\Omega$ by the projective action of $B\in\Aut(\Omega)$ has the expression
$$
B.\,x:=\frac{\left[B\matrix{x\\-1}\right]_{1,2}}{\left[-B\matrix{x\\-1}\right]_3} ~,
$$
where $[X]_{1,2}\in\R^2$ and $[X]_3\in\R$ denote the horizontal and vertical components of $X\in\R^3=\R^2\times\R$, respectively. Note that the natural isomorphism $\Aut(C)\cong\Aut(C^*)$ given by the inverse transpose $A\leftrightarrow\transp{\!A}^{-1}$ induces an $\Aut(C)$-action on $\Omega$, sending $x\in\Omega$ to $\transp{\!A}^{-1}\!.\,x$. 
\subsection{The action of $\Aut(C)\ltimes\R^3$ on $C$-spacelike and $C$-null planes.}\label{subsec_actioncspaclike}
We always view the semi-direct product $\SL(3,\R)\ltimes\R^3$ as the group of special affine transformations of $\A^3\cong\R^3$ in such a way that $(A,X)\in\SL(3,\R)\ltimes\R^3$ represents the transformation $Y\mapsto AY+X$. The components $A$ and $X$ are called the \emph{linear part} and \emph{translation part} of the affine transformation, respectively.

The subgroup $\Aut(C)\ltimes\R^3\subset \SL(3,\R)\ltimes\R^3$ of special affine transformations with linear part preserving $C$
naturally acts on the spaces $\hsp{C}$ and $\hnull{C}$ of $C$-spacelike and $C$-null planes. The identifications in Proposition \ref{prop_hspace} translate these actions to a natural action on $\overline{\Omega}\times\R$, with the following coordinate expression:
\begin{proposition}\label{prop_action}
	The action of $\Aut(C)\ltimes\R^3$ on $\hsp{C}\cup\hnull{C}\cong\overline{\Omega}\times\R$ is given by
	$$
	(A,X).(x,\xi)=\frac{1}{\left[-\transp{\!A}^{-1}\matrix{x\\-1}\right]_3}\left(\left[\transp{\!A}^{-1}\matrix{x\\-1}\right]_{1,2}\,,\ \xi+A^{-1}X\cdot\matrix{x\\-1}\right)
	$$
for any $(A,X)\in\Aut(C)\ltimes\R^3$ and  $(x,\xi)\in\overline{\Omega}\times\R$. 
\end{proposition}
Comparing with the last paragraph of the previous subsection, one sees that the horizontal component of $(A.X).(x,\xi)$ is exactly $\transp{\!A}^{-1}\!.\,x$, namely the image of $x$ by the projective action of $A$. The reason for this will be clear in \S \ref{subsec_iso} below.
\begin{proof}
We need to determine $(x',\xi'):=(A,X).(x,\xi)$. Under the identification \eqref {eqn_identification_r3}, $(x,\xi)$ corresponds to the plane $P$ that is the graph of the affine function  $y\mapsto x\cdot y-\xi$. Equivalently, we have
$$P=\left\{\matrix{y\\ \eta}\in\R^3\ \middle|\ \matrix{y\\ \eta}\cdot \matrix{x\\ -1}=\xi\right\}.$$
Let us now compute the image of $P$ by $(A,X)$:
\begin{align*}
(A,X).P&=\left\{A\matrix{y\\ \eta}+X\ \middle|\ \matrix{y\\ \eta}\cdot \matrix{x\\ -1}=\xi\right\} \\
&=\left\{A\matrix{y\\ \eta}+X\ \middle|\ \left(A\matrix{y\\ \eta}+X\right)\cdot \left(\transp{\!A}^{-1}\matrix{x\\ -1}\right)=\xi+X\cdot \transp{\!A}^{-1}\matrix{x\\ -1}\right\} \\
&=\left\{\matrix{y'\\ \eta'}\ \middle|\ \matrix{y'\\ \eta'}\cdot \left(\transp{\!A}^{-1}\matrix{x\\ -1}\right)=\xi+A^{-1}X\cdot \matrix{x\\ -1} \right\}.
\end{align*}
Dividing by $\left[-\transp{\!A}^{-1}\matrix{x\\-1}\right]_3$, this shows that $(A,X).P$ corresponds to the pair $(x',\xi')\in \overline\Omega\times\R$ with the required expression
$$x'=\frac{\left[\transp{\!A}^{-1}\matrix{x\\-1}\right]_{1,2}}{\left[-\transp{\!A}^{-1}\matrix{x\\-1}\right]_3}~,\quad \xi'=\frac{\xi+A^{-1}X\cdot\matrix{x\\-1}}{\left[-\transp{\!A}^{-1}\matrix{x\\-1}\right]_3}~.
$$ 	
\end{proof}

\subsection{Automorphisms of convex tube domains}\label{subsec_auto}
We refer to a subset of $\R^3$ of the form $T=\Omega\times\R$ as a \emph{convex tube domain} if $\Omega$ is a bounded convex domain in $\R^2$. Viewing $\R^3$ as an affine chart in $\RP^3$, we are interested in projective transformations $\Phi\in\PGL(4,\R)$ of $\RP^3$ preserving $T$. Note that such a $\Phi$ fixes the point at infinity
$$
p_T:=\overline{T}\cap(\RP^3\setminus\R^3)
$$
of $T$ (where $\overline{T}$ denotes the closure of $T$ in $\RP^3$), which is just the common point at infinity of the vertical lines $\{x\}\times\R\subset\Omega\times\R$.

The space of vertical lines in $T$
naturally identifies with $\Omega$. Since $\Phi$ sends one vertical line to another, it induces a self-mapping of this space, which is actually a projective transformation of $\Omega$ because $\Phi$ sends a line in the space (\ie the set of those $\{x\}\times\R$ with $x$ belonging to a line in $\Omega$) to another line. This gives a projection
$$
\pi:\{\Phi\in\PGL(4,\R)\mid \Phi(T)=T\}\to\Aut^\pm(\Omega),
$$
where we let $\Aut^\pm(\Omega)$ denote the group of projective transformations of $\Omega$, and reserve the notation $\Aut(\Omega)$ for orientation-preserving projective transformations.

\begin{definition}\label{def_auto}
A projective transformation $\Phi$ is said to be an \emph{automorphism} of a convex tube domain $T$ if it preserves $T$ and satisfies the following extra conditions:
\begin{enumerate}[label=(\roman*)]
	\item\label{item_auto1} $\Phi$ does not switch the two ends of $T$. In other words, if we endow each vertical line in $T$ with the upward orientation, then $\Phi$ sends one line to another in an orientation-preserving way.
	\item\label{item_auto2} The projection $\pi(\Phi)$ is in $\Aut(\Omega)$ (\ie $\pi(\Phi)$ is orientation-preserving).
	\item\label{item_auto3} The eigenvalue of $\Phi$ at $p_T$ is $\pm1$.
\end{enumerate} 
Denote the group of all automorphism of $T$ by $\Aut(T)$.
\end{definition}

\begin{remark}
In Condition \ref{item_auto3}, by an \emph{eigenvalue} of $\Phi$, we mean an eigenvalue of a representative $\widetilde{\Phi}\in\GL(4,\R)$ of $\Phi\in\PGL(4,\R)$ with $\det(\widetilde{\Phi})=\pm1$. Since there are two such $\widetilde{\Phi}$'s opposite to each other and neither of them is privileged over the other, this eigenvalue is well defined only up to sign. Condition \ref{item_auto3} rules out, for example, dilations of the $\R$-factor. 
\end{remark}

We will study graphs of (extended-real-valued) functions on $\Omega$ or $\pa\Omega$ as geometric objects in the convex tube domain $\Omega\times\R$ or its boundary $\pa\Omega\times\R$, and will freely use the following basic facts, for any  $\Phi\in\Aut(\Omega\times\R)$:
\begin{itemize}
\item for any function $u$ on $\Omega$ or $\pa\Omega$, the image of the graph $\gra{u}$ by $\Phi$ is again the graph of some function $\widetilde{u}$ on $\Omega$ or $\pa\Omega$;
\item  if $u$ is lower/upper semicontinuous, convex or smooth, so is $\widetilde{u}$;
\item  if two functions $u_1$ and $u_2$ satisfy $u_1\leq u_2$, then we also have $\widetilde{u}_1\leq\widetilde{u}_2$.
\end{itemize}
 Here and below, we denote the graph of any extended-real-valued function $f$ on a set $E\subset\R^2$ by
 $$
 \gra{f}:=\{(x,\xi)\in E\times\R\mid f(x)=\xi\}.
 $$ 

\subsection{The isomorphism $\Aut(C)\ltimes\R^3\cong\Aut(\Omega\times\R)$}\label{subsec_iso}
We now return to the setting of \S \ref{subsec_cspacelike}$\sim$\ref{subsec_actioncspaclike}, where $\Omega$ is induced from a proper convex cone $C\subset\R^3$. Precisely, $\Omega$ is a section of $-C^*$, so that $\Aut(\Omega)$ identifies with $\Aut(C^*)$ (see \S \ref{subsec_omega}).

From the expression in Prop.\@ \ref{prop_action}, one can check that $\Aut(C)\ltimes\R^3$ acts on $\hsp{C}\cong\Omega\times\R$ by automorphisms of the convex tube domain $\Omega\times\R$, hence gives a homomorphism 
\begin{equation}\label{eqn_homo}
\Aut(C)\ltimes\R^3\to\Aut(\Omega\times\R).
\end{equation}
We proceed to show that \eqref{eqn_homo} is an isomorphism, which, together with the framework built in the previous subsections, implies Prop.\@ \ref{prop_intro2} in the introduction.
\begin{proposition}\label{prop_iso}
Let $C$ and $\Omega$ be as in \S \ref{subsec_cspacelike}. Then there is a natural isomorphism 
$$
\Aut(\Omega\times\R)\cong\left\{
\matrix{B&\\\transp{Y}&1}\ \Bigg|\ B\in\Aut(C^*),\ Y\in\R^3
\right\},
$$ 	
through which the homomorphism \eqref{eqn_homo} can be written as
\begin{equation}\label{eqn_prop_iso}
(A,X)\longmapsto\matrix{	\transp{\!A}^{-1}&\\[3pt]\transp{(A^{-1}X)}&1}=\transp{\!\matrix{A&-X\\[3pt]&1}}^{-1}.
\end{equation}
As a result, \eqref{eqn_homo} is an isomorphism, and form the following commutative diagram together with the isomorphism $\Aut(C)\overset\sim\to\Aut(C^*)=\Aut(\Omega)$, $A\mapsto \transp{\!A}^{-1}$ and the natural projections:
$$
\begin{tikzcd}
\Aut(C)\ltimes\R^3 \arrow[r, "\sim"] \arrow[d]
& \Aut(\Omega\times\R) \arrow[d] \\
\Aut(C) \arrow[r, "\sim"]
& \Aut(\Omega)
\end{tikzcd}
$$
\end{proposition}

\begin{proof}
Since $\Omega$ is the section of $-C^*$ by $\R^2\times\{-1\}$,
we can view $\Omega\times\R\cong\Omega\times\{-1\}\times\R$ as the section of the cone
$(C^*\cup(-C^*))\times\R$ in $\R^4$
by the affine plane $\R^2\times\{-1\}\times\R$ in $\R^4$. Thus, a projective transformation $\Phi\in\PGL(4,\R)$ preserves $\Omega\times\R$ if and only if the linear transformation $\widetilde{\Phi}\in\GL(4,\R)$ representing it (defined up to multiplication by scalar matrices) preserves this cone. In this case, the image $(x',\xi')=\Phi.(x,\xi)$  of $(x,\xi)\in\Omega\times \R$ is determined by the condition
\begin{equation}\label{eqn_prop_iso_proof}
\widetilde{\Phi}\matrix{x\\-1\\\xi}\parallelslant \matrix{x'\\-1\\\xi'},
\end{equation}
where ``$\parallelslant$'' denotes the colinear relation of vectors.
But it is elementary to check that $\widetilde{\Phi}$ preserves that cone if and only if it has the form 
$$
\widetilde{\Phi}=
\begin{pmatrix}
B&\\[5pt]
\transp{Y}&\lambda
\end{pmatrix},
\ \text{ with } B(C^*)=\pm C^*.
$$
Here $B\in\GL(3,\R)$, $Y\in\R^3$ and $\lambda\neq0$ are arbitrary. We then deduce from \eqref{eqn_prop_iso_proof} that
\begin{equation}\label{eqn_prop_iso_proof2}
\Phi(x,\xi)=\frac{1}{\left[-B\matrix{x\\-1}\right]_3}\left(\left[B\matrix{x\\-1}\right]_{1,2},\,\lambda\xi+\transp{Y}\matrix{x\\-1}\right).
\end{equation}

Each such $\Phi$ has a unique representative $\widetilde{\Phi}$ as above with $B\in\SL(3,\R)$, so we henceforth let $\widetilde{\Phi}$ only denote this representative. From the expression \eqref{eqn_prop_iso_proof2} of $\Phi.(x,\xi)$, we see that the projection $\pi(\Phi)\in\Aut^\pm(\Omega)$ is the projective transformation of $\Omega$ given by $B$ (\cf the last paragraph of \S \ref{subsec_omega}). So the three defining conditions for $\Phi$ to be an automorphisms of $\Omega\times\R$ are reflected in the components $\lambda$ and $B$ of $\widetilde{\Phi}$ as follows:
\begin{itemize}
	\item Condition \ref{item_auto1} is equivalent to $\lambda>0$.
	\item Condition \ref{item_auto2} is equivalent to $B\in\Aut(C^*)=\Aut(\Omega)$ (otherwise, we have $B(C^*)=-C^*$, and $B$ gives an orientation-reversing projective transformation of $\Omega$).
	\item Condition \ref{item_auto3} is equivalent to $\lambda=\pm1$.
\end{itemize}
Therefore, we conclude that the map
$$
\Aut(\Omega\times\R)\to\left\{
\matrix{B&\\\transp{Y}&1}\ \Bigg|\ B\in\Aut(C^*),\ Y\in\R^3
\right\},\quad \Phi\mapsto\widetilde{\Phi}
$$
is an isomorphism. Then, comparing \eqref{eqn_prop_iso_proof2} with the expression of the $\Aut(C)\ltimes\R^3$-action on $\Omega\times\R$ in Prop.\@ \ref{prop_action}, we see that the homomorphism \eqref{eqn_homo} resulting from the action has the required expression \eqref{eqn_prop_iso}. The proof is completed by the elementary fact that \eqref{eqn_prop_iso} does give an isomorphism and fit into the required diagram.
\end{proof}

The discussions till now are based on defining $\A^{3*}$ as the space of non-vertical affine planes in $\A^3$. But the construction of dual polarized affine spaces is involutive in the sense that we can equally identify $\A^3$ with the space of non-vertical affine planes in $\A^{3*}$, which are exactly the affine planes crossing the convex tube domain $\Omega\times\R\cong\hsp{C}\subset \A^{3*}$.  This identification has the following basic property:
\begin{lemma}\label{lemma_plane}
The identification $\A^3\cong\big\{\text{non-vertical affine planes in $\A^{3*}$}\}$ is equivariant with respect to the actions of $\Aut(C)\ltimes\R^3\cong\Aut(\Omega\times\R)$ in the sense that if $p_1,p_2\in\A^3$  correspond to the affine planes $P_1,P_2\subset\A^{3*}$, respectively, 
and $\Phi\in\Aut(\Omega\times\R)$ sends the section $P_1\cap(\Omega\times\R)$ to $P_2\cap(\Omega\times\R)$,
then the element of $\Aut(C)\ltimes\R^3$ corresponding to $\Phi$ sends $p_1$ to $p_2$.
\end{lemma}
Note that if we let $\overline{P}_i$ be the closure of $P_i$ in $\mathbb{RP}^3$ (\ie  the projective plane formed by $P_i$ and its line at infinity), then the assumption $\Phi\big(P_1\cap(\Omega\times\R)\big)=P_2\cap(\Omega\times\R)$ in the lemma is equivalent to $\Phi(\overline{P}_1)=\overline{P}_2$, but does not imply $\Phi(P_1)=P_2$.
\begin{proof}
Since $\A^{3*}$ is the space of non-vertical affine planes in $\A^3$ and  $\Omega\times\R\subset\A^{3*}$ is the subset of $C$-spacelike planes, we have
$$
P_i\cap(\Omega\times\R)=\big\{\text{$C$-spacelike planes in $\A^3$ passing through $p_i$}\big\}
$$
for $i=1,2$. By construction of the isomorphism $\Aut(C)\ltimes\R^3\cong\Aut(\Omega\times\R)$, the condition $\Phi\big(P_1\cap(\Omega\times\R)\big)=P_2\cap(\Omega\times\R)$ means that the affine transformation $(A,X)\in\Aut(C)\ltimes\R^3$ corresponding to $\Phi$ sends each $C$-spacelike plane passing through $p_1$ to one passing through $p_2$. It follows that $(A,X)$ sends $p_1$ to $p_2$.
\end{proof}

\subsection{$\Aut(\Omega)$ as a subgroup of $\Aut(\Omega\times\R)$}\label{subsec_subgroup}
Although the linear transformation group $\Aut(C)$ is a quotient of the affine transformation group $\Aut(C)\ltimes\R^3$ in a canonical way, the inclusion of the former into the latter depends on our \emph{ad hoc} choice of the identification $\A^3\cong\R^3$ in \S \ref{subsec_cspacelike}, or more precisely, choice of an origin point in $\A^3$, so that $\Aut(C)$ is the subgroup fixing the point. Different choices give rise to conjugate copies of $\Aut(C)$ in $\Aut(C)\ltimes\R^3$. 

The same thing can be said about $\Aut(\Omega)$ and $\Aut(\Omega\times\R)$ through the isomorphism in Prop.\@ \ref{prop_iso}. In fact, our coordinates on $\Omega\times\R$ are so-chosen that the image of $\Aut(C)< \Aut(C)\ltimes\R^3$ in $\Aut(\Omega\times\R)$ is the copy of $\Aut(\Omega)$ which consists of the automorphisms of $\Omega\times\R$ preserving the zero horizontal slice $\Omega\times\{0\}$. More explicitly, by the expression \eqref{eqn_prop_iso_proof2} of the $\Aut(\Omega\times\R)$-action on $\Omega\times\R$ given in the proof of Prop.\@ \ref{prop_iso}, the action of $A\in\Aut(\Omega)$ is
\begin{equation}\label{eqn_zeroslice}
A(x,\xi)=\frac{1}{\left[-A\matrix{x\\-1}\right]_3}\left(\left[A\matrix{x\\-1}\right]_{1,2},\,\xi\right).
\end{equation}
Observe that this action commutes with the involution $(x,\xi)\to(x,-\xi)$.

We will need the following lemma about images of vertically aligned points in $\Omega\times\R$ by the action of $\Aut(\Omega\times\R)$ and the subgroup $\Aut(\Omega)$:
\begin{lemma}\label{lemma_covariance}
Let $(x,\xi_1)$ and $(x,\xi_2)$ be points in $\Omega\times\R$ with the same projection $x\in\Omega$, and $\Phi\in\Aut(\Omega\times\R)$ be an automorphism of $\Omega\times\R$ projecting to $A:=\pi(\Phi)\in \Aut(\Omega)$. Suppose $\Phi(x,\xi_i)=(x',\xi_i')$, $i=1,2$, where $x'=Ax$. Then
	\begin{enumerate}[label=(\arabic*)]
		\item\label{item_covariance1} For any $s\in\R$, we have
		$$
		\Phi(x,(1-s)\xi_1+s\xi_2)=(x',(1-s)\xi_1'+s\xi_2').
		$$
		\item\label{item_covariance2}
	    Viewing $\Aut(\Omega)$ as the subgroup of $\Aut(\Omega\times\R)$ preserving $\Omega\times\{0\}$, we have
		$$
		A(x,\xi_1-\xi_2)=(x',\xi_1'-\xi_2').
		$$
		As a consequence, if $u_1,u_2:\Omega\to\R$ are functions such that the graph $\gra{u_1}\subset\Omega\times\R$ is preserved by $\Phi$, then $\gra{u_2}$ is also preserved by $\Phi$ if and only if $\gra{u_1-u_2}$ is preserved by $A$.
	\end{enumerate}
\end{lemma}
\begin{proof}
The map from the vertical line $\{x\}\times\R$ to $\{x'\}\times\R$ induced by $\Phi$ is an affine transformations because it is a projective transformation fixing the point at infinity, or alternatively, because of the expression
$$
\Phi(x,\xi)=\frac{1}{\left[-A\matrix{x\\-1}\right]_3}\left(\left[A\matrix{x\\-1}\right]_{1,2},\,\xi+\transp{Y}\matrix{x\\-1}\right)
$$
(see the proof of Prop.\@ \ref{prop_iso}). Part \ref{item_covariance1} follows as a consequence because for any affine transformation $f:\R\to\R$ we have $f\big((1-s)\xi_1+s\xi_2\big)=(1-s)f(\xi_1)+sf(\xi_2)$. Part \ref{item_covariance2} can be verified directly by comparing this expression of $\Phi(x,\xi)$ and that of $A(x,\xi)$ given above in \eqref{eqn_zeroslice}.
\end{proof}

\section{$C$-regular domains and affine $(C,k)$-surfaces}\label{sec:cregular}
In the section, we first give some background materials on convex analysis, then we review the theory of $C$-regular domains, $C$-convex surfaces and affine $(C,k)$-surfaces developed in \cite{nie-seppi}, and explain the correspondences in the last three rows of Table \ref{table_dic}.
\subsection{Convex functions}\label{subsec_convexfunction}
In this paper, a lower semicontinuous function is assumed to take values in $\R\cup\{+\infty\}$ if not otherwise specified. Let $\LC(\R^2)$ denote the space of lower semicontinuous, convex functions on $\R^2$ that are not constantly $+\infty$. Given $u\in\LC(\R^2)$, if the \emph{effective domain}
$$
\dom{u}:=\{x\mid u(x)<+\infty\}
$$ 
of $u$ has nonempty interior $U:=\interior\dom{u}$, then the values of $u$ on $\pa U$ (hence the values on the whole $\R^2$) are determined by the restriction $u|_U$, because given any $x_0\in\pa U$, it can be shown that
\begin{equation}\label{eqn_boundary value}
u(x_0)=\liminf_{U\ni x\to x_0}u(x)=\lim_{s\to0^+}u((1-s)x_0+sx_1)\in(-\infty,+\infty]
\end{equation}
for any $x_1\in U$ (see \cite[\S 4.1]{nie-seppi}). 

Therefore, given a convex domain $U\subset \R^2$ and a convex function $u:U\to\R$, we define the \emph{boundary value} of $u$ as the function on $\pa U$ whose value at $x_0\in\pa U$ is the liminf or limit in \eqref{eqn_boundary value}, which are equal and independent of $x_1\in U$. 
We slightly abuse the notation for restrictions and denote this function by $u|_{\pa U}$. By \cite[Prop.\@ 4.1]{nie-seppi}, the extension of $u$ to $\R^2$ given by $u|_{\pa U}$ and by setting the value to be $+\infty$ outside of $\overline{U}$ is an element of $\LC(\R^2)$. This gives a canonical way of viewing every convex function on a convex domain as an element of $\LC(\R^2)$, which also explains the notation $u|_{\pa U}$.

In this setting, given $x_0\in\pa U$, we say that $u$ has \emph{infinite inner derivatives} at $x_0$ if either $u(x_0)=+\infty$ or $u(x_0)$ is finite but
\begin{equation}\label{eqn_inner}
\lim_{s\to0^+}\frac{u(x_0+s(x_1-x_0))-u(x_0)}{s}=-\infty
\end{equation}
for any $x_1\in U$. Note that the fraction is an increasing function in $s\in(0,1]$ by convexity of $u$, hence the limit exists in $[-\infty,+\infty)$. We refer to \cite[\S 4]{nie-seppi} for more properties of this definition, especially the fact that if \eqref{eqn_inner} holds for one $x_1\in U$, then it holds for all $x_1\in U$.


Now fix a \emph{bounded} convex domain $\Omega\subset\R^2$. For any function $\psi:\pa\Omega\to\R\cup\{+\infty\}$ that is bounded from below and is not constantly $+\infty$, we define the \emph{convex envelope} $\env{\psi}$ of $\psi$ as the function on $\R^2$ given by
$$
\env{\psi}(x):=\sup\{a(x)\mid \text{ $a:\R^2\to\R$ is an affine function with $a|_{\pa\Omega}\leq \psi$}\}.
$$
The reason for the assumptions on $\psi$ in the definition is that otherwise it would yield constant functions $-\infty$ or $+\infty$, which are not interesting.

The convex envelope $\env{\psi}$ has the following fundamental properties.
\begin{itemize}
	\item $\env{\psi}$ belongs to $\LC(\R^2)$.
	\item $\env{\psi}|_{\pa\Omega}\leq\psi$ on $\pa\Omega$, and the equality holds everywhere if and only if $\psi$ is lower semicontinuous and restricts to a convex function on any line segment in $\pa\Omega$ (see \cite[Lemma 4.6]{nie-seppi}).
	\item $\dom{\env{\psi}}$ is the convex hull of $\{x\in\pa\Omega\mid \env{\psi}(x)<+\infty\}$ in $\R^2$ (see \cite[Prop.\@ 4.8]{nie-seppi}). In particular, if $\psi$ only takes values in $\R$, then $\dom{\env{\psi}}=\overline{\Omega}$.
	\item $\env{\psi}$ is pointwise no less than any function in $\LC(\R^2)$ majorized by $\psi$ on $\pa\Omega$ (see \cite[Cor.\@ 4.5]{nie-seppi}). In particular, for any convex $u:\Omega\to\R$ with boundary value $u|_{\pa\Omega}\leq\psi$, we have $u\leq \env{\psi}$ in $\Omega$.
\end{itemize}

By the first two properties, the subset of $\LC(\R^2)$ formed by all convex envelopes can be understood as follows. Denote
$$
\LC(\pa\Omega):=
\left\{u:\pa\Omega\to\R\cup\{+\infty\}\ \Bigg|\ \parbox[l]{8cm}{$u$ is lower semicontinuous and not constantly $+\infty$; the restriction of $u$ to any line segment in $\pa\Omega$ is convex}\right\}~.
$$ 
Then the assignment $\phi\mapsto\env{\phi}$ is a bijection from $\LC(\pa\Omega)$ to the subset of $\LC(\R^2)$ consisting of all functions of the form $\env{\psi}$, and its inverse is just the restriction map $u\mapsto u|_{\pa\Omega}$.

We will freely use the fact that the constructions of boundary values and convex envelopes are covariant with respect to automorphisms of the convex tube domain $\Omega\times\R$ (see \S \ref{subsec_auto}) in the sense that
	\begin{itemize}
		\item 
		if $u_1,u_2:\Omega\to\R$ are convex functions such that $\Phi\in\Aut(\Omega\times\R)$ brings $\gra{u_1}$ to $\gra{u_2}$, then the action of $\Phi$ on $\pa\Omega\times\R$ brings $\gra{u_1|_{\pa\Omega}}$ to $\gra{u_2|_{\pa\Omega}}$;
		\item 
		if $\psi_1,\psi_2:\pa\Omega\to\R\cup\{+\infty\}$ are bounded from below and not constantly $+\infty$, such that the action of $\Phi\in\Aut(\Omega\times\R)$ on $\pa\Omega\times\R$ brings $\gra{\psi_1}$ to $\gra{\psi_2}$, then $\Phi$ also brings $\gra{\env{\psi}_1}$ to $\gra{\env{\psi}_2}$.
	\end{itemize} 
The first bullet point can be seen from definition \eqref{eqn_boundary value} of boundary values. The second follows from the definition of convex envelopes and the fact that automorphisms of $\Omega\times\R$ send graphs of affine functions to graphs of affine functions.

\subsection{Legendre transform}\label{subsec_legendre}
The \emph{Legendre transform} of $u\in\LC(\R^2)$ is by definition the function $u^*\in\LC(\R^2)$ given by
$$
u^*(y):=\sup_{x\in\R^2}(x\cdot y-u(x)).
$$
It is a fundamental fact that the Legendre transformation $u\mapsto u^*$ is an involution on $\LC(\R^2)$ (see \cite[\S 4.5]{nie-seppi}). 
If $u$ is an $\R$-valued convex function only defined on a convex domain $U\subset\R^2$, we define its Legendre transform $u^*$ by viewing $u$ as an element of $\LC(\R^2)$ via the canonical extension described in 
\S \ref{subsec_convexfunction}. Note that we have $u_1\leq u_2$ if and only if $u_1^*\geq u_2^*$.

We now give a geometric interpretation of $u^*$ using the notion of dual polarized affine space from the introduction. As in \S \ref{subsec_cspacelike}, we first identify $\A^3$ with $\R^3=\R^2\times\R$ and endow it with the polarization given by the point at infinity of vertical lines $\{x\}\times\R$, so that the dual affine space $\A^{3*}$ is the set of non-vertical affine planes in $\A^3$, which can be identified with $\R^3$ as well through the map \eqref{eqn_identification_r3}. We then view the graphs or epigraphs of $u$ and $u^*$ as subsets of $\A^3$ and $\A^{3*}$, respectively, through the two identifications. Under this setup, we have: 
%
\begin{proposition}\label{prop_legendre}
Let $u\in\LC(\R^2)$ be such that $\dom{u}$ is bounded. Then $\dom{u^*}=\R^2$.
In this case, given $(y,\eta)\in\R^3$, we have $u^*(y)=\eta$ if and only if the graph of the affine function $x\mapsto x\cdot y-\eta$ is a supporting plane of the graph $\gra{u}\subset\R^3$ of $u$. 

As a consequence, if we view the graph and epigraph of $u$ (resp.\@ $u^*$) as subsets of $\A^3$ (resp.\@ $\A^{3*}$) in the aforementioned way, then $\gra{u^*}$ is exactly the set of non-vertical supporting planes of $\gra{u}$. Similarly, $\sepi{u^*}\subset\A^{3*}$ is the set of non-vertical affine planes in $\A^3$ disjoint from $\epi{u}$. 
\end{proposition}
Here and below, by a \emph{supporting plane} of a set $E\subset\R^3$, we mean an affine plane $P\subset\R^3$ such that $\overline{E}\cap P\neq\emptyset$ and $E$ is contained in one of the two closed half-spaces of $\R^3$ with boundary $P$. For any extended-real-valued function $f$ on $\R^2$, we let
$$
\epi{f}:=\big\{(x,\xi)\in\R^3\,\big|\, f(x)\leq\xi\big\},\ \ \sepi{f}:=\big\{(x,\xi)\in\R^3\,\big|\, f(x)<\xi\big\}=\epi{f}\setminus\gra{f}
$$
denote the epigraph and strict epigraph of $f$, respectively.
\begin{proof}
The boundedness of $\dom{u}$ implies that the supremum in the definition of $u^*(y)$ is actually a maximum. More precisely,
given $y\in \R^2$, since the upper semicontinuous function $x\mapsto x\cdot y-u(x)$ is $-\infty$ outside of the bounded set $\dom{u}$, it attains its maximum at some $x_0\in\dom{u}$, which means exactly that $u^*(y)=x_0\cdot y-u(x_0)$. In particular, we have $u^*(y)<+\infty$ for any $y\in\R^2$, hence $\dom{u^*}=\R^2$.

As a consequence, the condition  ``$(y,\eta)\in\gra{u^*}$'' is equivalent to ``$\eta=x_0\cdot y-u(x_0)$, where $x_0$ is a maximal point of $x\mapsto x\cdot y-u(x)$''. The latter condition can be written alternatively as ``$u(x)\geq x\cdot y-\eta$ for all $x\in\R^2$, with equality achieved at $x_0$'', which means exactly that $\gra{x\mapsto x\cdot y-\eta}$ is a supporting plane of $\gra{u}$ at $x_0$. It follows that $\gra{u^*}$ is the set of non-vertical supporting planes of $\epi{u}$, because the identification $\R^3\cong\A^{3*}$ is so defined that $(y,\eta)\in\R^3$ corresponds to the non-vertical affine plane $\gra{x\mapsto x\cdot y-\eta}$. The last statement is similar.
\end{proof}

\begin{example}\label{example_legendre}
Figure \ref{figure_examples} shows the graphs of some radially symmetric convex functions $u$ on the unit disk $\mathbb{D}:=\{x\in\R^2\mid |x|<1\}$ and their Legendre transforms $u^*$. Each example of $u$ is pointwise smaller than the next one. In the first and the last examples, where $u=0$ and $u(x)=\sqrt{1-|x|^2}$, respectively, the graph of $u^*$ is the boundary of the future light cone $C_0:=\big\{(x,\xi)\in\R^3\mid |x|<\xi\big\}$ and the hyperboloid $\mathbb{H}^2:=\big\{(x,\xi)\in\R^3\mid |x|^2+1=\xi^2\big\}$ in $C_0$, respectively.
\begin{figure}[h]
	\includegraphics[width=12.5cm]{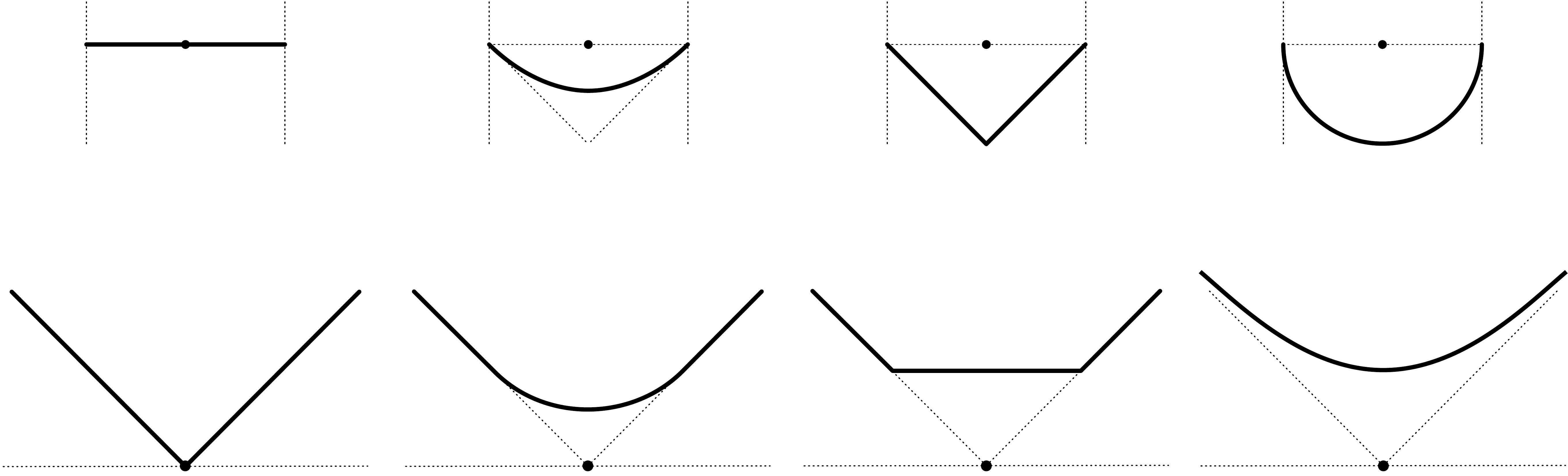}
	\caption{Section view of graphs of some radially symmetric convex functions on the unit disk (first row) and their Legendre transforms (second row). The black dots represent the origin of $\R^3$.}
	\label{figure_examples}
\end{figure}
\end{example}

The following lemma relates the Legendre transform of a convex $\C^1$-function with the gradient, and give useful interpretations of the gradient-blowup property:
\begin{lemma}
	\label{lemma_gradient}
Let $U\subset\R^2$ be a convex domain and $u:U\to\R$ be a strictly convex $\C^1$-function. Then 
\begin{enumerate}[label=(\arabic*)]
	\item\label{item_gradient1} The gradient map $x\mapsto\D u(x)$ is a homeomorphism from $U$ to the convex domain
	$$
	\D u(U):=\big\{\D u(x)| x\in U\big\}\subset\R^2
	$$    
	\item\label{item_gradient2} The restriction of the Legendre transform $u^*\in\LC(\R^2)$ to $\D u(U)$ is given by
    $$
    u^*(\D u (x))=x\cdot \D u(x)-u(x),\ \ \forall x\in U.
    $$
    \item\label{item_gradient3} The following conditions are equivalent to each other:
    \begin{itemize}
    	\item $u$ has infinite inner derivative at every point of $\pa U$ (see \S \ref{subsec_convexfunction});
    	\item $\|\D u(x)\|$ tends to $+\infty$ as $x\in U$ tends to $\pa U$.
    \end{itemize}
Furthermore, if $U$ is bounded, both conditions are equivalent to $\D u(U)=\R^2$.
\end{enumerate}
\end{lemma}
See Example \ref{example_legendre} again for illustrations of these results. Among the four instances of $u$ in Figure \ref{figure_examples}, the 2nd and the 4th are smooth and strictly convex in $\mathbb{D}$, and only the 4th has infinite inner derivatives at boundary points. In the 2nd one, the graph of $u^*$ over $\D u(\mathbb{D})$ is the exactly the part of $\gra{u^*}$ in $C_0$.

\begin{proof}
\ref{item_gradient1} Since $x\mapsto\D u(x)$ is continuous in $U$, by Brouwer Invariance of Domain, it suffices to show the injectivity. Suppose by contradiction that $\D u(x_1)=\D u(x_2)$ for $x_1,x_2\in U$. Then the graph of the convex function $x\mapsto u(x)-x\cdot \D u(x_1)$ has horizontal supporting planes at both $x_1$ and $x_2$, which implies that the function is a constant on the line segment joining $x_1$ and $x_2$. It follows that $u$ is an affine function on that segment, contradicting the strict convexity. 

\ref{item_gradient2} Since $u$ is $\C^1$ at every $x_0\in U$, the only supporting plane of the graph $\gra{u}$ at $(x_0,u(x_0))$ is the tangent plane, which is the graph of the affine function 
$$
x\mapsto u(x_0)+(x-x_0)\cdot \D u(x_0)=x\cdot \D u(x_0)-\big(x_0\cdot \D u(x_0)-u(x_0)\big).
$$ 
It follows from Prop.\@ \ref{prop_legendre} that $u^*(\D u(x_0))=x_0\cdot \D u(x_0)-u(x_0)$, as required.

\ref{item_gradient3} We view $u$ as an element of $\LC(\R^2)$ by extending it to the whole $\R^2$ in the way described in \S \ref{subsec_convexfunction}. By \cite[Prop.\@ 4.13]{nie-seppi}, for any $x_0\in\pa U$, the following conditions are equivalent to each over (Condition \ref{item_gradient33} below is formulated as (iii) in \cite[Prop.\@ 4.13]{nie-seppi} using the notion of \emph{subgradients}):
\begin{enumerate}[label=(\alph*)]
	\item\label{item_gradient31} $u$ has finite inner derivative at $x_0$;
	\item\label{item_gradient32} there is a sequence of points $(x_i)_{i=1,2,\cdots}$ in $U$ tending to $x_0$ such that $\|\D u(x_i)\|$ does not tend to $+\infty$;
	\item\label{item_gradient33} the graph $\gra{u}$ has a non-vertical supporting plane at $(x_0,u(x_0))$.
\end{enumerate}

The first (resp.\@ second) bullet point in the required statement means exactly that there does not exist $x_0\in\pa U$ satisfying \ref{item_gradient31} (resp.\@ \ref{item_gradient32}), so the two bullet points are equivalent to each other. 

To complete the proof, we now assume $U$ is bounded and only need to show that $\D u(U)\neq\R^2$ if and only if $\gra{u}$ has a non-vertical supporting plane at $(x_0,u(x_0))$ for some $x_0\in\pa U$.

Suppose $\gra{u}$ has a non-vertical supporting plane $P$ at $(x_0,u(x_0))$. We have $P=\gra{x\mapsto x\cdot y-\eta}$ for some $y\in\R^2$, $\eta\in\R$. With the same reasoning as in the proof of Part \ref{item_gradient1}, we see that $\D u(x_1)\neq y$ for any $x_1\in U$, otherwise $u$ would be an affine function on the line segment joining $x_0$ and $x_1$, contradicting the strict convexity. This shows the ``if'' part.

Conversely, suppose $\D u(U)\neq \R^2$ and pick $y\in \R^2\setminus \D u(U)$. By Prop.\@ \ref{prop_legendre}, we have $\eta:=u^*(y)<+\infty$, hence $P:=\gra{x\mapsto x\cdot y-\eta}$ is a supporting plane of $\gra{u}$. But $P$ cannot be the supporting plane of $\gra{u}$ at any point in $U$ because $y\notin\D u(U)$, hence must be a supporting plane at some $x_0\in\pa U$. This shows the ``only if'' part and completes the proof.
\end{proof}

Now fix a proper convex cone $C\subset\R^3$ and let $\Omega\subset\R^2$ be the corresponding convex domain as in \S \ref{sec:correspondence}, so that there is an isomorphism $\Aut(C)\ltimes\R^3\cong\Aut(\Omega\times\R)$ (see Prop.\@ \ref{prop_iso}).
We henceforth only consider convex functions $u\in\LC(\R^2)$ with $\dom{u}\subset\overline{\Omega}$. We shall view the graph of $u$ over $\Omega$ and the graph $\gra{u^*}$ of the Legendre transform of $u$ as geometric objects in the convex tube domain $\Omega\times\R\subset\A^{3*}$ and the affine space $\A^3$, respectively, which are acted upon by $\Aut(\Omega\times\R)$ and $\Aut(C)\ltimes\R^3$, respectively (comparing to Prop. \ref{prop_legendre}, the roles of $\A^3$ and $\A^{3*}$ are switched here). 

It follows from Prop.\@ \ref{prop_legendre} that the two actions are actually intertwined. We now give a formal statement only for convex $\C^1$-functions:
\begin{lemma}\label{lemma_actiongraph}
	Let $C$ and $\Omega$ be as in \S \ref{subsec_cspacelike}, $\Phi$ be an element of $\Aut(\Omega\times\R)$, and $u_1,u_2\in\C^1(\Omega)$ be convex functions such that $\Phi$ brings the graph $\gra{u_1}$ to $\gra{u_2}$. Let $\Sigma_i$  be the graph over $\D u_i(\Omega)$ of the Legendre transform $u_i^*$ ($i=1,2$), \ie 
	$$
	\Sigma_i:
	=\gra{u_i^*|_{\D u_i(\Omega)}}
	=\big\{\big(\D u_i(x),\,x\cdot \D u_i(x)-u_i(x)\big)\,\big|\, x\in U_i\big\}\subset \R^3\cong\A^{3}
	$$
(see Lemma \ref{lemma_gradient} \ref{item_gradient2}). 
	Then $\Sigma_2$ is the image of $\Sigma_1$ by the element of $\Aut(C)\ltimes\R^3$ which corresponds to $\Phi$ under the isomorphism $\Aut(C)\ltimes\R^3\cong\Aut(\Omega\times\R)$.
\end{lemma}
Note that by Lemma \ref{lemma_gradient} \ref{item_gradient3}, $\Sigma_i$ is an entire graph (\ie $\D u_i(\Omega)=\R^2$) only when $u_i$ has the gradient-blowup property. 
\begin{proof}
If we identify $\A^3$ as at the set of non-vertical affine planes in $\A^{3*}$ (see  the last part of \S \ref{subsec_iso}), then  by Prop.\@ \ref{prop_legendre}, $\Sigma_i\subset\A^3$  consists exactly of the tangent planes to the graph $\gra{u_i}\subset\Omega\times\R\subset\A^{3*}$. So the required statement follows from Lemma \ref{lemma_plane}. 
\end{proof}

\subsection{$C$-regular domains and $C$-convex surfaces}\label{subsec_cregular}
In view of the notions of $C$-null and $C$-spacelike planes introduced in \S \ref{subsec_cspacelike}, we call a closed half-space $H\subset\A^3$ $C$-null (resp.\@ $C$-spacelike) if $H$ contains a translation of the cone $C$ and the boundary plane $\pa H$ is  $C$-null (resp.\@ $C$-spacelike). In other words, a $C$-null/spacelike half-space is the upper half of $\A^3$ cut out by a $C$-null/spacelike plane. We further define:
\begin{itemize}
	\item A \emph{$C$-regular domain} is a nonempty, open, proper subset $D\subsetneq\A^3$ which is the interior of the intersection of a collection $\mathcal{S}$ of $C$-null half-spaces
	\footnote{The definitions in \cite{nie-seppi} and this paper have two inessential differences: first, the empty set and the whole $\A^3$ are both considered as $C$-regular domains in \cite{nie-seppi}, but not here; second, in \cite{nie-seppi}, the function $\phi\equiv+\infty$ is included as an element of $\LC(\pa\Omega)$ (which corresponds to $D=\A^3$ in the sense of Thm.\@ \ref{thm_graph} \ref{item_thmgraph1} below), but not here.}.
	If $\mathcal{S}$ consists of all the $C$-null half-spaces containing a set $E\subset \A^3$, we call $D$ the $C$-regular domain \emph{generated} by $E$.
	\item A \emph{$C$-convex surface} is an open subset $\Sigma$ of the boundary of some convex domain $U\subset \R^3$ such that the supporting half-space of $U$ at every point of $\Sigma$ is  $C$-spacelike. $\Sigma$ is said to be \emph{complete} if it is properly embedded, or equivalent, it is the entire boundary of $U$.
\end{itemize}
Here, by a \emph{supporting half-space} of a convex domain $U\subset\R^3$ at a boundary point $p\in\pa U$, we mean a closed half-space $H\subset\R^3$ containing $U$ such that $p\in\pa H$. 
\begin{remark}
	For the future light cone $C_0$ of the Minkowski space $\R^{2,1}$, $C_0$-regular domains are classically known as \emph{regular domains} or \emph{domains of dependence}, whereas $C_0$-convex surfaces are just \emph{future-convex}, spacelike surfaces. 
See \cite[\S 3.1]{nie-seppi} for more discussions about these definitions and the backgrounds.
\end{remark}

In order to identify the space of $C$-convex surfaces, we introduced the following subset of $\LC(\R^2)$ (\cf \S \ref{subsec_convexfunction}): let $\SC_0(\Omega)$ denote the set of all $u\in\LC(\R^2)$ satisfying
\begin{itemize}
	\item[-] $U:=\interior\dom{u}$ is nonempty and contained in $\Omega$ (see \S \ref{subsec_convexfunction} for the notation);
	\item[-] $u$ is smooth and locally strongly convex in $U$;
	\item[-] the norm of gradient $\|\D u(x)\|$ tends to $+\infty$ as $x\in U$ tends to $\pa U$. 
\end{itemize}
Here, a $C^2$-function is said to be locally \emph{strongly} convex if its Hessian is positive definite (which is stronger than being locally \emph{strictly} convex). Although a general $u\in\SC_0(\Omega)$ may take the value $+\infty$ in $\Omega$, in applications later on, we mainly consider those $u$'s such that $u|_{\pa\Omega}$ is $\R$-valued. In this case, $u$ is $\R$-valued in $\Omega$ as well and can be viewed as an element of $\C^\infty(\Omega)$. 

Using the notations introduced in the previous subsections, we can formulate the correspondences in the 2nd-to-last and 3rd-to-last rows of Table \ref{table_dic} as: 
\begin{theorem}[{\cite[Thm.\@ 5.2 and 5.13]{nie-seppi}}]\label{thm_graph}
	Let $C$ and $\Omega$ be as in \S \ref{subsec_cspacelike}. Then the following statements hold.
	\begin{enumerate}[label=(\arabic*)]
		\item\label{item_thmgraph1} The assignment 
		$\phi\mapsto D=\sepi{\env{\phi}^*}$
		gives a bijection from $\LC(\pa\Omega)$ to the space of all $C$-regular domains in $\A^3$. Moreover, $D$ is proper (see \S \ref{subsec_cspacelike}) if and only if the convex hull of $\dom{\phi}:=\{x\in\pa\Omega\mid \phi(x)<+\infty\}$ in $\R^2$ has nonempty interior.
		\item\label{item_thmgraph2} 
		The assignment $u\mapsto\Sigma=\gra{u^*}$ gives a bijection from $\SC_0(\Omega)$ to the space of all smooth, strongly convex, complete, $C$-convex surfaces in $\A^3$.
		\item\label{item_thmgraph3} 
		Let $u\in\SC_0(\Omega)$ and suppose $\phi:=u|_{\pa\Omega}$ is not constantly $+\infty$ (hence belongs to $\LC(\pa\Omega)$). Then the $C$-regular domain $D=\sepi{\env{\phi}^*}$ is generated by the $C$-convex surface $\Sigma=\gra{u^*}$. 
		\item\label{item_thmgraph4} Suppose $u\in\SC_0(\Omega)$ and $\phi\in\LC(\pa\Omega)$. Then the $C$-convex surface $\Sigma=\gra{u^*}$ is asymptotic to the boundary of the $C$-regular domain $D=\sepi{\env{\phi}^*}$ if and only if the following conditions are satisfied:
		\begin{itemize}
			\item $\dom{u}$ coincides with $\dom{\env{\phi}}$ (it is shown in \S \ref{subsec_convexfunction} that the latter set equals the convex hull of $\dom{\phi}\subset\pa\Omega$ in $\R^2$);
            \item $\env{\phi}(x)-u(x)$ tends to $0$ as $x\in\interior \dom{u}$ tends to the boundary of $\dom{u}$.			
		\end{itemize}
	 In particular, we have $\phi=u|_{\pa\Omega}$ in this case.
	\end{enumerate}
\end{theorem}
In the last part, ``$\Sigma$ is asymptotic to $\pa D$'' means the distance from $p\in\Sigma$ to $\pa D$ tends to $0$ as $p$ goes to infinity in $\Sigma$. The last two parts of the theorem actually imply that this condition is in general strictly stronger than ``$\Sigma$ generates $D$''\footnote{If $\phi\in\C^0(\pa\Omega)$, the two conditions are equivalent (see \cite[p.30]{nie-seppi}). In contrast, an example of a complete $C$-convex surface generating a $C$-regular domain but not asymptotic to the boundary of the domain is given in \cite[Example 5.14]{nie-seppi}, where we have $\phi=+\infty$ on a part of $\pa\Omega$. There also exist examples where $\phi$ is $\R$-valued and bounded.}. 

\begin{example}
The first and the last examples in Example \ref{example_legendre} are the simplest cases of Parts \ref{item_thmgraph1} and \ref{item_thmgraph2} of the above theorem, respectively, which yield the light cone $C_0$ itself as a $C_0$-regular domain and the hyperboloid $\mathbb{H}^2$ as a complete $C_0$-convex surface. In contrast, in the 2nd and 3rd examples, the convex function $u$ is neither a convex envelope nor contained in $\SC_0(\Omega)$, hence $\gra{u^*}$ is neither the boundary of a $C_0$-regular domain nor a complete $C_0$-convex surface. Nevertheless, a part of $\gra{u^*}$ is still an incomplete $C_0$-convex surface in both examples. See Remark \ref{remark_piece} below.
\end{example}

\begin{remark}\label{remark_piece}
	By Lemma \ref{lemma_gradient} \ref{item_gradient2}, the $C$-convex surface in Part \ref{item_thmgraph2} of the theorem can be written as
	$$
	\gra{u^*}=\big\{\big(\D u(x),\,x\cdot \D u(x)-u(x)\big)\,\big|\, x\in \interior\dom{u}\big\},
	$$
	and the completeness of the surface corresponds to the gradient blowup condition in the definition of $\SC_0(\Omega)$ via Lemma \ref{lemma_gradient} \ref{item_gradient3}. Putting aside the completeness condition, one can also show that for any convex $\C^1$-function $u$ on an open set $U\subset\Omega$, the surface $\big\{\big(\D u(x),\,x\cdot \D u(x)-u(x)\big)\,\big|\, x\in U\big\}$, \ie the graph of $u^*$ over $\D u(U)$, is a $C$-convex surface, although the whole graph $\gra{u^*}$ might be not. 
\end{remark}

\begin{remark}\label{remark_sc}
In \cite{nie-seppi}, we also considered a set of functions $\SC(\Omega)\subset\LC(\R^2)$  containing $\SC_0(\Omega)$, whose definition does not assume any differentiability or strict convexity, and showed in \cite[Thm.\@ 5.2]{nie-seppi} that $\SC(\Omega)$ identifies with the space of all complete $C$-convex surfaces (not necessarily smooth or strictly convex) in $\A^3$. This generalizes Thm.\@ \ref{thm_graph} \ref{item_thmgraph2} above. See also \cite[Example 5.14]{nie-seppi} for an example of a function in $\SC(\Omega)\setminus\SC_0(\Omega)$ and the corresponding $C$-convex surface, which is not strictly convex.
\end{remark}

The geometric interpretation of Legendre transformation in Prop.\@ \ref{prop_legendre} provides the following more detailed descriptions of the bijections \ref{item_thmgraph1} and \ref{item_thmgraph2} (the roles of $\A^3$ and $\A^{3*}$ in Prop.\@ \ref{prop_legendre} are switched here in order to be consistent with our convention that $C$-regular domains live in $\A^3$ and convex tube domains in $\A^{3*}$, so we are actually viewing $\A^3$ as the space of non-vertical affine planes in $\A^{3*}$):
\begin{itemize}
	\item Given $\phi\in\LC(\pa\Omega)$, points in the $C$-regular domain $D=\sepi{\env{\phi}^*}\subset\A^3$ are 
	exactly the non-vertical affine planes in $\A^{3*}$ disjoint from the epigraph $\epi{\phi}\subset\pa\Omega\times\R\subset\A^{3*}$, or in other words, the affine planes crossing $\Omega\times\R$ from below of $\gra{\phi}$. 
	\item Given $u\in\SC_0(\Omega)$, points on the $C$-convex surface $\Sigma=\gra{u^*}$ are exactly the non-vertical supporting planes of the surface $\gra{u}\subset\overline{\Omega}\times\R$.
\end{itemize}

\subsection{Hyperbolic affine spheres}
The theory of \emph{Affine Differential Geometry} studies properties of surfaces in the affine space $\A^3$ that are invariant under special affine transformations. A crucial ingredient in the theory is the fact that any smooth locally strongly convex surface $\Sigma\subset \A^3$ carries a canonical transversal vector field $N_\Sigma$, called the \emph{affine normal field}, pointing towards the convex side of $\Sigma$. Using $N_\Sigma$, one can define the \emph{affine shape operator} $S_\Sigma$ of $\Sigma$, which is a smooth section of $\End(\T\Sigma)$, in the same way as in classical surface theory, and then define the \emph{affine Gaussian curvature} $\kappa_\Sigma:\Sigma\to\R$ to be the determinant $\det(S_\Sigma)$ (see \eg \cite{MR1311248} for details).

We are interested in certain $\Sigma$'s with Constant Affine Gaussian Curvature (CAGC). A simple yet crucial sub-class of them are \emph{hyperbolic affine spheres}, which are by definition those $\Sigma$ with $S_\Sigma=\id$. This condition is equivalent to the existence of a \emph{center} $o\in\A^3$ of $\Sigma$ such that at any $p\in\Sigma$, the affine normal $N_\Sigma(p)$ equals the vector $\overrightarrow{op}$. 

The above discussion is local in nature. We henceforth restrict ourselves to ``global'' hyperbolic affine spheres, \ie properly embedded ones, of which the simplest example is the hyperboloid $\mathbb{H}^2$ in the light cone $C_0\subset \R^{2,1}$. These affine spheres are classified by a theorem of Cheng and Yau \cite{chengyau1}, stating that they are in $1$-to-$1$ correspondence with proper convex cones in $\R^3$, just in the way how $\mathbb{H}^2$ corresponds to $C_0$. Namely, each cone $C$ contains a unique properly embedded hyperbolic affine sphere, denoted by $\Sigma_C$, which is asymptotic to $\pa C$, and conversely every properly embedded affine sphere centered at $0$ is some $\Sigma_C$ for a unique $C$. 

In the proof of Cheng-Yau's theorem, one encodes a convex surface $\Sigma$ by a function on a bounded convex domain, and translates the geometrical condition that $\Sigma$ is an affine sphere to a PDE on that function. Using the framework in \S \ref{subsec_cregular}, we can formulate this translation process and the statement of the theorem itself as:
\begin{theorem}[\cite{chengyau1}]\label{thm_chengyau}
For any bounded convex domain $\Omega\subset\R^2$, there exists a unique convex function $w_\Omega\in\C^0(\overline{\Omega})\cap\C^\infty(\Omega)$ solving the Dirichlet problem of Monge-Amp\`ere equation
$$
\begin{cases}
\det\D^2w=w^{-4},\\
w|_{\pa\Omega}=0.
\end{cases}
$$
Furthermore, $w_\Omega$ has the following properties:
\begin{itemize}
	\item the norm of gradient $\|\D w_\Omega(x)\|$ tends to $+\infty$ as $x\in\Omega$ tends to $\pa\Omega$;
	\item the graph $\gra{w_\Omega}\subset\Omega\times\R$ is invariant under the action of $\Aut(\Omega)$, viewed as the subgroup of $\Aut(\Omega\times\R)$ preserving the slice $\Omega\times\{0\}$ (see \S \ref{subsec_subgroup}).
\end{itemize}
Moreover, if $C$ and $\Omega$ are as in \S \ref{subsec_cspacelike}, then the graph $\Sigma_C=\gra{w_\Omega^*}$ of the Legendre transform $w_\Omega^*$ is the unique hyperbolic affine sphere asymptotic to $\pa C$. 
\end{theorem}
The second bullet point is not included in \cite{chengyau1} but is equivalent to the fact that the affine sphere $\Sigma_C=\gra{w_\Omega^*}$ is invariant under automorphisms of $C$, which is in turn a consequence of the uniqueness of $\Sigma_C$, or equivalently, the uniqueness of $w_\Omega$. See Lemma \ref{lemma_actiongraph} for a precise explanation of these equivalences.

We refer to $w_\Omega$ as the \emph{Cheng-Yau support function} of $\Omega$. The simplest example is already given in Example \ref{example_legendre}: we have $w_{\mathbb{D}}(x)=\sqrt{1-|x|^2}$ for the unique disk $\mathbb{D}$, and the corresponding affine sphere is just $\Sigma_{C_0}=\mathbb{H}^2$.
\subsection{Affine $(C,k)$-surfaces}\label{subsec_affineck}
We proceed to consider a wider sub-class of smooth strongly convex CAGC surfaces. By \cite[\S 3.2]{nie-seppi}, a sufficient (but not necessary) condition for $\kappa_\Sigma$ to be a positive constant $k$ is that the surface $N_\Sigma(\Sigma)$ in the vector space $\R^3$ (\ie the surface formed by all affine normal vectors of $\Sigma$) is contained in a scaled affine sphere of the form $k^{\frac{1}{4}}\Sigma'$, where $\Sigma'$ is some hyperbolic affine sphere centered at $0$. If in particular $\Sigma'$ is the Cheng-Yau affine sphere $\Sigma_C$ for a proper convex cone $C$, we call $\Sigma$ an \emph{affine $(C,k)$-surface}. 
\begin{remark}
While most of the properties studied in Affine Differential Geometry, such as the property of having CAGC, are invariant under special affine transformations, the property of being an affine $(C,k)$-surface is only invariant under $\Aut(C)\ltimes\R^3$. In fact, a general special affine transformation $(A,X)\in\SL(3,\R)\ltimes\R^3$ sends an affine $(C,k)$-surface to an affine $(A(C),k)$-surface.
\end{remark}

It is easy to see that an affine $(C,k)$-surface is $C$-convex in the sense of \S \ref{subsec_cregular} (see \cite[\S 3.2]{nie-seppi} for details). 
Therefore, by Theorem \ref{thm_graph} \ref{item_thmgraph2}, every complete affine $(C,k)$-surface can be written as $\gra{u^*}$ for some $u\in\SC_0(\Omega)$. It is essentially shown in \cite{lisimoncrelle} that the exact condition on such a $u$ for $\gra{u^*}$ to be an affine $(C,k)$-surface is a Monge-Amp\`ere equation involving the function $w_\Omega$ from Thm.\@ \ref{thm_chengyau}, which is itself the solution to a Monge-Amp\`ere equation (hence called a ``two-step Monge-Amp\`ere equation'' in \cite{lisimoncrelle}). By \cite[\S 7.2]{nie-seppi}, we can formulate this result as follows, also covering incomplete pieces of $C$-convex surfaces discussed in Remark \ref{remark_piece}:

\begin{proposition}
	\label{prop_mongeampere}
Let $C$ and $\Omega$ be as in \S \ref{subsec_cspacelike},  $U\subset\Omega$ be an open set and $u\in\C^\infty(U)$ be strictly convex. Let $\Sigma$ denote the graph over $\D u(U)$ of the Legendre transform $u^*$, \ie
$$
\Sigma:=\gra{u^*|_{\D u(U)}}=\big\{\big(x,\,x\cdot \D u(x)-u(x)\big)\,\big|\, x\in U\big\}\subset \R^3\cong\A^3
$$
(see Lemma \ref{lemma_gradient} \ref{item_gradient2} and Remark \ref{remark_piece}). Fix $k>0$. Then $\Sigma$ is an affine $(C,k)$-surface if and only if $u$ satisfies the Monge-Amp\`ere equation 
$$
\det\D^2u=k^{-\frac{2}{3}}w_\Omega^{-4}~
$$
in $U$. As a consequence, the $1$-to-$1$ correspondence in Thm.\@ \ref{thm_graph} \ref{item_thmgraph2} restricts to a correspondence between all complete affine $(C,k)$-surfaces and those $u\in\SC_0(\Omega)$ which satisfies the above equation in the interior of $\dom{u}$.
\end{proposition}
This proposition is a more detailed version of \cite[Cor.\@ 7.5]{nie-seppi} (which only deals with complete affine $(C,k)$-surfaces) and the proof is the same.
\subsection{Foliation by $C$-convex surfaces}\label{subsec_foliation} 
In view of Theorem \ref{thm_graph} \ref{item_thmgraph2}, we further ask the following question: Let $(u_t)_{t\in\R}$ be a one-parameter family of functions in $\SC_0(\Omega)$ with common boundary value $\phi\in\LC(\pa\Omega)$, so that $\Sigma_t=\gra{u_t^*}$ ($t\in\R$) are a family of complete $C$-convex surfaces  generating the same $C$-regular domain $D=\sepi{\overline{\phi}^*}$, $\phi\in\LC(\pa\Omega)$. Then when is $(\Sigma_t)_{t\in \R}$ a foliation of $D$?

In \cite[Thm.\@ 5.15]{nie-seppi}, we gave a necessary and sufficient condition on $(u_t)$ for $(\Sigma_t)$ to be a \emph{convex} foliation in the sense that the leaves are the level surfaces of a convex function on $D$. Restricting to the case where $\phi$ only takes values in $\R$, we have:
\begin{proposition}[{\cite[Thm.\@ 5.15]{nie-seppi}}]\label{prop_foliation}
	Let $\phi:\pa\Omega\to\R$ be a lower semicontinuous function and let $(u_t)_{t\in\R}\subset\SC_0(\Omega)$ be such that $u_t|_{\pa\Omega}=\phi$ for every $t$. Then the following conditions are equivalent:
	\begin{itemize}
		\item for every fixed $x\in\Omega$, $t\mapsto u_t(x)$ is a strictly increasing concave function on $\R$, with value tending to $-\infty$ and $\overline{\phi}(x)$ as $t$ tends to $-\infty$ and $+\infty$, respectively;
		\item  there is a convex function $K:\sepi{\overline{\phi}^*}\to\R$ such that $\gra{u_t^*}=K^{-1}(t)$ for every $t\in\R$.
	\end{itemize}
\end{proposition}
Since concave functions from $\R$ to $\R$ are continuous, the first bullet point implies that the graphs of the $u_t$'s themselves form a foliation of the strict lower epigraph 
$$
T^-:=\big\{(x,\xi)\in\Omega\times\R\,\big|\,\xi<\overline{\phi}(x))\big\}
$$ 
of $\overline{\phi}$. Therefore, we can define a map 
$$
F:T^-\to D
$$ 
in the following way: given $p\in T^-$, pick the leaf passing through $p$ in the foliation $(\gra{u_t})$, then let $F(p)$ be the supporting plane of this leaf at $p$, which can be understood as a point in $D$ (\cf the last paragraph of \S \ref{subsec_cregular}). See Figure \ref{figure_map}. 
\begin{figure}[h]
	\includegraphics[width=11cm]{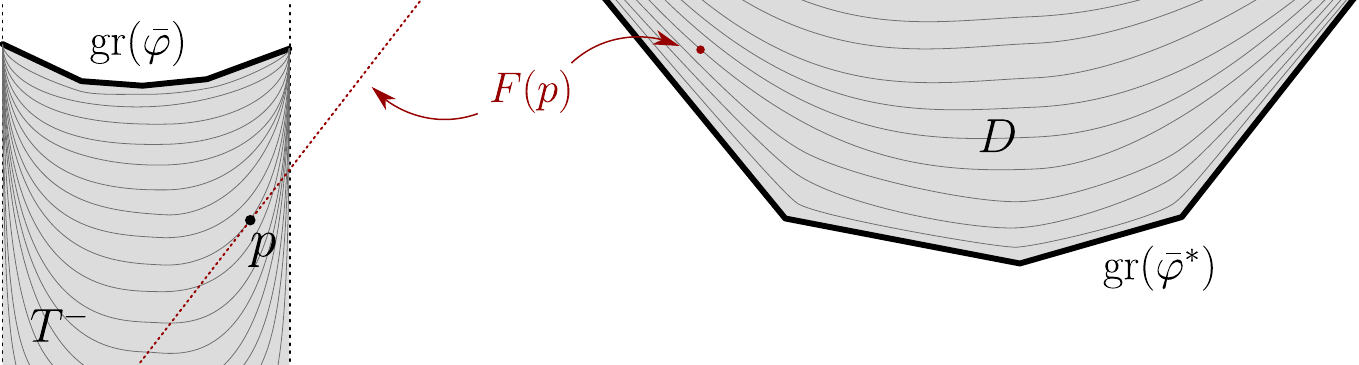}
	\caption{A section view of the foliations and the map $F$.}
	\label{figure_map}
\end{figure}

By Prop.\@ \ref{prop_legendre} and Lemma \ref{lemma_gradient}, $F$ sends each leaf $\gra{u_t}$ in $T^-$ to the leaf $\gra{u_t^*}$ in $D$, and has the expression
$$
F(x,u_t(x))=\big(\D u_t(x)\,,\, x\cdot\D u_t(x)-u_t(x)\big).
$$
As an ingredient in the proof of Corollary \ref{coro_intro}, we shall show:
\begin{proposition}\label{prop_homeo}
Suppose $(u_t)$ satisfies the conditions in Prop.\@ \ref{prop_foliation} and further assume that the common boundary value $\phi$ of the $u_t$'s is a bounded function. Then the map $F$ defined above is a homeomorphism.
\end{proposition}

\begin{proof}
By the above expression of $F$, we can write $F=F_2\circ F_1^{-1}$ with the maps $F_1$ and $F_2$ defined as follows:
\begin{align*}
F_1:&\Omega\times\R\to T^-,\quad F_1(x,t):=\big(x,u_t(x)\big)\\
F_2:&\Omega\times\R\to D,\quad F_2(x,t):=\big(\D u_t(x),\,x\cdot\D u_t(x)-u_t(x)\big).
\end{align*}
We shall prove the proposition by showing that $F_1$ and $F_2$ are both homeomorphisms. 

The assumption clearly implies that $F_1$ is bijective. $F_2$ is also bijective because on one hand, by Lemma \ref{lemma_gradient}, $F_2$ sends each slice $\Omega\times\{t\}$ bijectively to the graph $\gra{u_t^*}$; on the other hand, $(\gra{u_t^*})_{t\in\R}$ is a foliation of $D$ by Prop.\@ \ref{prop_foliation}. 

We proceed to show that $F_1$ and $F_2$ are continuous via the following two claims.

\vspace{5pt} 
\textbf{Claim 1: For any $t_0\in\R$ and compact set $\Omega_0\subset\Omega$, $u_t$ converges to $u_{t_0}$ uniformly on $\Omega_0$ as $t\to t_0$.} Suppose by contradiction that it is not the case. Then there exists a sequence $(t_n)_{n=1,2,\cdots}$ in $\R$ converging to $t_0$ and a sequence $(x_n)_{n=1,2,\cdots}$ in $\Omega_0$ such that $|u_{t_n}(x_n)-u_{t_0}(x_n)|\geq\delta$ for all $n$ and some fixed $\delta>0$. By restricting to a subsequence, we may assume that $x_n$ converges to some $x_0\in\Omega_0$ as $n\to\infty$, and that either $t_n> t_0$ for all $n$ or $t_n< t_0$ for all $n$. Moreover, since $u_{t_n}(x_0)\to u_{t_0}(x_0)$ as by assumption, we may further assume that $x_n\neq x_0$ for all $n$.

We first treat the $t_n>t_0$ case, in which we have $u_{t_n}(x_n)-u_{t_0}(x_n)\geq\delta$ because $u_t(x)$ is increasing in $t$. Since $u_{t_0}(x_n)\to u_{t_0}(x_0)$ by continuity of $u_{t_0}$ and $u_{t_n}(x_0)\to u_{t_0}(x_0)$ by assumption, we have
$$
u_{t_n}(x_n)-u_{t_n}(x_0)= \big(u_{t_n}(x_n)-u_{t_0}(x_n)\big)+\big(u_{t_0}(x_n)-u_{t_0}(x_0)\big)+\big(u_{t_0}(x_0)-u_{t_n}(x_0)\big)\geq \frac{\delta}{2}
$$
when $n$ is large enough. Let $x_n'$ and $x_n''$ be the boundary points of $\Omega$ such that $x_n'$, $x_0$, $x_n$ and $x_n''$ lie on the same line in this order. We shall compare the restriction of $u_{t_n}$ to this line with the affine function $h_n$ on the line which takes the same values as $u_{t_n}$ at $x_0$ and $x_n$. Namely, $h_n$ is given by 
$$
h_n(x_0+s(x_n-x_0))=u_{t_n}(x_0)+s\big(u_{t_n}(x_n)-u_{t_n}(x_0)\big)
$$
for all $s\in\R$, and we have $u_{t_n}(x_n'')\geq h_n(x_n'')$ by convexity of $u_{t_n}$. It follows that
$$
\phi(x_n'')=u_{t_n}(x_n'')\geq h_n(x_n'')= u_{t_n}(x_0)+\frac{x_n''-x_0}{x_n-x_0}\big(u_{t_n}(x_n)-u_{t_n}(x_0)\big).
$$
This leads to a contradiction: The right-hand side tends to $+\infty$ as $n\to\infty$ because $u_{t_n}(x_0)\to u_{t_0}(x_0)$, $\frac{x_n''-x_0}{x_n-x_0}\to+\infty$ and $u_{t_n}(x_n)-u_{t_n}(x_0)\geq\frac{\delta}{2}$, whereas the left-hand side is bounded by assumption.

In the $t_n<t_0$ case, we can replace $x_n''$ by $x_n'$ in the above argument (note that $\frac{x_n'-x_0}{x_n-x_0}\to-\infty$) and arrive at a contradiction in the same way. This proves Claim 1.

\vspace{5pt}
\textbf{Claim 2: The map $(x,t)\to\D u_t(x)$ is continuous.} Pick $(x_0,t_0)\in\Omega\times\R$. Since modifying $(u_t)$ by adding the same affine function to every $u_t$ does not affect the statement, we may assume $u_{t_0}(x_0)=0$ and $\D u_{t_0}(x_0)=0$. As a consequence, there is a constant $C>0$, such that when $r>0$ is small enough, we have
$$
|u_{t_0}(x)|\leq Cr^2,\ \ \forall x\in B(x_0,r)
$$
Given such an $r$, by Claim 1, we find $\delta_r>0$ such that
\begin{equation}\label{eqn_prooffoli2}
|u_t(x)|\leq |u_t(x)-u_{t_0}(x)|+|u_{t_0}(x)|\leq 2Cr^2,\ \  \forall x\in B(x_0,r),\, t\in(t_0-\delta_r,t_0+\delta_r).
\end{equation}

We now use \eqref{eqn_prooffoli2} to show that for any $r$ as above, we have
\begin{equation}\label{eqn_prooffoli3}
|\D u_t(x_1)|\leq 8Cr,\ \ \forall (x_1,t)\in B(x_0,r/2)\times(t_0-\delta_r,t_0+\delta_r),
\end{equation}
which would imply the claim. To this end, we fix $(x_1,t)\in B(x_0,r/2)\times(t_0-\delta_r,t_0+\delta_r)$ and consider the supporting affine function $h$ of $u_t$ at $x_1$, \ie
$$
h(x):=u_t(x_1)+\D u_t(x_1)\cdot(x-x_1).
$$
We have $h\leq u_t$ by convexity of $u_t$, hence $h\leq 2Cr^2$ on $B(x_0,r)$ by \eqref{eqn_prooffoli2}. On the other hand, we have $h(x_1)=u_t(x_1)\geq-2Cr^2$ also by \eqref{eqn_prooffoli2}. Therefore, \eqref{eqn_prooffoli3} follows from the lemma below, and the proof of Claim 2 is finished.

The two claims imply that the maps $F_1$ and $F_2$ are continuous. Since $F_1$ and $F_2$ are already shown to be bijective, by Brouwer Invariance of Domain, they are indeed homeomorphisms, as required.
\end{proof}
\begin{lemma}
Let $\epsilon,r>0$ and $x_0\in\R^2$. If an affine function $h(x)=x\cdot y+\eta$ ($y\in\R^2$, $\eta\in\R$) on $\R^2$ satisfies
\begin{itemize}
	\item $h\leq \epsilon$ on the boundary of the disk $B(x_0,r)$, 
	\item $h(x_1)\geq -\epsilon$ for some $x_1\in B(x_0,r/2)$,  
\end{itemize}
then the gradient $y$ of $h$ satisfies
$$
|y|\leq \frac{4\epsilon}{r}.
$$
\end{lemma}
\begin{proof}
Since shifting $h$ by a translation of $\R^2$ does not affect the statement, we may assume $x_0=0$. The assumptions imply 
$$
\epsilon\geq h\big(\tfrac{r}{|y|}y\big)=r|y|+\eta,
$$
$$
-\epsilon\leq h(x_1)=x_1\cdot y+\eta\leq |x_1||y|+\eta\leq \tfrac{r}{2}|y|+\eta.
$$
Subtracting the first inequality by the second, we get 
$2\epsilon\geq \tfrac{r}{2}|y|$, which implies the required inequality.
\end{proof}

\section{Affine deformations of quasi-divisible convex cones}\label{sec:affine deformation}
In this section, we first review backgrounds on parabolic automorphisms of convex cones, then we define admissible cocycles and prove Prop.\@ \ref{prop_intro1} in the introduction. 

\subsection{Quasi-divisibility and parabolic elements}\label{subsec_parabolic}
For the future light cone $C_0$ in $\R^{2,1}$, the automorphism group $\Aut(C_0)$ is the identity component $\SO_0(2,1)$ of $\SO(2,1)$ and identifies with the group $\Isom^+(\mathbb{H}^2)$ of orientation-preserving isometries of the hyperboloid $\mathbb{H}^2\subset C_0$. It is well known that non-identity elements in this group
are classified into three types, namely elliptic, hyperbolic and parabolic ones, and that the fundamental group of a complete hyperbolic surface with finite volume, viewed as a Fuchsian group in $\SO_0(2,1)$, only contains the last two types of elements, with parabolic ones corresponding to the cusps of the surface.

More generally, an element $A\neq I$ in $\SL(3,\R)$ is said to be \emph{parabolic} if it is conjugate to a parabolic element of $\SO(2,1)$, or equivalently, if $A$ has the Jordan form
$$
\matrix{1&1&\\&1&1\\&&1},
$$ 
whereas $A$ is said to be \emph{hyperbolic} if it is diagonalizable. 
Both types of elements can occur in the automorphism group $\Aut(C)$ of a proper convex cone $C\subset\R^3$, and their action can be visualized in the projectivized picture as Figure \ref{figure_parabolic}.
\begin{figure}[h]
	\includegraphics[width=5cm]{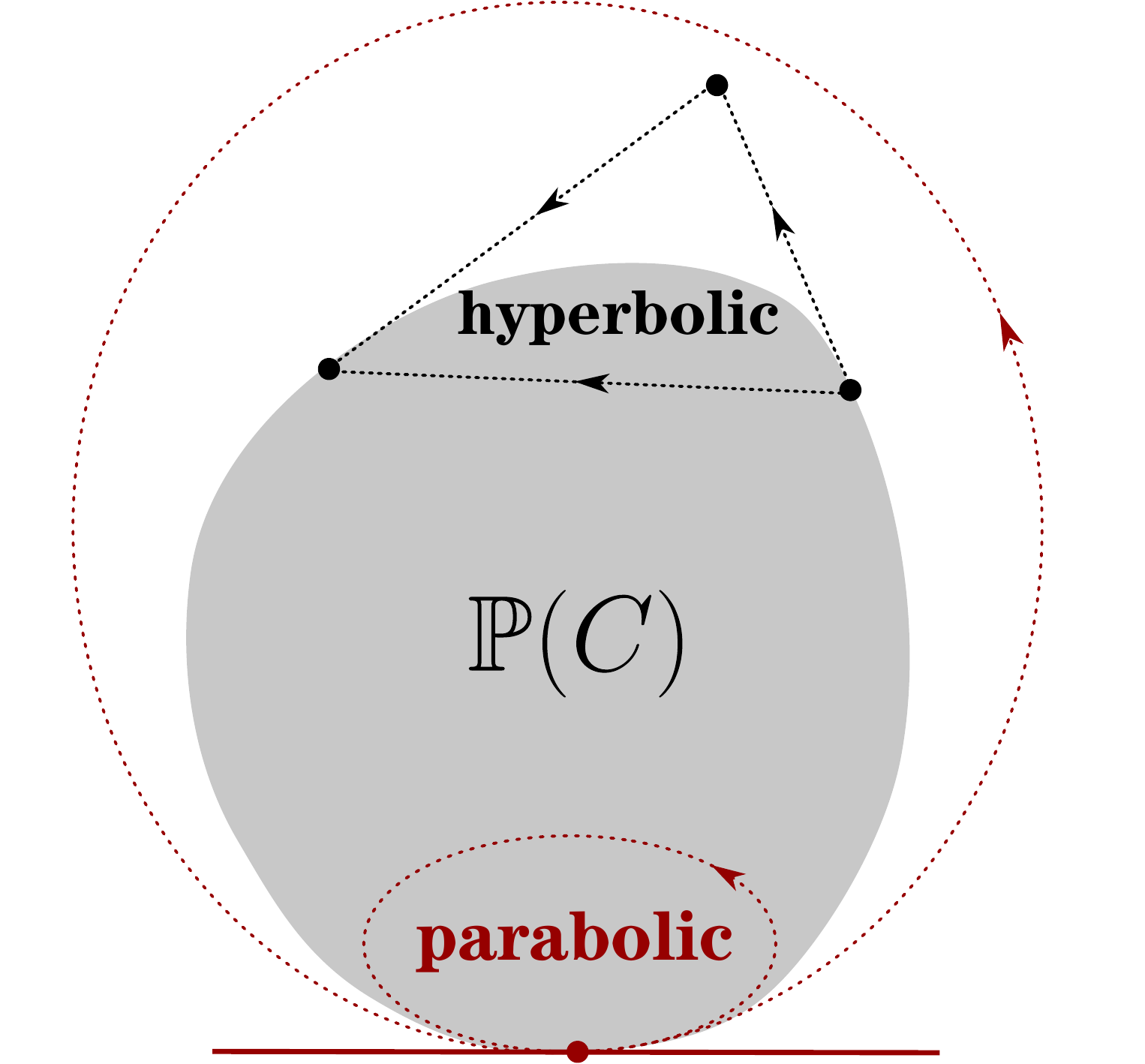}
	\caption{Parabolic and hyperbolic elements in $\Aut(C)$.}
	\label{figure_parabolic}
\end{figure}

For a parabolic element $A\in\Aut(C)$, the eigenline and the $2$-dimensional generalized eigenspace are, respectively, the unique line in $\pa C$ fixed by $A$ and the unique $C$-null subspace of $\R^3$ (\cf Definition \ref{def_cspacelike}) fixed by $A$. In the projectivized picture, they correspond to a point $p_A\in \pa\P(C)$ and the tangent line of $\pa\P(C)$ at $p_A$, respectively.
As shown by Benoist-Hulin \cite{benoist-hulin}, there is a pencil of conics in $\RP^2$ based at $p_A$, in which every conic is preserved by $A$, and $\pa\P(C)$ is pinched between two of them (see Figure \ref{figure_parabolic}). We formulated this result as:

\begin{lemma}[{\cite[Prop.\@ 3.6 (3)]{benoist-hulin}}]\label{lemma_benoisthulin}
	Suppose $C\subset\R^3$ is a proper convex cone and $A\in\Aut(C)$ is parabolic. Then there are projective disks $D_1$, $D_2$ in $\RP^2$ preserved by $A$, such that $D_1\subset \P(C)\subset D_2$, and the boundaries $\pa D_1$ and $\pa D_2$ only touch at the fixed point of $A$.
\end{lemma}

The definition of \emph{divisible} and \emph{quasi-divisible} convex cones are briefly reviewed in the introduction, and we refer to \cite{benoist_survey, marquis} for details. However, instead of the definition itself, we will rather use the following characterization of quasi-divisibility due to Marquis, which generalizes the aforementioned fact for Fuchsian groups: 
\begin{theorem}[{\cite{marquis_fourier}}]\label{thm_marquis}
Let $C\subset\R^3$ be a proper convex cone and $\Gamma$ be a torsion-free discrete subgroup of $\SL(3,\R)$ contained in $\Aut(C)$. Then $C$ is quasi-divisible by $\Gamma$ if and only if $S:=\P(C)/\Gamma$ is homeomorphic to a closed surface with finitely many (possibly zero) punctures and every puncture has a neighborhood of the form $D/\langle A\rangle$, where $D\subset\P(C)$ is a projective disk and $A\in\Gamma$ is a parabolic element preserving $D$. In this case, the holonomy of every non-peripheral loop on $S$ is hyperbolic.
\end{theorem}

We collect some other well known results about quasi-divisible convex cones that will be used later on:
\begin{proposition}\label{prop_qd}
Let $C\subset\R^3$ be a proper convex cone quasi-divisible by a torsion-free group $\Gamma<\SL(3,\R)$. Then the following statements hold. 
\begin{enumerate}[label=(\arabic*)]
\item\label{item_qd1}
$\pa\P(C)$ is strictly convex and $\C^1$.
\item\label{item_qd3}
In $\pa\P(C)$, every $\Gamma$-orbit is dense. 
\item\label{item_qd2} 
Suppose $\gamma_1,\gamma_2\in\Gamma\setminus\{I\}$ are primitive (\ie $\gamma_i$ is not a positive power of any other element of $\Gamma$). Then $\gamma_1$ and $\gamma_2$ share a fixed point if and only if $\gamma_1=\gamma_2^{\pm1}$.
\item\label{item_qd4} 
Let $\Gamma^*<\SL(3,\R)$ be the image of $\Gamma$ under the isomorphism $\Aut(C)\cong\Aut(C^*)$ (see \S \ref{subsec_omega}). Then the dual cone $C^*$ of $C$ is quasi-divisible by $\Gamma^*$.
\end{enumerate}
\end{proposition}
Parts \ref{item_qd1} and \ref{item_qd4} are Thm.\@ 0.5 and Thm.\@ 0.4, respectively, of the aforementioned work of Marquis \cite{marquis_fourier}, whereas \ref{item_qd3} and  \ref{item_qd2}  are generalizations of well known facts for Fuchsian groups and can be shown with the same argument as in the Fuchsian case (see \eg \cite[\S 5]{Beardon})

\subsection{Affine transformation with parabolic linear part}
In order to study affine deformations of quasi-divisible convex cones, let us consider affine transformations $(A,X)\in\Aut(C)\ltimes\R$ such that $A$ is parabolic. The following fundamental property of such $(A,X)$'s is the origin of our definition of admissible cocycles:
\begin{proposition}\label{prop_parabolic}
	Let $C\subset\R^3$ a proper convex cone and let $(A,X)\in\Aut(C)\ltimes\R^3$ be such that $A$ is parabolic. Let $L_0\subset\R^3$ be the eigenline of $A$ and $P_0\subset\R^3$ be the $C$-null subspace preserved by $A$. Then the following conditions are equivalent to each other:
	\begin{enumerate}[label=(\alph*)]
		\item\label{item_parabolic1} the vector $X$ lies in $P_0$;
		\item\label{item_parabolic2} $(A,X)$ is conjugate to $A$ through a translation;
		\item\label{item_parabolic3} $(A,X)$ preserves an affine plane in $\A^3$.
	\end{enumerate}
	When these conditions are satisfied, the affine planes preserved by $(A,X)$ are exactly those parallel to $P_0$, and there is a distinguished one $P_1$ among these planes, such that the set of fixed points of $(A,X)$ in $\A^3$ is an affine line  lying on $P_1$ parallel to $L_0$.
\end{proposition}
Here, in Condition \ref{item_parabolic2}, $\Aut(C)$ and $\R^3$ are both viewed as subgroups of $\Aut(C)\ltimes\R^3$, and the normal subgroup $\R^3$ is referred to as the group of \emph{translations}.

The core of Prop.\@ \ref{prop_parabolic} is the following basic property of the Jordan form of $A$:
\begin{lemma}\label{lemma_jordan}
	The following statements hold for the linear transformation
	$$
	A_0:=\matrix{1&1&\\&1&1\\&&1}.
	$$ 
	\begin{enumerate}[label=(\arabic*)]
		\item\label{item_jordan1} The affine planes in $\R^3$ preserved by $A_0$ are exactly the horizontal planes $\R^2\times\{\xi\}$.
		\item\label{item_jordan2} Let $X\in\R^3$ be a vector not in the generalized eigenspace $\R^2\times\{0\}$ of $A_0$. Then the affine transformation $(A_0,X)\in\SL(3,\R)\ltimes\R^3$ does not preserve any affine plane.
	\end{enumerate}
\end{lemma}
\begin{proof}
	It is an elementary fact that if an affine transformation $(A,X)\in\SL(3,\R)\ltimes\R^3$ preserves an affine plane $P$, then the linear part $A$ preserves the $2$-dimensional subspace of $\R^3$ parallel to $P$. 
	
	One readily checks that every horizontal plane $\R^2\times\{\xi\}$ is preserved by $A_0$. Conversely, by the above fact, these are the only affine planes preserved because the generalized eigenspace $\R^2\times\{0\}$ is the only $2$-dimensional subspace preserved. This proves Part \ref{item_jordan1}. 
	
	The above fact also implies that an affine plane preserved by $(A_0,X)$ can only has the form $\R^2\times\{\xi\}$ as well. So we obtain Part \ref{item_jordan2} by noting that if $X$ is not horizontal then $\R^2\times\{\xi\}$ is not preserved.
\end{proof}

\begin{proof}[Proof of Prop.\@ \ref{prop_parabolic}]
	Since $A$ is conjugate to the $A_0$ in Lemma \ref{lemma_jordan}, by the lemma, the affine planes preserved by $A$ are exactly those parallel to $P_0$. Also, the fixed points of $A$ are exactly the points of the eigenline $L_0$. Therefore, when Condition \ref{item_parabolic2} is satisfied, the last two statements of the proposition holds. This proves the ``Moreover'' part. It remain to show the equivalence between \ref{item_parabolic1}, \ref{item_parabolic2} and \ref{item_parabolic3}.
	
	The conjugate of $A$ in $\Aut(C)\ltimes\R^3$ by a translation $X_0\in\R^3$ is $(A,(I-A)X_0)$, because it sends $Y\in\R^3$ to $A(Y-X_0)+X_0=AY+(I-A)X_0$. Therefore, the implication ``\ref{item_parabolic1}$\Rightarrow$\ref{item_parabolic2}'' follows from the fact that
	$$
	\big\{(I-A)X_0\,\mid\,X_0\in\R^3\big\}=P_0,
	$$
	which can be easily checked using the Jordan form $A_0$. The implication ``\ref{item_parabolic2}$\Rightarrow$\ref{item_parabolic3}'' is obvious. Finally, 
	``\ref{item_parabolic3}$\Rightarrow$\ref{item_parabolic1}'' follows immediately from Lemma \ref{lemma_jordan}.
\end{proof}

\begin{remark}
	In terms of fixed planes, affine transformations with hyperbolic linear part behave differently from those with parabolic linear part. In fact, given a hyperbolic $A\in\Aut(C)$, one readily checks that 
	\begin{itemize}
		\item
		if the middle eigenvalue of $A$ is $1$, then every $(A,X)\in\Aut(C)\ltimes\R^3$ is conjugate through a translation to some $(A,X_0)$ with $X_0$ in the middle eigenspace;
		\item 
		otherwise, $(A,X)$ must be conjugate through a translate to $A$ itself. 
	\end{itemize}
	In any case, $(A,X)$ has two fixed $C$-null planes, parallel to the two $C$-null subspaces fixed by $A$.
\end{remark}

Now let $C$ and $\Omega$ be as in \S \ref{sec:correspondence}, so that we have isomorphisms $\Aut(C)\cong\Aut(C^*)=\Aut(\Omega)$ and $\Aut(C)\ltimes\R^3\cong\Aut(\Omega\times\R)$ (see Prop.\@ \ref{prop_iso}). The former isomorphism sends a parabolic $A\in\Aut(C)$ to the inverse transpose $\transp{\!A}^{-1}\in\Aut(\Omega)$, which is again parabolic. Therefore, by the commutative diagram in Prop.\@ \ref{prop_iso}, an element $(A,X)$ of $\Aut(C)\ltimes\R^3$ with $A$ parabolic corresponds, under the latter isomorphism, to some $\Phi\in\Aut(\Omega\times\R)$ whose projection $\pi(\Phi)\in\Aut(\Omega)$ is parabolic. Thus, the correspondence between the two geometries explained in the introduction and the previous section translates Prop.\@ \ref{prop_parabolic} into the following dual result about $\Phi$:
\begin{corollary}\label{coro_fixedaffinefunction}
	Let $\Phi$ be an automorphism of the convex tube domain $\Omega\times\R\subset\R^3$ whose projection $\pi(\Phi)\in\Aut(\Omega)$ is parabolic with fixed point $p\in\pa\Omega$.
	Then we have the following dichotomy: either
	\begin{itemize}
		\item $\Phi$ does not have any fixed point in $\R^3$, or
		\item the set of fixed points of $\Phi$ is the vertical line $\{p\}\times\R$.
	\end{itemize}
	In the latter case, there is a distinguished fixed point $(p,\xi_0)$ and a non-vertical line $L\subset\R^3$ tangent to $\Omega\times\R$ at $(p,\xi_0)$ such that the affine planes in $\R^3$ preserved by $\Phi$ are exactly those containing $L$ (see Figure \ref{figure_tube}).
\end{corollary}
\begin{figure}[h]
	\includegraphics[width=8.8cm]{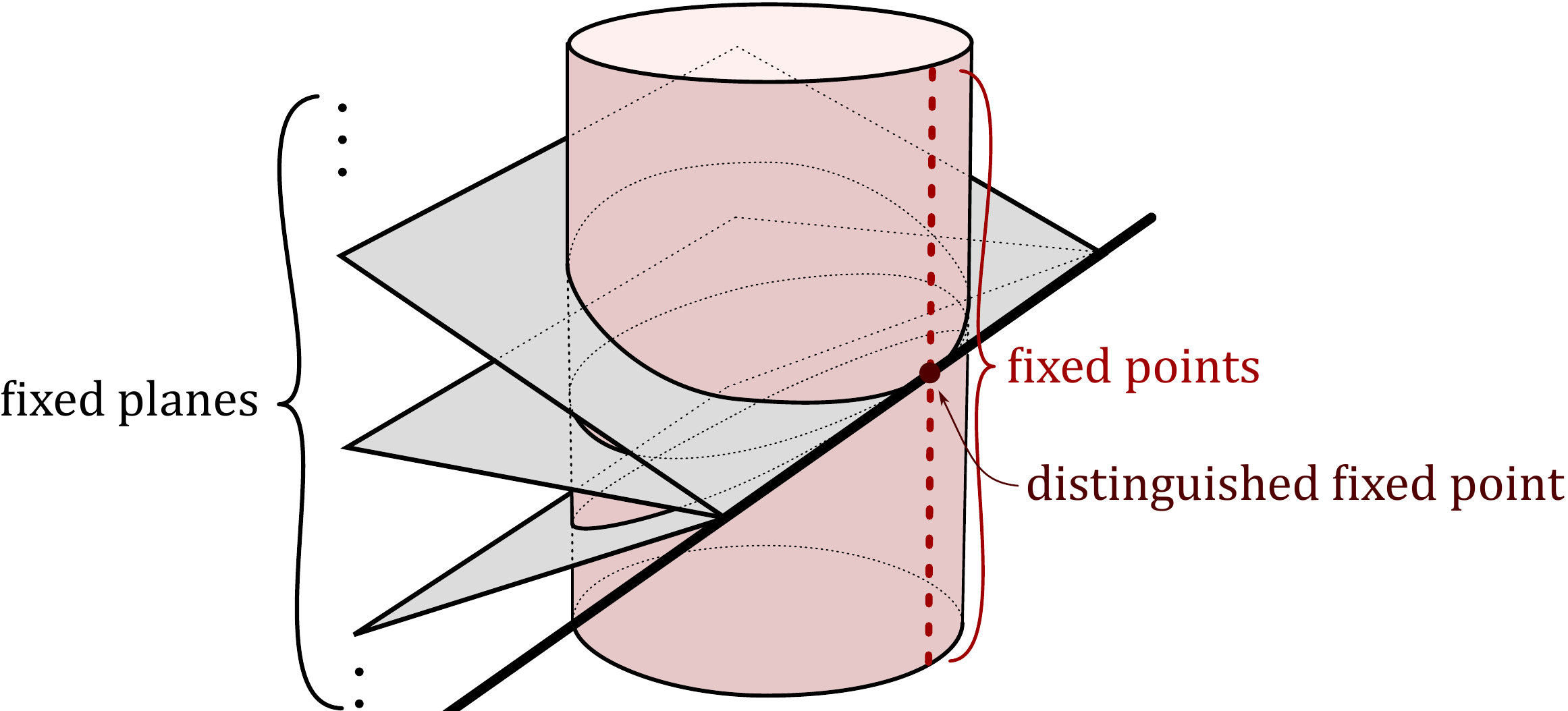}
	\caption{All fixed points and affine planes of $\Phi\in\Aut(\Omega\times\R)$ with $\pi(\Phi)$ parabolic, when $\Phi$ does have fixed points.}
	\label{figure_tube}
\end{figure}

In particular, if we view a parabolic element $A$ of $\Aut(\Omega)$ with fixed point $p\in\pa\Omega$ as an automorphisms of $\Omega\times\R$ preserving the slice $\Omega\times\{0\}$ (see \S \ref{subsec_subgroup}), then the distinguished fixed point of $A$ is just $(p,0)$. We will need the following result about the existence of certain smooth convex functions with $A$-invariant graph:
\begin{lemma}\label{lemma_parabolicfunction}
	Let $A$ be a parabolic element in $\Aut(\Omega)$ fixing $p\in\pa\Omega$. Then for any $\mu\geq0$, there is a convex $f\in\C^\infty(\Omega)$ such that the boundary value of $f$ at $p$ (in the sense of \S \ref{subsec_convexfunction}) is $-\mu$ and the graph of $f$ is preserved by $A$ (as an automorphisms of $\Omega\times\R$).
\end{lemma}
Note that the $\mu=0$ case of the lemma follows immediately from Cor.\@ \ref{coro_fixedaffinefunction}, which implies that we can actually take $f$ to be an affine function in this case.
\begin{proof}
By Lemma \ref{lemma_benoisthulin}, $\Omega$ is contained in an $A$-invariant projective disk whose boundary passes through $p$. By choosing the coordinates of $\R^2$ appropriately, we may assume that $p=(-1,0)$ and the projective disk is just $\mathbb{D}:=\{x\in\R^2\mid |x|<1\}$. Let $\psi$ be the lower semicontinuous function on $\pa\mathbb{D}$ such that $\psi=0$ on $\pa\mathbb{D}\setminus\{p\}$ and $\psi(p)=-\mu$. 

We view $A$ as a projective transformation of the round tube domain $\mathbb{D}\times\R$ preserving both the subdomain $\Omega\times\R$ and the slice $\mathbb{D}\times\{0\}$. Its action on $\pa\mathbb{D}\times\R$ preserves the punctured circle $(\pa\mathbb{D}\!\setminus\{p\})\times\{0\}$ and pointwise fixes the vertical line $\{p\}\times\R$. It follows that the graph of $\psi$ is preserved by $A$, hence so is the graph of the convex envelop $\env{\psi}$. 

Let $u$ denote the function on $\overline{\mathbb{D}}$ whose restriction to any line segment $I$ joining $p$ and another point $x\in\pa\mathbb{D}$ is the affine function on $I$ interpolating the values $-\mu$ and $0$ at $p$ and $x$. By elementary calculations, we find the expression of $u$ in $\mathbb{D}$ to be
$$
u(x)=\mu\left(\frac{(x_1+1)^2+x_2^2}{2(x_1+1)}-1\right) \ \ \text{ for any }x=(x_1,x_2)\in\mathbb{D},
$$
and a further calculation shows that the Hessian of $u$ is positive semi-definite, hence $u$ is convex. Also, the boundary value of $u$ is exactly $\psi$. 

We claim that $\env{\psi}=u$ in $\mathbb{D}$. In fact, on one hand, we have $\env{\psi}\leq u$ because by convexity of $\env{\psi}$ and the fact that $\psi=u$ on $\pa\mathbb{D}$, the restriction of $\env{\psi}$ to the aforementioned segment $I$ is less than or equal to the affine function described; on the other hand, $\env{\psi}\geq u$ because $\env{\psi}$ is pointwise no less than any convex function with boundary value $\psi$. 

As a consequence of the claim,  $\env{\psi}=u$ is smooth in $\mathbb{D}$ by the above expression. Also, $\env{\psi}$ has $A$-invariant graph and has boundary value $-\mu$ at $p$. Thus, the restriction of $\env{\psi}$ to $\Omega$ gives the required function $f$.  
\end{proof}

\begin{remark}
In contrast to Lemma \ref{lemma_parabolicfunction}, there is no convex function $f:\Omega\to\R$ with $A$-invariant graph whose boundary value at $p$ is strictly positive. This can be proved by reducing again to the case $\Omega=\mathbb{D}$ and showing that for any iteration sequence $(A^n(x))_{n=1,2,\cdots}$ in $\pa\mathbb{D}\setminus\{p\}$ converges to $p$, we have $\lim_{n\to+\infty}f(A^n(x_0))=0$ by invariance of the graph. More generally, given any $\Phi\in\Aut(\Omega\times\R)$ such that $\pi(\Phi)$ is parabolic and $\Phi$ has fixed points, we deduces from Prop.\@ \ref{prop_parabolic} that $\Phi$ is conjugate to $\pi(\Phi)$. So the above results imply that if $(p,\xi_0)$ is the distinguished fixed point of $\Phi$ in the sense of Cor.\@ \ref{coro_fixedaffinefunction}, then there exists a convex function with $\Phi$-invariant graph and with boundary value $\xi$ at $p$ if and only if $\xi\leq\xi_0$.
\end{remark}

\subsection{Admissible cocycles}
Given a group $\Gamma<\SL(3,\R)$ and a map $\tau:\Gamma\to\R^3$, denote
$$
\Gamma_\tau:=\big\{(A,\tau(A))\in\SL(3,\R)\ltimes\R^3\mid A\in\Gamma\big\}.
$$ 
The following facts are well known and easy to verify:
\begin{itemize}
	\item
$\Gamma_\tau$ is a subgroup of $\SL(3,\R)\ltimes\R^3$ if and only if $\tau$ is a \emph{$1$-cocycle} on $\Gamma$ with values in the $\Gamma$-module $\R^3$, which means more precisely that $\tau$ belongs to the vector space
$$
Z^1(\Gamma,\R^3):=\big\{\tau:\Gamma\to\R^3\ \big|\ \tau(AB)=\tau(A)+A\tau(B),\,\forall A,B\in\Gamma\big\}.
$$
In this case, we call $\Gamma_\tau$ an \emph{affine deformation} of $\Gamma$.
\item
Two affine deformations $\Gamma_{\tau_1}$ and $\Gamma_{\tau_2}$ are said to be \emph{equivalent} if there is $X_0\in\R^3$ such that $(A,\tau_1(A))$ is conjugate to $(A,\tau_2(A))$ through the translation $X_0$ for any $A\in\Gamma$.
It is the case if and only if $\tau_1-\tau_2$ lies in the vector space of \emph{$1$-coboundaries}
$$
B^1(\Gamma,\R^3):=\big\{\tau_X\ \big|\ X\in\R^3\big\}, \text{ where }\tau_X(A):=(I-A)X.
$$
\end{itemize}
Therefore, the first cohomology of $\Gamma$ with values in $\R^3$, \ie the vector space 
$$
H^1(\Gamma,\R^3):=Z^1(\Gamma,\R^3)\big/B^1(\Gamma,\R^3),
$$ 
is the space of equivalence classes of affine deformations. In view of Prop.\@ \ref{prop_parabolic}, we further define:
\begin{definition}
Let $C\subset\R^3$ be a proper convex cone quasi-divisible by a torsion-free group $\Gamma<\SL(3,\R)$. Then a cocycle $\tau\in Z^1(\Gamma,\R^3)$ is said to be \emph{admissible} if for every parabolic element $A$ in $\Gamma$, the vector $\tau(A)$ is in the $C$-null subspace preserved by $A$. 
\end{definition}
In particular, since  $(I-A)X$ is contained in the $C$-null subspace preserved by $A$ for any parabolic $A\in\Aut(C)$ and $X\in\R^3$, every $1$-coboundary $\tau_X$ is admissible.

Although the above definitions are made for a subgroup $\Gamma<\SL(3,\R)$, one can readily adapt them to representations $\rho\in\Hom(\Pi,\SL(3,\R))$, where $\Pi$ is an abstract group. Precisely, the space of $1$-cocycles and $1$-coboundaries for $\rho$ are defined as  
$$
Z^1_\rho(\Pi,\R^3):=\big\{\tau:\Pi\to\R^3\ \big|\ \tau(\alpha\beta)=\tau(\alpha)+\rho(\alpha)\tau(\beta),\,\forall \alpha,\beta\in\Pi\big\},
$$
$$
B^1_\rho(\Pi,\R^3):=\big\{\tau_X\ \big|\ X\in\R^3\big\}, \text{ where } \tau_X(\gamma):=(I-\rho(\gamma))X,
$$
so that every representation of $\Pi$ in $\SL(3,\R)\ltimes\R^3$ can be written as 
$$
\rho_\tau(\gamma):=(\rho(\gamma),\tau(\gamma))
$$
for some $\rho\in\Hom(\Pi,\SL(3,\R))$ and $\tau\in Z^1_\rho(\Pi,\R^3)$, and two such representations are conjugate to each other through a translation if and only if they have the form $\rho_{\tau_1}$ and $\rho_{\tau_2}$, with $\tau_1-\tau_2\in B^1_\rho(\Pi,\R^3)$. Moreover, when there is a proper convex cone $C\subset\R^3$ quasi-divisible by the image $\rho(\Pi)$, we call $\tau\in Z^1_\rho(\Pi,\R^3)$ \emph{admissible} if $\tau(\gamma)$ is in the $C$-null subspace preserved by $\rho(\gamma)$ whenever $\rho(\gamma)$ is parabolic.

\subsection{Moduli spaces}
Let $C\subset\R^3$ be a proper convex cone quasi-divisible by a torsion-free group $\Gamma<\SL(3,\R)$. Since $\P(C)$ is topologically a disk and $\Gamma$ acts on it properly discontinuously by orientation preserving homeomorphism, Theorem \ref{thm_marquis} implies that the surface $S:=\P(C)/\Gamma$ is homeomorphic to the orientable surface $S_{g,n}$ with genus $g$ and $n$ punctures, where the nonnegative integers $g$ and $n$ satisfy:
\begin{itemize}
	\item $(g,n)\neq (0,0)$ or $(0,1)$, as $S$ cannot be homeomorphic to the sphere or the disk;
	\item $(g,n)\neq(0,2)$, since otherwise $\Gamma$ is the cyclic group generated by a single element in $\Aut(C)$, and it would be impossible that both punctures of $S\approx S_{0,2}$ have the property in the conclusion of Theorem \ref{thm_marquis}.
\end{itemize}
Therefore, we have either $(g,n)=(1,0)$ or $2-2g-n<0$. Namely, $S$ is homeomorphic to either the torus or some $S_{g,n}$ with negative Euler characteristic.

For each allowed $(g,n)$, the \emph{moduli space of convex projective structures with finite volume} on $S_{g,n}$ is defined as the topological quotient
$$
\mathcal{P}_{g,n}:=\Hom_0(\pi_1(S_{g,n}),\SL(3,\R))\big/\SL(3,\R),
$$
where $\SL(3,\R)$ acts by conjugation on the space of representations
\begin{align*}
&\Hom_0(\pi_1(S_{g,n}),\SL(3,\R))\\
&:=\left\{\rho\in \Hom(\pi_1(S_{g,n}),\SL(3,\R))\ \Bigg|\ \parbox{6.4cm}{$\rho$ is faithful and there is a proper convex cone $C\subset\R^3$ quasi-divisible by $\rho(\pi_1(S_{g,n}))$}\right\}
\end{align*}

Similarly, we define the moduli space of \emph{admissible affine deformations of convex projective structures with finite volume} on $S_{g,n}$ as the quotient
$$
\widehat{\mathcal{P}}_{g,n}:=\Hom_0(\pi_1(S_{g,n}),\SL(3,\R)\ltimes\R^3)\big/\SL(3,\R)\ltimes\R^3,
$$
where $\SL(3,\R)\ltimes\R^3$ acts by conjugation on 
\begin{align*}
&\Hom_0(\pi_1(S_{g,n}),\SL(3,\R)\ltimes\R^3)\\
&:=\left\{\rho_\tau\in \Hom(\pi_1(S_{g,n}),\SL(3,\R)\ltimes\R^3)\ \Bigg|\ \parbox{5cm}{$\rho$ is in $\Hom_0(\pi_1(S_{g,n}),\SL(3,\R))$, $\tau\in Z^1_\rho(\pi_1(S_{g,n}),\R^3)$ is admissible}\right\}~.
\end{align*}

The natural projection  
\begin{equation}\label{eqn_projection}
\Hom_0(\pi_1(S_{g,n}),\SL(3,\R)\ltimes\R^3)\to\Hom_0(\pi_1(S_{g,n}),\SL(3,\R)),\quad \rho_\tau\mapsto \rho
\end{equation}
induces a projection $\widehat{\mathcal{P}}_{g,n}\to\mathcal{P}_{g,n}$, which is a surjective continuous map. We now prove Proposition \ref{prop_intro1} in the introduction by showing that when $2-2g-n<0$, the latter projection is a vector bundle of rank $6g-6+2n$. 
\begin{proof}[Proof of Prop.\@ \ref{prop_intro1}]
We view $\widehat{\mathcal{P}}_{g,n}$ a two-step quotient of $\Hom_0(\pi_1(S_{g,n}),\SL(3,\R)\ltimes\R^3)$:
\begin{equation}\label{eqn_quotient}
\widehat{\mathcal{P}}_{g,n}=\left(\Hom_0(\pi_1(S_{g,n}),\SL(3,\R)\ltimes\R^3)\big/\R^3\right)\big/ \SL(3,\R).
\end{equation}
Namely, first quotienting by the normal subgroup of translations $\R^3\lhd \SL(3,\R)\ltimes\R^3$, then by the quotient group $\SL(3,\R)=(\SL(3,\R)\ltimes\R^3)/\R^3$. We claim that 
\begin{itemize}
\item
the projection \eqref{eqn_projection} is a vector bundle of rank $6g-3+2n$;
\item the first quotient in \eqref{eqn_quotient} is the quotient of this vector bundle by a sub-bundle of rank $3$.
\end{itemize}

To prove the claim, we pick a standard set $\big\{\alpha_1,\cdots,\alpha_g,\beta_1,\cdots,\beta_g,\gamma_1,\cdots,\gamma_n\big\}$ of generators of $\pi_1(S_{g,n})$, with generating relation
$$
[\alpha_1,\beta_1]\cdots[\alpha_g,\beta_g]\gamma_1\cdots\gamma_n=\id,
$$
and consider the map
\begin{align}
\Hom_0(\pi_1(S_{g,n}),\SL(3,\R)\ltimes\R^3)&\to \Hom_0(\pi_1(S_{g,n}),\SL(3,\R))\times(\R^3)^{2g+n}\label{eqn_map}\\
\rho_\tau&\mapsto\big(\rho, \tau(\alpha_1),\cdots,\tau(\gamma_n)\big),\nonumber
\end{align}
which assigns to each $ \rho_\tau$ its projection $\rho$ and the values of the admissible cocycle $\tau$ at the generators. Viewing the target of the map \eqref{eqn_map} as the trivial vector bundle of rank $6g+3n$ over $\Hom_0(\pi_1(S_{g,n}),\SL(3,\R))$, we shall prove the first part of the claim by showing that the map identifies the source with a sub-bundle of rank $6g-3+2n$ in the target.

To this end, fix a representation $\rho\in\Hom_0(\pi_1(S_{g,n}),\SL(3,\R))$ and let $C\subset\R^3$ be the cone divisible by the image of $\rho$. Any admissible cocycle  $\tau\in Z^1_\rho(\pi_1(S_{g,n}),\R^3)$ is completely determined by its values at the generators because of the cocycle condition
$$
\tau(\alpha\beta)=\tau(\alpha)+\rho(\alpha)\tau(\beta);
$$ 
whereas the only constraints on these values are 
\begin{itemize}
\item[-]
$\tau(\gamma_j)$ belongs to the $C$-null subspace $V_j(\rho)\subset\R^3$ preserved by $\rho(\gamma_j)$ for $j=1,\cdots,n$;
\item[-]
if we use the cocycle condition to expand
$\tau([\alpha_1,\beta_1]\cdots[\alpha_g,\beta_g]\gamma_1\cdots\gamma_n)$ into an expression only involving the values of $\rho$ and $\tau$ on the generators, then the expression gives $0$. 
\end{itemize}
By a calculation, we find the expansion to be
\begin{align*}
&\tau([\alpha_1,\beta_1]\cdots[\alpha_g,\beta_g]\gamma_1\cdots\gamma_n)=\big(I-\rho(\alpha_1)\rho(\beta_1)\rho(\alpha_1)^{-1}\big)\tau(\alpha_1)\\
&\quad+\rho(\alpha_1)\big(I-\rho(\beta_1)\rho(\alpha_1)^{-1}\rho(\beta_1)^{-1}\big)\tau(\beta_1)+\bullet\tau(\alpha_2)+\cdots+\bullet\tau(\gamma_n),
\end{align*}
where the bullet in front of each $\tau(\alpha_i)$, $\tau(\beta_i)$ with $i\geq2$ and each $\tau(\gamma_j)$ is a specific linear transformation of $\R^3$ given by the $\rho(\alpha_i)$'s, $\rho(\beta_i)$'s and $\rho(\gamma_j)$'s. 
So the two constraints together require  $(\tau(\alpha_1),\cdots,\tau(\gamma_n))$ to be in the kernel of the linear map
\begin{align*}
\mathcal{L}_\rho:(\R^3)^{2g}\oplus V_1(\rho)\oplus\cdots\oplus V_n(\rho)&\to \R^3\\
\mathcal{L}_\rho(X_1,\cdots X_g, Y_1,\cdots, Y_g, Z_1,\cdots, Z_n)&:=\big(I-\rho(\alpha_1)\rho(\beta_1)\rho(\alpha_1)^{-1}\big)X_1\\
&\hspace{-1cm}+\rho(\alpha_1)\big(I-\rho(\beta_1)\rho(\alpha_1)^{-1}\rho(\beta_1)^{-1}\big)Y_1+\bullet X_2+\cdots+\bullet Z_n.
\end{align*} 

In order to show that $\mathcal{L}_\rho$ is surjective, we define the \emph{axis} of any hyperbolic element $A$ in $\Aut(C)$ to be the projective line $\Axis(A)\subset\RP^2$ which is the projectivization of the $2$-subspace of $\R^3$ spanned by the eigenlines of  the largest and smallest eigenvalues of $A$. In particular, if $C$ is the future light cone $C_0$, then $\P(C_0)$ is the Klein model of the hyperbolic plane and the geodesic $\Axis(A)\cap \P(C_0)$ is the axis of $A$ in the sense of hyperbolic geometry. In general, if the middle eigenvalue of $A$ is $1$, then the $2$-subspace projecting to $\Axis(A)$ is exactly the image of the linear map $I-A$, hence
$$
\Axis(A)=\P\big\{(I-A)X\ \big|\ X\in\R^3\big\}.
$$

By Thm.\@ \ref{thm_marquis} and Prop.\@ \ref{prop_qd} \ref{item_qd2}, $\rho(\alpha_1)$ and $\rho(\beta_1)$ are hyperbolic elements with different axes. So we can show the surjectivity of  $\mathcal{L}_\rho$ in the following cases separately:
\begin{itemize}
	\item If the middle eigenvalue of $\rho(\beta_1)$ is not $1$, then we have
	$$
	\big\{\big(I-\rho(\alpha_1)\rho(\beta_1)\rho(\alpha_1)^{-1}\big)X_1\ \big|\  X_1\in\R^3\big\}=\R^3,
	$$
	hence $\mathcal{L}_\rho$ is surjective.
	\item If the middle eigenvalue of $\rho(\alpha_1)$ is not $1$, $\mathcal{L}_\rho$ is also surjective because
	$$
	\big\{\rho(\alpha_1)\big(I-\rho(\beta_1)\rho(\alpha_1)^{-1}\rho(\beta_1)^{-1}\big)Y_1\ \big|\  Y_1\in\R^3\big\}=\R^3.
	$$
	\item If both $\rho(\alpha_1)$ and $\rho(\beta_1)$ have middle eigenvalue $1$, the two projective lines
	\begin{align*}
	&\P\big\{\big(I-\rho(\alpha_1)\rho(\beta_1)\rho(\alpha_1)^{-1}\big)X_1\ \big|\  X_1\in\R^3\big\}=\rho(\alpha_1)\Axis(\rho(\beta_1)),\\
	&\P\big\{\rho(\alpha_1)\big(I-\rho(\beta_1)\rho(\alpha_1)^{-1}\rho(\beta_1)^{-1}\big)Y_1\ \big|\  Y_1\in\R^3\big\}
	=\rho(\alpha_1)\rho(\beta_1)\Axis(\rho(\alpha_1))
	\end{align*}
  are different because $\Axis(\rho(\alpha_1))\neq\Axis(\rho(\beta_1))$. As a result, we have
  $$
  \big\{\big(I-\rho(\alpha_1)\rho(\beta_1)\rho(\alpha_1)^{-1}\big)X_1+\rho(\alpha_1)\big(I-\rho(\beta_1)\rho(\alpha_1)^{-1}\rho(\beta_1)^{-1}\big)Y_1\ \big|\ X_1,Y_1\in\R^3\big\}=\R^3.
  $$
  Therefore, $\mathcal{L}_\rho$ is surjective in this case as well.
\end{itemize}

The surjectivity of $\mathcal{L}_\rho$ and the obvious fact that $V_j(\rho)$ and $\mathcal{L}_\rho$ depend continuously on $\rho$ imply that the image of the map \eqref{eqn_map}, namely the set
$$
\mathcal{V}:=\big\{(\rho,\ker\mathcal{L}_\rho)\ \big|\ \rho\in\Hom_0(\pi_1(S_{g,n}),\SL(3,\R))\big\},
$$
is the total space of a vector bundle over $\Hom_0(\pi_1(S_{g,n}),\SL(3,\R))$ of rank $6g-3+2n$, whose fiber at $\rho$ is $\ker\mathcal{L}_\rho$. 
Since the map is a homeomorphism from $\Hom_0(\pi_1(S_{g,n}),\SL(3,\R)\ltimes\R^3)$ to $\mathcal{V}$, this implies the first part of the claim.

For the second part of the claim, note that by the discussion of coboundaries in the previous subsection, the action of $X\in\R^3$ on $\Hom_0(\pi_1(S_{g,n}),\SL(3,\R)\ltimes\R^3)$ sends $\rho_\tau$ to $\rho_{\tau+\tau_{X}}$, where $\tau_{X}(\gamma):=(I-\rho(\gamma))X$ ($\forall \gamma\in\pi_1(S_{g,n})$). Therefore, if we identify $\Hom_0(\pi_1(S_{g,n}),\SL(3,\R)\ltimes\R^3)$ with $\mathcal{V}$ via the map \eqref{eqn_map} as above, then its quotient by $\R^3$ is given by quotienting every fiber $\ker\mathcal{L}_\rho$ of $\mathcal{V}$ by the image of the linear map
\begin{align*}
\R^3&\to\ker\mathcal{L}_\rho\\
X&\mapsto \big((I-\rho(\alpha_1))X,\cdots, (I-\rho(\alpha_g))X,\\
&\hspace{0.6cm}(I-\rho(\beta_1)X,\cdots, (I-\rho(\beta_g))X,\,(I-\rho(\gamma_1))X,\cdots, (I-\rho(\gamma_n))X\big). 
\end{align*}
With an argument similar to the above proof for surjectivity of $\mathcal{L}_\rho$, one can show that this map is injective by looking at the components $(I-\rho(\alpha_1))X$ and $(I-\rho(\beta_1))X$ and considering the three cases as above according to middle eigenvalues of $\rho(\alpha_1)$ and $\rho(\beta_1)$. The map also depends continuously on $\rho$, hence the images for all $\rho$ together form a vector 
sub-bundle of $\mathcal{V}$ of rank $3$. This completes the proof of the claim.

The claim implies that the first quotient $\mathcal{E}:=\Hom_0(\pi_1(S_{g,n}),\SL(3,\R)\ltimes\R^3)\big/\R^3$ in \eqref{eqn_quotient} is the total space of a vector bundle of rank $6g-6+2n$ over the manifold $M:=\Hom_0(\pi_1(S_{g,n}),\SL(3,\R))$. One readily checks that the action of $G:=\SL(3,\R)$ on $M$ lifts to an action on $\mathcal{E}$  which sends fibers to fibers by linear isomorphisms. Therefore, by the two lemmas given in the next subsection, $\widehat{\mathcal{P}}_{g,n}=\mathcal{E}/G$ is a vector bundle of rank $6g-6+2n$ over $\mathcal{P}_{g,n}=M/G$, as required.
\end{proof}

\begin{remark}\label{remark_torus}
The above argument also works in the case $(g,n)=(1,0)$ and shows that if $C\subset\R^3$ is a proper convex cone quasi-divisible by a group $\Gamma$ isomorphic to $\mathbb{Z}^2$, then we have $Z^1(\Gamma,\R^3)=B^1(\Gamma,\R^3)$, which means that every affine deformation of $\Gamma$ is conjugate to $\Gamma$ itself by a translation. In this case, it is well known that $C$ is a triangular cone.
\end{remark}

\subsection{Appendix: two technical lemmas}
%
\begin{lemma}\label{lemma_proper}
The conjugation action of $\SL(3,\R)$ on $\Hom_0(\pi_1(S_{g,n}),\SL(3,\R))$ is free and proper.
\end{lemma}
\begin{proof}
A continuous action of a topological group $G$ on a Hausdorff space $M$ is proper if and only if for any compact set $K\subset M$, $\{g\in G\mid gK\cap K\neq\emptyset\}$ is relatively compact in $G$. With this criterion, it is easy to check that if $M$ and $N$ are Hausdorff $G$-spaces such that the action on $N$ is free and proper, and there is an equivariant continuous map $M\to N$, then the action on $M$ is free and proper as well. 

Now set $G=\SL(3,\R)$ and pick a standard set of generators $\alpha_1$, $\cdots$, $ \alpha_g$, $\beta_1$, $\cdots$, $\beta_g$, $\gamma_1$, $\cdots$, $\gamma_n$ of $\pi_1(S_{g,n})$. Given $\rho\in\Hom_0(\pi_1(S_{g,n}),G)$, we let $C$ be the corresponding quasi-divisible convex cone and consider $\rho(\alpha_1)$ and $\rho(\beta_1)$, which are hyperbolic elements in $\Aut(C)$ without common fixed points by Thm.\@ \ref{thm_marquis} and Prop.\@ \ref{prop_qd} \ref{item_qd2}. For any hyperbolic $A\in G$, let $p_1(A)$, $p_2(A)$ and $p_3(A)$ denote the fixed points of $A$ in $\RP^2$ corresponding to the largest, middle and smallest eigenvalues of $A$, respectively.
Note that if $A\in\Aut(C)$, then two of the three points, namely $p_1(A)$ and $p_3(A)$, are on $\pa\P(C)$, and the projective lines $\overline{p_1(A)p_2(A)}$ and $\overline{p_2(A)p_3(A)}$ are tangent to $\pa\P(C)$ at the two points, respectively (see Figure \ref{figure_parabolic}). Therefore, the strict convexity of $\pa\P(C)$ (see Prop.\@ \ref{prop_qd} \ref{item_qd1}) implies that the six points $p_i(\rho(\alpha_1))$, $p_i(\rho(\beta_1))$ ($i=1,2,3$) are in general position (\ie there is no line passing through any three of them).

As a consequence, the assignment $\rho\mapsto(\rho(\alpha_1),\rho(\beta_1))$ gives an equivariant map from $\Hom_0(\pi_1(S_{g,n}),G)$ to the $G$-space
$$
X:=\left\{(A,B)\in G\times G\ \Big|\ \parbox[l]{7.5cm}{$A$ and $B$ are hyperbolic and the six points $p_i(A)$, $p_i(B)$ ($i=1,2,3$) in $\RP^2$ are in general position}\right\}
$$
(with the $G$-action by conjugation). In light of the fact mentioned at the beginning, we now only need to show that the $G$-action on $X$ is free and proper.

The freeness is easy to check and we omit the details. To show the properness, suppose by contradiction that the $G$-action on $X$ is not proper, then there exist a convergent sequence $(A_n,B_n)$ in $X$ and an unbounded sequence $(C_n)$ in $G$ such that $(A_n',B_n'):=(C_nA_nC_n^{-1}, C_nB_nC_n^{-1})$ also converges in $X$. 
Using the Cartan decomposition of $\SL(3,\R)$, we write $C_n=P_n T_nQ_n$ for each $n$, where $P_n,Q_n\in\SO(3)$ and $(T_n)$ is an unbounded sequence in the space
$$
\left\{
\begin{pmatrix}
\lambda_1&&\\&\lambda_2&\\&&\lambda_3
\end{pmatrix}
\ \Bigg|\ \lambda_1\geq\lambda_2\geq\lambda_3>0, \ \lambda_1\lambda_2\lambda_3=1
\right\}\subset \SL(3,\R).
$$
By restricting to a subsequence, we may assume that both sequences $(P_n)$ and $(Q_n)$ converge in $\SO(3)$. 

Suppose $e_1,e_2,e_3\in\RP^2$ correspond to the coordinate axes in $\R^3$. We claim that for any sequence $(q_n)$ in $\RP^2$ converging to a point not on the projective line $\overline{e_2e_3}$, all the limit points of the sequence $(T_nq_n)$ must be on the line $\overline{e_1e_2}$. To show this, we may assume by contradiction that $q_n\to q_\infty\notin\overline{e_2e_3}$ and $T_nq_n\to q_\infty'\notin \overline{e_1e_2}$ as $n\to\infty$. Write $T_n=\mathsf{diag}(T_n^{(1)},T_n^{(2)},T_n^{(3)})$ and $q_n=[q_n^{(1)}:q_n^{(2)}:q_n^{(3)}]$. The condition $q_\infty'\notin \overline{e_1e_2}$ means $q_\infty'=[x_1:x_2:1]$ for some $x_1,x_2\in\R$, hence the convergence $T_nq_n\to q_\infty'$ implies that $q_n^{(3)}\neq0$ for $n$ large enough and that
$$
\frac{T_n^{(1)}}{T_n^{(3)}}\frac{q_n^{(1)}}{q_n^{(3)}}\to x_1, \  \ \frac{T_n^{(2)}}{T_n^{(3)}}\frac{q_n^{(2)}}{q_n^{(3)}}\to x_2,
$$
Since $(T_n)$ is unbounded and $T_n^{(1)}\geq T_n^{(2)}\geq T_n^{(3)}$ by assumption, we have $T_n^{(1)}/T_n^{(3)}\to+\infty$ and $T_n^{(2)}/T_n^{(3)}\geq1$. It follows that $q_n^{(1)}/q_n^{(3)}\to0$ and $q_n^{(2)}/q_n^{(3)}$ is bounded. As a result, we have $q_\infty\in\overline{e_2e_3}$, a contradiction. This proves the claim.

Now suppose $A_n'\to A_\infty'$ as $n\to\infty$. Then we have
$$
p_i(A'_n)=P_nT_nQ_np_i(A_n)\to p_i(A_\infty')
$$
for $i=1,2,3$. Also, assuming $A_n\to A_\infty$, we have $Q_np_i(A_n)\to Q_\infty p_i(A_\infty)$. Therefore, the above claim results in the implication
\begin{equation}\label{eqn_proper1}
p_i(A_\infty)\notin Q_\infty^{-1}\overline{e_2e_3}\ \Longrightarrow\  p_i(A'_\infty)\in P_\infty\overline{e_1e_2}~. 
\end{equation}
The same argument applies to $(B_n')$ and yields the implication
\begin{equation}\label{eqn_proper2}
p_i(B_\infty)\notin Q_\infty^{-1}\overline{e_2e_3}\ \Longrightarrow\  p_i(B'_\infty)\in P_\infty\overline{e_1e_2}~. 
\end{equation}
Since $(A_\infty,B_\infty)\in X$, at most two of the six points $p_i(A_\infty)$, $p_i(B_\infty)$ ($i=1,2,3$) can be on the line $Q_\infty^{-1}\overline{e_2e_3}$. This means that at least four of the six points satisfy the condition on the left-hand side of \eqref{eqn_proper1} or \eqref{eqn_proper2}. It follows that as least four of the six points $p_i(A'_\infty)$, $p_i(B'_\infty)$ ($i=1,2,3$) are on the line $P_\infty\overline{e_1e_2}$, contradicting the fact that  $(A'_\infty,B'_\infty)\in X$. This completes the proof.
\end{proof}

\begin{lemma}\label{lemma_bundle}
Let $G$ be a locally compact topological group, $M$ be a Hausdorff $G$-space and $E\to M$ be a topological vector bundle of rank $r$ such that the $G$-action on $M$ is free and proper, and lifts to an action on $E$ which sends fibers to fibers by linear isomorphisms. Then $E/G\to M/G$ is also a vector bundle of rank $r$.
\end{lemma}
\begin{proof}
H.\@ Cartan's axiom (FP) for principal bundles \cite{Cartan} is equivalent to the statement that $M\to M/G$ is a principal $G$-bundle (where $G$ is a locally compact group and $M$ a Hausdorff $G$-space) if and only if the action is free and every point of $M$ has a neighborhood $U$ such that $\{g\in G\mid g U\cap U\neq\emptyset\}$ is relatively compact in $G$ (this condition is weaker than properness). See also \cite{Palais}.

Therefore, under the assumption of the lemma, we can pick an open cover $\mathcal{U}$ of $M/G$ such that the preimage $\pi^{-1}(U)\subset M$ of each $U\in\mathcal{U}$ identifies with the product $U\times G$, on which $G$ acts by multiplication on the $G$-factor.

On the other hand, there is an open cover $\mathcal{W}$ of $M$ such that for each $W\in \mathcal{W}$, we have a bundle chart $f_W:E|_{W}\overset\sim\to W\times\R^r$. After refining the open cover $\mathcal{U}$ if necessary, we may find $g_U\in G$ for each $U\in\mathcal{U}$, such that the slice  $U\times\{g_U\}$ in $U\times G\cong\pi^{-1}(U)$ is contained in some $W\in\mathcal{W}$ (see Figure \ref{figure_bundle}). 
\begin{figure}[h]
	\includegraphics[width=9.5cm]{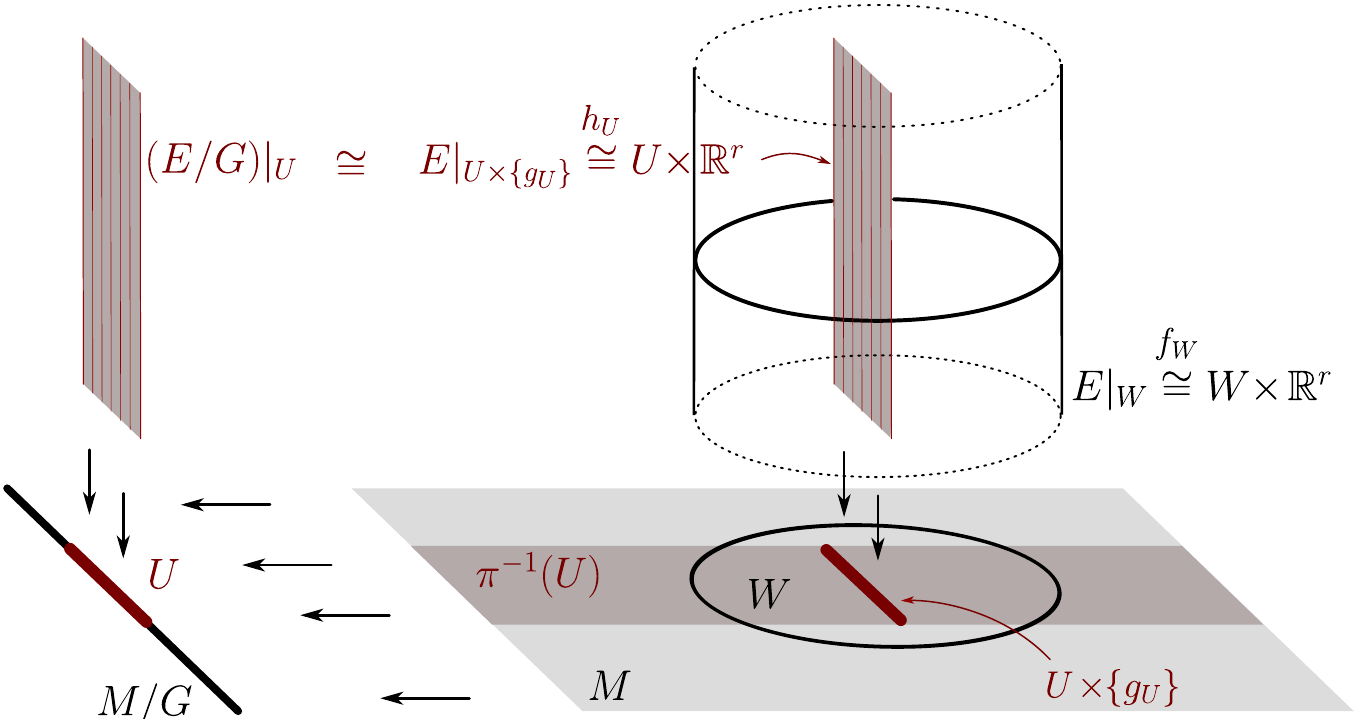}
	\caption{Schematic picture of the proof of Lemma \ref{lemma_bundle}.}
	\label{figure_bundle}
\end{figure}
By restricting the bundle chart $f_W$ to this slice, we get an identification $h_U:E|_{U\times\{g_U\}}\overset\sim\to U\times\R^r$. Meanwhile,  $E|_{U\times\{g_U\}}$ can be identified with $(E/G)|_{U}$ (the preimage of $U$ by the map $E/G\to M/G$) because every $G$-orbit in $E|_{\pi^{-1}(U)}$ passes through $E|_{U\times\{g_U\}}$ exactly once. It is routine to check that the family of maps $h_U:(E/G)|_U\cong E|_{U\times\{g_U\}}\overset\sim\to U\times\R^r$, $U\in\mathcal{U}$ form a bundle atlas for $E/G$, which completes the proof.
\end{proof}

\section{Proof of main results}\label{sec:last proof}
In this section, we first give a proof of Parts \ref{item_d1} and \ref{item_d1.5} of Theorem \ref{thm_intro3} using the framework set up in the previous sections and a lemma proved in \S\ref{subsec_continuous} below, then we deduce Theorem \ref{thm_intro3} \ref{item_d2} from our earlier work \cite{nie-seppi}, and explain how Theorems \ref{thm_intro1}, \ref{thm_intro2} and Corollary \ref{coro_intro} follow from Theorem \ref{thm_intro3}.

\subsection{Continuous boundary function with $\Lambda'$-invariant graph}\label{subsec_continuous}
We first show that Condition \ref{item_d11} in Theorem \ref{thm_intro3} \ref{item_d1} implies the existence of certain smooth functions on $\Omega$ with $\Lambda'$-invariant graph:
\begin{lemma}\label{lemma_smoothinvariant}
	Let $\Omega\subset\R^2$ be a bounded convex domain quasi-divisible by a torsion-free group $\Lambda<\Aut(\Omega)$. Suppose $\Lambda'<\Aut(\Omega\times\R)$ projects to $\Lambda$ bijectively, and every element with parabolic projection has a fixed point in $\pa\Omega\times\R$. For each puncture of the surface $S:=\Omega/\Lambda$, we  take a neighborhood homeomorphic to a punctured disk, assume these neighborhoods have disjoint closures, and let $U$ be their union. Then there exists a function $v\in\C^\infty(\Omega)$ with $\Lambda'$-invariant graph, such that the restriction of $v$ to each connected component of the lift $\widetilde{U}\subset\Omega$ of $U$ is an affine function.
\end{lemma}
If $\Omega$ is divisible by $\Lambda$, then $S$ is a closed surface and the lemma just asserts the existence of a smooth function with $\Lambda'$-invariant graph in this case.
\begin{proof}
	Since the connected components of $U$ have disjoint closures, we can enlarge $U$ to a bigger open set $U_0$ containing the closure $\overline{U}$, such that $U_0$ is still a disjoint union of neighborhoods of the punctures homeomorphic to a punctured disk. We then take simply connected open sets $U_1,\cdots, U_N$ in $S$ disjoint from $\overline{U}$ to form an open cover of $S$ together with $U_0$.
	
	Let $(\iota_i)$ be a $\C^\infty$-partition of unity subordinate to this open cover. Namely, each $\iota_i$ is a $\C^\infty$-function on $S$ taking values in $[0,1]$, with support contained in $U_i$, such that $\sum_{i=0}^N\iota_i=1$ on $S$. Note that $\iota_0=1$ on $\overline{U}$ because $\overline{U}$ is disjoint from any $U_i$ with $i\geq1$.
	
	Let $\widetilde{U}_i$ denote the lift of $U_i$, \ie the pre-image of $U_i$ by the covering map $\pi:\Omega\to S$, and $\widetilde{\iota}_i:=\iota_i\circ\pi\in\C^\infty(\Omega)$ denote the lift of $\iota_i$ to $\Omega$.
	In order to construct the required function $v$, we shall construct $v_i\in\C^\infty(\widetilde{U}_i)$ with $\Lambda'$-invariant graph for each $i=0,\cdots, N$ and sum up them using the partition of unity $(\widetilde{\iota}_i)$.  To this end, we treat $i=0$ and $i\geq1$ separately.
	
	Given $i\ge1$, since $U_i$ is simply connected, we can write $\widetilde{U}_i$ as a disjoint union
	$$
	\widetilde{U}_i=\bigcup_{A\in\Lambda}A(W)
	$$ 
	where $W$ is a connected component of $\widetilde{U}_i$. We can then take an arbitrary $w\in\C^\infty(W)$ and obtain the required $v_i\in \C^\infty(\widetilde{U}_i)$ using the $\Lambda'$-action on the graph of $w$.
	More precisely, $v_i$ is given by
	$$
	\gra{v_i}=\bigcup_{\Phi\in\Lambda'}\Phi\big(\gra{w}\big).
	$$
	For instance, if $w\equiv0$, then $u_i$ is an affine function on each component on $\widetilde{U}_i$. 
	
	For $i=0$, we may assume $U_0=\bigcup_{j=1}^nV_j$, where we label the punctures of $S$ by $1,\cdots, n$, and $V_j$ is a neighborhood of the $j^\text{th}$ puncture. To construct $v_0$, we only need to construct $v_0^{(j)}\in\C^\infty(\widetilde{V}_j)$ with $\Lambda'$-invariant graph on $\widetilde{V}_j:=\pi^{-1}(V_j)$ for each $j$ and put $v_0:=\sum_{j=1}^nv_0^{(j)}$. 
	So we fix $j$ and a connected component $Z$ of $\widetilde{V}_j$. The subgroup of $\Lambda$ preserving $Z$ is generated by some parabolic element $A\in\Lambda$. Since the element $\Phi_A$ of $\Lambda'$ projecting to $A$ has a fixed point by assumption, $\Phi_A$ preserves some non-vertical affine plane $P$ by Cor.\@ \ref{coro_fixedaffinefunction}. Therefore, we can define $v_0^{(j)}$ by first letting $v_0^{(j)}\big|_Z$ be the affine function whose graph is $P$, then using the $\Lambda'$-action to define $v_0^{(j)}$ on the other components of $\widetilde{V}_j$. Namely, $v_0^{(j)}$ is given by
	$$
	\gra{v_0^{(j)}}=\bigcup_{\Phi\in\Lambda'/\langle \Phi_A\rangle}\Phi\big(P\cap(Z\times\R)\big).
	$$
	This finishes the construction of the $v_i$'s.
	
	We can now construct the required $v$ as
	$$
	v:=\sum_{i=0}^N\widetilde{\iota}_i\,v_i\in\C^\infty(\Omega).
	$$
	In order to check that the graph $\gra{v}$ is preserved by any $\Phi\in\Lambda'$, we pick $x\in\Omega$ and let $x'\in\Omega$ be its image by $\pi(\Phi)\in\Aut(\Omega)$.  Since $\iota_i$ is the lift of a function on $S$ and $\gra{v_i}$ is preserved by $\Lambda'$, we have $\widetilde{\iota}_i(x)=\widetilde{\iota}_i(x')$ and $\Phi(x,v_i(x))=(x',v_i(x'))$. Therefore, by Lemma \ref{lemma_covariance} \ref{item_covariance1} (whose statement is only about two points but can be generalized to $N$ points by applying repeatedly), we have
	$$
	\Phi(x,v(x))=\Phi\Big(x,\,\sum_i\widetilde{\iota}_i(x)v_i(x)\Big)=\Big(x',\,\sum_i\widetilde{\iota}_i(x')v_i(x')\Big)=(x',v(x')).
	$$
	This shows that $\gra{v}$ is preserved by $\Phi$. 
	Finally, since $\widetilde{\iota}_0=1$ on $\widetilde{U}\subset\widetilde{U}_0$, we have $v=v_0$ in $\widetilde{U}$, which restricts to an affine function on each component of $\widetilde{U}$ by construction. Therefore, $v$ satisfies the requirements and the proof is complete.
\end{proof}

This lemma enables us to show the implication ``\ref{item_d11}$\Rightarrow$\ref{item_d12}'' in Thm.\@ \ref{thm_intro3} \ref{item_d1}, namely:
\begin{proposition}\label{prop_ab}
Let $\Omega$, $\Lambda$ and $\Lambda'$ be as in Lemma \ref{lemma_smoothinvariant}. Then there exists $\phi\in\C^0(\pa\Omega)$ with graph preserved by $\Lambda'$.
\end{proposition}
\begin{proof}
Let $U$ and $\widetilde{U}$ be as in Lemma \ref{lemma_smoothinvariant}, $K\subset\Omega\setminus\widetilde{U}$ be a compact set with $\bigcup_{A\in\Lambda}A(K)=\Omega\setminus\widetilde{U}$, and  $v\in\mathsf{C}^\infty(\Omega)$ be a function produced by the lemma. 
Since the Cheng-Yau support function $w_\Omega$ from Thm.\@ \ref{thm_chengyau} is smooth and strongly convex (\ie has positive definite Hessian) in $\Omega$, by compactness of $K$, we can pick a sufficiently large constant $M>0$ such that the smooth functions
$$
u_-:=v+M w_\Omega,\quad u_+:=v-Mw_\Omega
$$ 
are strongly convex and strongly concave in $K$, respectively. By Lemma \ref{lemma_covariance} \ref{item_covariance2}, the graph of $u_\pm$ is preserved by $\Lambda'$, hence $u_\pm$ is strongly concave/convex on $\Omega\setminus\widetilde{U}$. Since $v$ restricts to an affine functions on each component of $\widetilde{U}$ by construction, it follows that $u_\pm$ is strongly concave/convex throughout $\Omega$.
Moreover, the boundary values of $u_+$ and $u_-$ (in the sense of \S \ref{subsec_convexfunction}) are the same function $\phi$ on $\pa\Omega$ because $w_\Omega$ is continuous on $\overline{\Omega}$ with vanishing boundary value. Since the boundary value of a convex (resp. concave) function is lower (resp. upper) semicontinuous by construction (see \S \ref{subsec_convexfunction}), we conclude that $\phi$ is continuous with $\Lambda'$-invariant graph, as required. 
\end{proof}

\subsection{Lower semicontinuous boundary functions with $\Lambda$-invariant graph}\label{subsec_lower} 
Recall from \S \ref{subsec_subgroup} that the group $\Aut(\Omega)$ of orientation-preserving projective automorphisms of $\Omega$ can be identified with the group of those automorphisms of the convex tube domain $\Omega\times\R$ which preserve the slice $\Omega\times\{0\}$. As an ingredient in the proof of Thm.\@ \ref{thm_intro3} \ref{item_d1.5}, we now identify all the lower semicontinuous functions on $\pa\Omega$ with graph preserved by a subgroup $\Lambda<\Aut(\Omega)$ which quasi-divides $\Omega$:
\begin{proposition}\label{prop_psimu}
Let $\Omega\subset\R^2$ be a bounded convex domain quasi-divisible by a torsion-free group $\Lambda<\Aut(\Omega)$, $\mathcal{F}\subset\pa\Omega$ be the set of fixed points of parabolic elements in $\Lambda$, and pick $p_1,\cdots,p_n\in\mathcal{F}$ such that $\mathcal{F}$ is the disjoint union of the orbits $\Lambda.p_j$, $j=1,\cdots,n$. Given $\mu=(\mu_1,\cdots,\mu_n)\in\R_{\geq0}^n$, let $\psi_\mu$ be the function on $\pa\Omega$, with graph preserved by $\Lambda$ (as a group of automorphisms of $\Omega\times\R$, which also acts on $\pa\Omega\times\R$), such that
	$$\psi_\mu(p_j)=-\mu_j\ \text{ for all $j$;}\ \ \psi_\mu=0\ \text{ on $\pa\Omega\setminus\mathcal{F}$}.
$$ 
Then $\psi_\mu$ is lower semicontinuous. Moreover, these are the only lower semicontinuous functions $\pa\Omega\to\R\cup\{+\infty\}$ with graphs preserved by $\Lambda$ that are not constantly $+\infty$.
\end{proposition}
\begin{proof}
We claim that for any convex function $u:\Omega\to\R$ with $\Lambda$-invariant graph, the boundary value $u|_{\pa\Omega}$ (in the sense of \S \ref{subsec_convexfunction}) vanishes on $\pa\Omega\setminus\mathcal{F}$.

In order to show this, let us compare $u$ with the Cheng-Yau support function $w_\Omega$, which is strictly negative on $\Omega$ and has $\Lambda$-invariant graph.  Letting $\widetilde{U}\subset\Omega$ be the open set lifted from some neighborhoods of punctures of $\Omega/\Lambda$ as in Lemma \ref{lemma_smoothinvariant}, we can take a compact subset $K\subset \Omega$ such that $\bigcup_{A\in\Lambda}A(K)=\Omega\setminus\widetilde{U}$ and take a sufficiently large constant $M>0$ such that $-Mw_\Omega\geq u\geq Mw_\Omega$ in $K$. Then the $\Lambda$-invariance of the graphs of both $u$ and $w_\Omega$ imply the same inequalities on $\Omega\setminus\widetilde{U}$. It follows that $u|_{\pa \Omega}$ vanishes on $\pa\Omega\setminus\mathcal{F}$ because for any $x_0\in\pa\Omega\setminus\mathcal{F}$, $u(x_0)$ is the limit of $u(x)$ as $x$ tends to $x_0$ along a line segment joining $x_0$ with some $x_1\in\Omega$, and this segment contains a sequence of points in $\Omega\setminus\widetilde{U}$ tending to $x_0$. We have thus proven the claim.

To show that the specific function $\psi_\mu$ is lower semicontinuous, we first pick a generator $A_j$ for the stabilizer of $p_j$ in $\Lambda$ (which is an infinite cyclic group generated by a parabolic element), and apply Lemma \ref{lemma_parabolicfunction} to find an $A_j$-invariant smooth convex function $f_j:\Omega\to\R$ whose boundary value at $p_j$ is $-\mu_j$. Then we can construct a smooth function $v_\mu\in\C^\infty(\Omega)$ with $\Lambda$-invariant graph satisfying
$$
\lim_{s\to0}v_\mu((1-s)p_j+sx)=-\mu_j \ \text{ for any }x\in\Omega,\,j=1,\cdots,n
$$
using a partition of unity similarly as in the proof of Lemma  \ref{lemma_smoothinvariant}, only replacing the $A$-invariant affine plane $P$ used in that proof by the graph of $f_j$. Next, by taking a large enough $M>0$ similarly as in the proof of Prop.\@ \ref{prop_ab}, we obtain a convex function $u_\mu=v_\mu+Mw_\Omega\in\C^\infty(\Omega)$ with $\Lambda$-invariant graph, which has the same limit property as $v_\mu$ above because $w_\Omega$ is continuous on $\overline{\Omega}$ and has vanishing boundary value. This means that the boundary value $u_\mu|_{\pa\Omega}$ of $u_\mu$, which is a lower semicontinuous function on $\pa\Omega$ with $\Lambda$-invariant graph, coincides with $\psi_\mu$ at every $p_j$. The $\Lambda$-invariance then implies that $u_\mu|_{\pa\Omega}=\psi_\mu$ holds on $\mathcal{F}$. It follows that the equality actually holds on the whole boundary $\pa\Omega$, because $\psi_\mu$ vanishes on $\pa\Omega\setminus\mathcal{F}$ by definition and so does $u_\mu|_{\pa\Omega}$ by the above claim. As a consequence, $\psi_\mu$ is lower semicontinuous, as required.

To show the ``Moreover'' statement, let $\psi:\pa\Omega\to\R\cup\{+\infty\}$ be an arbitrary lower semicontinuous function with graph preserved by $\Lambda$, such that $\psi(x_0)<+\infty$ for some $x_0\in\pa\Omega$.
As explained in \S \ref{subsec_convexfunction}, $\psi$ is the boundary value of its convex envelope $\env{\psi}\in\LC(\R^2)$. The effective domain $\dom{\env{\psi}}:=\{x\in\R^2\mid \env{\psi}(x)<+\infty\}$ is a convex set containing the $\Lambda$-orbit of $x_0$, which is dense in $\pa\Omega$ by Prop.\@ \ref{prop_qd} \ref{item_qd3}, hence it contains $\Omega$. This means $\env{\psi}$ only takes finite values in $\Omega$, so we can invoke the above claim and conclude that $\psi=0$ on $\pa\Omega\setminus\mathcal{F}$. The density of $\pa\Omega\setminus\mathcal{F}$ and the lower semicontinuity of $\psi$ then imply  $\psi\leq0$ on $\mathcal{F}$. Therefore, $\psi$ must equal some $\psi_\mu$, as required.
\end{proof}

We can now prove the first two parts of Theorem \ref{thm_intro3}.
\begin{proof}[Proof of Thm.\@ \ref{thm_intro3} \ref{item_d1} and  \ref{item_d1.5}]
Suppose $\phi\in\C^0(\pa\Omega)$ has graph preserved by $\Lambda'$. Let $\widehat{\phi}:\pa\Omega\to\R\cup\{+\infty\}$ be a lower semicontinuous function also with graph preserved by $\Lambda'$ such that $\widehat{\phi}\not\equiv+\infty$. By Lemma \ref{lemma_covariance}, we
can write
$\widehat{\phi}=\phi+\psi$ for a lower semicontinuous $\psi:\pa\Omega\to\R\cup\{+\infty\}$ with graph preserved by $\Lambda$. By Prop.\@ \ref{prop_psimu}, $\psi$ must equal some $\psi_\mu$ described in that proposition. It follows that $\widehat{\phi}$ equals some $\phi_\mu$ described in the required statement \ref{item_d1.5}. Also, since every $\psi_\mu$ vanishes on the dense subset $\pa\Omega\setminus\mathcal{F}$ of $\pa\Omega$ by definition, the only continuous one among the $\psi_\mu$'s is the zero function $\psi_{0}$. It follows that the only continuous one among the $\phi_\mu$'s is $\phi$ itself. This proves Part \ref{item_d1.5}.

As for Part \ref{item_d1}, the implication ``\ref{item_d12}$\Rightarrow$\ref{item_d13}'' is trivial, and we have already shown ``\ref{item_d11}$\Rightarrow$\ref{item_d12}'' in Prop.\@ \ref{prop_ab}. To show the implication ``\ref{item_d13}$\Rightarrow$\ref{item_d11}'', let $\widehat{\phi}$ be as in \ref{item_d13}. We just proved that $\widehat{\phi}$ equals some $\phi_\mu$, hence in particular only takes values in $\R$. Then, for each $\Phi\in\Lambda'$ whose projection $\pi(\Phi)\in\Lambda$ is parabolic, if we let $p\in\pa\Omega$ be the fixed point of $\pi(\Phi)$, then $\Phi$ fixes $(p,\widehat{\phi}(p))\in\pa\Omega\times\R$ because of the $\Lambda'$-invariance of $\gra{\phi}$, hence \ref{item_d11} holds. This completes the proof of Part \ref{item_d1}.
\end{proof}

\subsection{Solving the Monge-Amp\`ere problem}\label{subsec_solving}
We shall deduce Part \ref{item_d2} of Theorem \ref{thm_intro3} from the following result in \cite{nie-seppi}. 
\begin{theorem}[{\cite[Thm.\@ A$'$ in \S 8.1]{nie-seppi}}]\label{thm_NS}
	Let $\Omega\subset\R^2$ be a bounded convex domain and $\phi:\pa\Omega\to\R$ be a lower semicontinuous function. Then
	\begin{enumerate}[label=(\arabic*)]
\item\label{item_thmmain1} For each $t\in\R$, there exists a unique convex solution $u_t\in\C^\infty(\Omega)$ to the Dirichlet problem of Monge-Amp\`ere equation
$$
	\begin{cases}
	\det\D^2 u=e^{-t}\,w_\Omega^{-4}\ \text{ in }\Omega,\\
	u|_{\pa\Omega}=\phi.
	\end{cases} 
$$
Here $w_\Omega$ is the Cheng-Yau support function of $\Omega$ (see Thm.\@ \ref{thm_chengyau}).
\item\label{item_thmmain2}
	For any $x_0\in\pa\Omega$, if $\Omega$ satisfies the exterior circle condition at $x_0$, then $u_t$ has infinite inner derivatives at $x_0$ (see \S \ref{subsec_legendre} for the definition).
		\item\label{item_thmmain3} For any fixed $x\in\Omega$, $t\mapsto u_t(x)$ is a strictly increasing concave function tending to $-\infty$ and $\env{\phi}(x)$ as $t$ tends to $-\infty$ and $+\infty$, respectively.
	\end{enumerate}
\end{theorem}
\begin{remark}
The original statement of \cite[Thm.\@ A$'$]{nie-seppi} is more general in that $\phi$ is only assumed to take values in $\R\cup\{+\infty\}$. It asserts the unique-existence of a lower semicontinuous $u:\overline{\Omega}\to\R\cup\{+\infty\}$ with $u|_{\pa\Omega}=\phi$, satisfying the same equation $\det\D^2 u=e^{-t}\,w_\Omega^{-4}$ in the convex domain $U:=\interior\dom{u}$ (ses \S \ref{subsec_convexfunction} for the notation), under the extra constraint that $u$ has infinite inner derivatives at every point of $\pa U\cap\Omega=\pa U\setminus\pa\Omega$.
\end{remark}

\begin{proof}[Proof of Thm.\@ \ref{thm_intro3} \ref{item_d2}]
The unique-existence of the required $u_t$ and the last bullet point in the statement of Thm.\@ \ref{thm_intro3} \ref{item_d2} are given immediately by  Thm.\@ \ref{thm_NS}.


We henceforth fix $t\in\R$ and denote $u:=u_t$ for simplicity. 
To see that the graph $\gra{u}\subset\Omega\times\R$ is preserved by any $\Phi\in\Lambda'$, we let $\widetilde{u}$ denote the convex function on $\Omega$ such that $\gra{\widetilde{u}}=\Phi(\gra{u})$, which has the same boundary value $\phi$ as $u$ because $\gra{\phi}\subset\pa\Omega\times\R$ is preserved by $\Lambda'$. By Prop.\@ \ref{prop_mongeampere}, the graph $\Sigma$ of the Legendre transform $u^*$ over $\D u(\Omega)$ is an affine $(C,e^\frac{3t}{2})$-surface (at this stage, we do not know whether the graph is entire, \ie whether $\D u(\Omega)=\R^2$), whereas it follows from Lemma \ref{lemma_actiongraph} that the graph $\widetilde{\Sigma}$ of $\widetilde{u}^*$ over $\D \widetilde{u}(\Omega)$ is the image of $\Sigma$ by the affine transformation in $\Aut(C)\ltimes\R^3$ corresponding to $\Phi$. Since the property of being an affine $(C,k)$-surface is preserved by $\Aut(C)\ltimes\R^3$, $\widetilde{\Sigma}$
is an affine $(C,e^\frac{3t}{2})$-surface as well, which implies, again by Prop.\@ \ref{prop_mongeampere}, that $\widetilde{u}$ is also a solution to the same Dirichlet problem. Thus, we have $u=\widetilde{u}$ by the uniqueness. This shows that $\gra{u}$ is preserved by $\Lambda'$.

Now it only remains to be shown that $\|\D u\|$ tends to $+\infty$ on $\pa\Omega$, which is equivalent, by Lemma \ref{lemma_gradient} \ref{item_gradient3}, to the condition that $u$ has infinite inner derivatives at every point of $\pa\Omega$.  

To this end, we first claim that $u$ is strictly smaller than the convex envelope $\overline{\phi}$ in $\Omega$ (we already have $u\leq\overline{\phi}$ by construction, see \S \ref{subsec_convexfunction}). Suppose by contradiction that $u(x_0)=\varphi(x_0)$ for some $x_0\in\Omega$ and let $a(x):=u(x_0)+(x-x_0)\cdot\D u(x_0)$ be the supporting affine function of $u$ at $x_0$. Then we have 
$a\leq u\leq \overline{\phi}$
in $\Omega$, with both equalities achieved at $x_0$. By \cite[Lemma 4.9]{nie-seppi}, the set $\big\{x\in\overline{\Omega}|\env{\phi}(x)=a(x)\big\}$ is the convex hull of some subset of $\pa\Omega$. Since this convex hull contains $x_0$, it also contains some line segment $I$ passing through $x_0$. It follows that $u=a$ on $I$, contradicting the strict convexity of $u$, hence the claim is proved.

If $\Omega$ is divisible by $\Lambda$, we take a compact set $K\subset \Omega$ with $\bigcup_{A\in\Lambda}A(K)=\Omega$. Since $\env{\phi}-u$ is strictly positive by the claim, we can pick
sufficiently small $\epsilon>0$ such that 
\begin{equation}\label{eqn_proof f2}
u\leq \env{\phi}+\epsilon w_\Omega
\end{equation}
in $K$. By Lemma \ref{lemma_covariance} \ref{item_covariance2}, the graphs of the functions on both sides are preserved by $\Lambda'$, so this inequality actually holds on the whole $\Omega$. Also, the equality is achieved on $\pa\Omega$. Therefore, for any $x_0\in\pa\Omega$, $x_1\in\Omega$ and $s\in(0,1]$, we have 
\begin{align}
&\frac{u(x_0+s(x_1-x_0))-u(x_0)}{s}\label{eqn_proof f2 2}\\
&\leq\frac{\env{\phi}(x_0+s(x_1-x_0))-\env{\phi}(x_0)}{s}+\epsilon\frac{w_\Omega(x_0+s(x_1-x_0))-w_\Omega(x_0)}{s}.\nonumber
\end{align}
Each of the three fractions in \eqref{eqn_proof f2 2} is increasing in $s$ by convexity, hence has a limit in $[-\infty,+\infty)$ as $s\to0$, and the limit of the last fraction is $-\infty$ because $w_\Omega$ has infinite inner derivatives at $x_0$ by Theorem \ref{thm_chengyau} and Lemma \ref{lemma_gradient} \ref{item_gradient3}. As a result, the left-hand side tends to $-\infty$, hence $u$ has infinite inner derivatives at $x_0$, as required.

If $\Omega$ is quasi-divisible but not divisible by $\Lambda$, we adapt the argument as follows. Let $U$ be an open subset of the punctured surface $\Omega/\Lambda$ as in Lemma \ref{lemma_smoothinvariant}, consisting of neighborhoods of punctures, $\widetilde{U}\subset\Omega$ be the lift of $U$, and $\mathcal{F}$ be the subset of $\pa\Omega$ consisting of fixed points of parabolic elements in $\Lambda$. Then the same reasoning as above shows that \eqref{eqn_proof f2} holds on $\Omega\setminus\widetilde{U}$. It follows that \eqref{eqn_proof f2 2} holds for those $s\in(0,1]$ such that $x_0+s(x_1-x_0)\in\Omega\setminus\widetilde{U}$. Now, if $x_0\in\pa\Omega$ is not in $\mathcal{F}$, then there exists a sequence of such $s$'s converging to $0$, and it follows that $u$ has infinite inner derivatives at $x_0$ in the same way as before. Otherwise, $x_0\in\mathcal{F}$ is the fixed point of some parabolic element in $\Lambda$, hence $\Omega$ satisfies the exterior circle condition at $x_0$ by Lemma \ref{lemma_benoisthulin}. In this case, we can apply Thm.\@ \ref{thm_NS} \ref{item_thmmain1} and conclude that $u$ has infinite inner derivatives at $x_0$ as well. This completes the proof.
\end{proof}

\subsection{Proofs of Thm.\@ \ref{thm_intro1}, \ref{thm_intro2} and Cor.\@ \ref{coro_intro}}\label{subsec_proofs}
To prove the results in the introduction about affine deformations, we 
now take a proper convex cone $C\subset\R^3$ quasi-divisible by a torsion-free group $\Gamma<\Aut(C)$. By Prop.\@ \ref{prop_qd} \ref{item_qd4}, the dual cone $C^*\subset \R^{3*}$ is quasi-divisible by the image $\Gamma^*:=\big\{\transp{\!A}^{-1} \mid A\in\Gamma\big\}$ of $\Gamma$ in $\Aut(C^*)$ (\cf \S \ref{subsec_omega}).

The dictionary between the two geometries explained in the introduction and \S \ref{sec:correspondence} then translates each affine deformation $\Gamma_\tau$ of $\Gamma$ to a group $\Gamma^*_\tau<\Aut(\Omega\times\R)$ of automorphisms of the convex tube domain $\Omega\times\R$ (where $\Omega\cong\P(C^*)$, see \S \ref{subsec_omega}) which projects bijectively to $\Gamma^*$ through the projection $\pi:\Aut(\Omega\times\R)\to\Aut(\Omega)=\Aut(C^*)$ (\cf \S \ref{subsec_auto}). Furthermore, by Prop.\@ \ref{prop_parabolic} and the identification $\hnull{C}\cong\pa\Omega\times\R$ (\cf Prop.\@ \ref{prop_hspace}), the cocycle $\tau$ is admissible if and only if every $\Phi\in\Gamma^*_\tau$, whose projection $\pi(\Phi)\in\Gamma^*$ is parabolic, has a fixed point in $\pa\Omega\times\R$.

The above discussion and the framework set up in the previous sections allow us to deduce most of the statements in Theorems \ref{thm_intro1}, \ref{thm_intro2} and Corollary \ref{coro_intro} of the introduction immediately from Theorem \ref{thm_intro3}. We proceed to give a formal proof. Let us reformulate the statements as follows:
\begin{theorem}[Thm.\@ \ref{thm_intro1}, \ref{thm_intro2} and Cor.\@ \ref{coro_intro}]\label{thm_main}
In the above setting, the followings hold.
\begin{enumerate}[label=(\arabic*)]
	\item\label{item_main1} 
	There exists a $C$-regular domain in $\A^3$ preserved by $\Gamma_\tau$ if and only if $\tau$ is admissible.
	In this case, there is a unique equivariant continuous map $f$ from $\pa\P(C)$ to the space of $C$-null planes in $\A^3$. The complement of $\bigcup_{x\in\pa\P(C)}f(x)$ in $\A^3$ has two connected components $D^+$ and $D^-$, which are $C$-regular and $(-C)$-regular domains preserved by $\Gamma_\tau$, respectively. 
	
	\item\label{item_main2} If $C$ is divisible by $\Gamma$, then $D^+$ is the unique $C$-regular domain preserved by $\Gamma_\tau$. Otherwise, assume the surface $S:=\P(C)/\Gamma$ has $n\geq1$ punctures and $\tau$ is admissible, then all the $C$-regular domains preserved by $\Gamma_\tau$ form a family $(D_\mu)$ parameterized by $\mu\in\R_{\geq0}^n$, such that $D_{(0,\cdots,0)}=D^+$ and we have $D_\mu\subset D_{\mu'}$ if and only if $\mu$ is coordinate-wise larger than or equal to $\mu'$. 
	
	\item\label{item_main3} For any $C$-regular domain $D$ preserved by $\Gamma_\tau$, the $\Gamma_\tau$-action on $D$ is free and properly discontinuous, and there exists a homeomorphism between the quotient $D/\Gamma_\tau$ and $S\times\R$ satisfying the following conditions:
		\begin{itemize}
		\item	
Let $K:D\to\R$ be the function given by composing the quotient map $D\to D/\Gamma_\tau$ with the projection $D/\Gamma_\tau\cong S\times\R\to \R$ to the $\R$-factor. Then $K$ is convex.
		\item
			For each $t\in\R$, the level surface $K^{-1}(t)$ of this function is a complete affine $(C,e^t)$-surface in $\A^3$ generating $D$, which is unique.
		\end{itemize} 
Moreover, the surface $K^{-1}(t)$ is asymptotic to the boundary of $D$ (\cf \S \ref{subsec_cregular}).
	\end{enumerate}
\end{theorem}
Parts \ref{item_main1} and \ref{item_main2} here are exactly the first two parts of Thm.\@ \ref{thm_intro1}, and it is easy to see that Part \ref{item_main3} contains Thm.\@ \ref{thm_intro1} \ref{item_intro13}, Thm.\@ \ref{thm_intro2} and Cor.\@ \ref{coro_intro} at the same time.
\begin{proof}
	\ref{item_main1}
	 We apply Thm.\@ \ref{thm_intro3} \ref{item_d1} to the convex tube domain $\Omega\times\R$ in the aforementioned setting, taking $\Lambda'$ and $\Lambda$ in the assumption of the theorem to be $\Gamma^*_\tau$ and $\Gamma^*$, respectively. As explained, Condition \ref{item_d11} in Thm.\@ \ref{thm_intro3} \ref{item_d1} is  equivalent to the admissibility of $\tau$, whereas Condition \ref{item_d13} is equivalent to the existence of a $\Gamma_\tau$-invariant $C$-regular domain by the correspondence between $C$-regular domains and lower semicontinuous functions on $\pa\Omega$ (see Thm.\@ \ref{thm_graph} \ref{item_thmgraph1}). Therefore, the required ``if and only if'' statement is a consequence of Thm.\@ \ref{thm_intro3} \ref{item_d1}. 
	
	 To show the statements about the map $f$, we first note that the domain $\pa\P(C)$ of the map can be replaced by $\pa\Omega=\pa\P(C^*)$. In fact, since $\pa\P(C^*)$ can be identified with the set of projective lines in $\RP^2$ tangent to $\pa\P(C)$, by the strict convexity and $\C^1$ property of $\pa \P(C)$ (see Prop.\@ \ref{prop_qd} \ref{item_qd1}), we have a equivariant homeomorphism $\pa\P(C)\cong\pa\P(C^*)$ given by assigning to each $x\in\pa\P(C)$ the line tangent to $\pa\P(C)$ at $x$.
	  
	 In view of the identification between the space $\hnull{C}$ of $C$-null planes and $\pa\Omega\times\R$ (see \S \ref{subsec_cspacelike}), one can check that a map $f:\pa\Omega\to\hnull{C}$ is equivariant if and only if it has the form
	$$
	f(x)=(f_1(x),\phi(x))\in\pa\Omega\times\R\cong\hnull{C},\ \forall x\in\pa\Omega.
	$$
	for some $\Gamma^*$-equivariant map $f_1:\pa\Omega\to\pa\Omega$ and some function $\phi:\pa\Omega\to\R$ whose graph is preserved by $\Gamma_\tau^*$. If $f$ is further assumed to be continuous, then $f_1$ can only be the identity because the fixed points of elements in $\Gamma^*$ form a dense subset of $\pa\Omega$ (see Prop.\@ \ref{prop_qd} \ref{item_qd1}), and $\phi$ can only be the one provided by Thm.\@ \ref{thm_intro3} \ref{item_d1}. This shows the unique-existence of $f$.
	
	Now let $D^+:=\sepi{\env{\phi}^*}$ be the $C$-regular domain corresponding to this $\phi$ (see Thm.\@ \ref{thm_graph} \ref{item_thmgraph1}), and $D^-\subset\A^3$ be the $(-C)$-regular domain obtained in a symmetric way.  
    In other words,	as explained in the last paragraph of \S \ref{subsec_cregular}, a point in $D^+$ (resp.\@ $D^-$) corresponds to an affine plane in $\A^{3*}$ which cross the convex tube domain $\Omega\times\R\subset\A^{3*}$ from below (resp.\@ above) of $\gra{\phi}$. So $D^+$ and $D^-$ are disjoint convex domains.
    Moreover, we have
	$$
	\A^3\setminus(D^+\cup D^-)=\bigcup_{x\in\pa\P(C)}f(x)
	$$
	because on one hand, a point in the set on the left-hand side corresponds to a non-vertical affine plane in $\A^{3*}$ which intersects $\gra{\phi}$; on the other hand, for each $x\in\pa\P(C)\cong\pa\Omega$, points in the $C$-null plane $f(x)$ correspond to non-vertical plans in $\A^{3*}$ passing through the point $(x,\phi(x))\in\gra{\phi}$. Therefore, $D^\pm$ are exactly the two connected components of $\A^3\setminus \bigcup_{x\in\pa\P(C)}f(x)$. Also, the $\Gamma_\tau^*$-invariance of $\gra{\phi}$ implies that $D^\pm$ are preserved by $\Gamma_\tau$. This completes the proof of Part \ref{item_main1}.
	
	\ref{item_main2} The $C$-regular domains in $\A^3$ preserved by $\Gamma_\tau$ are exactly the domains of the form $\sepi{\env{\widehat{\phi}}^*}$ for some $\widehat{\phi}$ described by Condition \ref{item_d13} in Thm.\@ \ref{thm_intro3} \ref{item_d1}. Therefore, by Thm.\@ \ref{thm_intro3} \ref{item_d1.5}, $D^+$ is the unique such domain in the divisible case, whereas in the quasi-divisible but not divisible case, 
	$$
	D_\mu:=\sepi{\env{\phi}_\mu^*}, \ \ \mu\in\R_{\geq0}^n
	$$
	are exactly all such domains. Moreover, the condition $D_\mu\subset D_{\mu'}$ is equivalent to $\env{\phi}_\mu^*\geq \env{\phi}_{\mu'}^*$, which is in turn equivalent to  $\env{\phi}_\mu\leq \env{\phi}_{\mu'}$ and hence $\phi_\mu\leq \phi_{\mu'}$ by basic properties of Legendre transforms and convex envelopes. By construction of $\phi_\mu$, the last inequality means $\mu$ is coordinate-wise larger than or equal to $\mu'$, as required.
	
	\ref{item_main3} Fix a $C$-regular domain $D=D_\mu$ as above, with $\mu\in\R_{\geq0}^n$, and let $u_t\in\C^\infty(\Omega)$ denote the unique convex solution to the Dirichlet problem 
$$
\begin{cases}
\det\D^2u=e^{-\frac{2t}{3}}w_\Omega^{-4}\\
u|_{\pa\Omega}=\phi_\mu
\end{cases}
$$
produced by Thm.\@ \ref{thm_intro3} \ref{item_d2} for each $t\in\R$ (although the equation here differs from the one in Thm.\@ \ref{thm_intro3} \ref{item_d2} in that the parameter $t$ is multiplied by $\frac{2}{3}$, the conclusions are clearly not affected). Then $u_t$ has $\Gamma_\tau^*$-invariant graph and the one-parameter family $(u_t)$ fulfills the first condition in Prop.\@ \ref{prop_foliation}. Moreover, $u_t$ satisfies the inequalities
\begin{equation}\label{eqn_barrier}
\env{\phi}_\mu+e^{-\frac{t}{3}}w_\Omega\leq u_t\leq \env{\phi}_\mu~.
\end{equation}
The second inequality is just because $u_t|_{\pa\Omega}=\phi_\mu$ (see \S \ref{subsec_convexfunction}), while the first one is given by \cite[Lemma 8.5]{nie-seppi} and follows easily from the Comparison Principle for Monge-Amp\`ere equations and basic properties of Monge-Amp\`ere measures. 

We shall deduce the required homeomorphism $S\times\R\overset\sim\to D/\Gamma_\tau$ from the map
\begin{align*}
F:T^-:=\big\{(x,\xi)\in\Omega\times\R\,\big|\,\xi<\overline{\phi}(x))\big\}&\longrightarrow D=\sepi{\overline{\phi}^*_\mu}\\
(x,u_t(x))&\mapsto\big(\D u_t(x)\,,\, x\cdot\D u_t(x)-u_t(x)\big)
\end{align*}
studied in \S \ref{subsec_foliation}, which is a homeomorphism by Prop.\@ \ref{prop_homeo}. 

While the group $\Gamma_\tau$ acts on the target $D$ of $F$ by affine transformation, it also acts on the domain $T^-$ by projective transformations via $\Gamma_\tau^*$ (\ie via the isomorphism $\Aut(C)\ltimes\R^3\cong\Aut(\Omega\times\R)$), and the geometric  definition of $F$ in \S \ref{subsec_foliation} implies that $F$ is equivariant with respect to the two actions.
But $T^-$ is foliated by the graphs $(\gra{u_t})$ (\cf Figure \ref{figure_map}), and the quotient $T^-/\Gamma_\tau^*$ is homeomorphic to $S\times\R=(\Omega/\Gamma^*)\times\R$ in such a way that each leaf $\gra{u_t}$ corresponds to the slice $S\times\{t\}$ (more precisely, this homeomorphism is given by the bijection $\Omega\times\R\overset\sim\to T^-$, $(x,t)\mapsto (x,u_t(x))$, which is shown to be a homeomorphism in the proof of Prop.\@ \ref{prop_homeo}). Therefore, both actions are free and properly discontinuous, and $F$ induces a homeomorphism between the quotients
$$
S\times\R\cong T^-/\Gamma_\tau^*\overset\sim\longrightarrow D/\Gamma_\tau,
$$
sending each slice $S\times\{t\}$ to the quotient of the surface $\gra{u_t^*}$.

This homeomorphism has the required properties in the two bullet points because $K$ is exactly the function studied in Prop.\@ \ref{prop_foliation} and proved to be convex therein, whereas $K^{-1}(t)=\gra{u_t^*}$ is the unique affine $(C,e^t)$-surface generating $D$ by Thm.\@ \ref{thm_graph}, Prop.\@ \ref{prop_mongeampere} and  Thm.\@ \ref{thm_intro3} \ref{item_d2}. 

Finally, since $w_\Omega$ is continuous on $\overline{\Omega}$ and vanishes on $\pa\Omega$, Inequality \eqref{eqn_barrier} implies that
$\env{\phi}_\mu(x)-u_t(x)\to 0$ as $x\in\Omega$ tends to $\pa\Omega$. By Thm.\@ \ref{thm_graph} \ref{item_thmgraph4}, this means $K^{-1}(t)$ is asymptotic to $\pa D$, as required. The proof is complete.

\end{proof}
\bibliographystyle{amsalpha} 
\bibliography{deformation}
\end{document}